%% file: TestSBM-HLi-2018.tex
\documentclass[aos,preprint]{imsart}  

\setattribute{journal}{name}{} 

\usepackage{color}
\usepackage{enumerate}
\usepackage{amsmath,amsthm,amsfonts}
\usepackage{graphicx}
\usepackage[export]{adjustbox}
\usepackage{multirow}
\usepackage{hyperref}
\usepackage{extarrows}
\usepackage{algorithm, algpseudocode}
\usepackage{soul}
\usepackage{MnSymbol}
\usepackage{natbib}
\newcommand{\ignore}[1]{}
\startlocaldefs

\numberwithin{table}{section} 

\theoremstyle{plain}
\newtheorem*{assum*}{Assumption}
\theoremstyle{plain}
\newtheorem{thm}{Theorem}
\newtheorem{lem}[thm]{Lemma}

\newtheorem{cor}[thm]{Corollary}
\newtheorem{defn}{Definition}[section]
\newtheorem{assum}{Assumption}


\newcounter{parentnumber}

\newtheoremstyle{TheoremNum}
{\topsep}{\topsep}              
{\itshape}                      
{}                              
{\bfseries}                     
{.}                             
{ }                             
{\thmname{#1}\thmnote{ \bfseries #3}}

\theoremstyle{TheoremNum}

\newcommand {\Reals}  {{\rm I \! R}}

\newcommand {\Procrustes}  {{\cP\cT}}

\newcommand{\1}{{\bf 1}}

\newcommand{\bbA}{\mathbb{A}}
\newcommand{\bbB}{\mathbb{B}}

\newcommand{\bbE}{\mathbb{E}}

\newcommand{\bbH}{\mathbb{H}}
\newcommand{\bbI}{\mathbb{I}}
\newcommand{\bbG}{{\mathbb{G}}}

\newcommand{\bbM}{{\mathbb{M}}}

\newcommand{\bbP}{\mathbb{P}}

\newcommand{\bbT}{\mathbb{T}}
\newcommand{\bbU}{\mathbb{U}}

\newcommand{\bbV}{\mathbb{V}}
\newcommand{\bbW}{{\mathbb{W}}}
\newcommand{\bbX}{{\mathbb{X}}}

\newcommand{\bbZ}{\mathbb{Z}}

\newcommand{\expc}{{\E}}

\newcommand{\cA}{\mathcal{A}}
\newcommand{\cB}{\mathcal{B}}
\newcommand{\cC}{\mathcal{C}}

\newcommand{\cG}{\mathcal{G}}
\newcommand{\cK}{\mathcal{K}}

\newcommand{\cL}{\mathcal{L}}
\newcommand{\cS}{\mathcal{S}}
\newcommand{\cT}{\mathcal{T}}
\newcommand{\cU}{\mathcal{U}}

\renewcommand{\b}{{\mathbf{b}}}

\renewcommand{\v}{{\mathbf{v}}}

\newcommand{\x}{{\mathbf{x}}}

\newcommand{\I}{{\mathbf{I}}}

\newcommand{\E}{{\mathbf{E}}}
\renewcommand{\P}{{\mathbf{P}}}

\newcommand{\Var}{{\rm Var}}

\makeatletter
\newcommand{\doublewidetilde}[1]{{%
		\mathpalette\double@widetilde{#1}%
}}
\newcommand{\double@widetilde}[2]{%
	\sbox\z@{$\m@th#1\widetilde{#2}$}%
	\ht\z@=.9\ht\z@
	\widetilde{\box\z@}%
}
\makeatother

\newcommand{\cN}{{\cal N}}
\newcommand{\cO}{{\cal O}}
\newcommand{\cP}{{\cal P}}
\newcommand{\cQ}{{\cal Q}}

\newcommand{\DiaEig}{{ \Lambda }}
\newcommand{\diag}{{\rm diag}}

\newcommand{\rank}{{\rm rank}}
\newcommand{\tr}{{\rm tr}}

\newcommand{\NumOfB}{{K }}
\newcommand{\BernP}{\text{Bernoulli}}
\newcommand{\BExpect}{\bbB_{(n)}}
\newcommand{\Irenyi}{{\I_{\text{Renyi}}^{\frac12}}}
\newcommand{\WignerLikeRM}{{ \bbH_{(n)} }}
\newcommand{\SparsityFactor}{{ s_{n} }}
\newcommand{\word}{{ \omega }}
\newcommand{\sentence}{{\cal S}}

\newcommand{\dps}{\displaystyle}
\endlocaldefs

\numberwithin{equation}{section}
\begin{document}

\begin{frontmatter}
	\title{Two-sample Test of Community Memberships of Weighted Stochastic Block Models}
	\runtitle{Inference of Weighted Stochastic Block Models}
	\begin{aug}
		\author{\fnms{Yezheng}  \snm{Li}\ead[label=e1]{yezheng@sas.upenn.edu}}
		\and
		\author{\fnms{Hongzhe} \snm{Li}\corref{Correspondence}\ead[label=e3]{hongzhe@pennmedicine.upenn.edu}\thanksref{t1}}
		\runauthor{Y. Li and H. Li}
		\affiliation{University of Pennsylvania}
		\address{Yezheng Li\\Program in Applied Mathematics and \\
			Computational Science \\ University of Pennsylvania \\ Philadelphia, PA 19104, USA. \\ \printead{e1}}
		\address{Hongzhe Li \\ Department of Biostatistics, \\
			Epidemiology and Informatics (DBEI) \\ University of Pennsylvania \\ Philadelphia, PA 19104, USA. \\ \printead{e3}}
		\thankstext{t1}{This research was supported in part by the National Institutes of Health Grants GM123056 and GM129781.}
	\end{aug}

		\begin{abstract}
	Suppose two networks are observed for the same set of nodes, where each network is  assumed to be generated from a weighted stochastic block model. This paper  
	considers the problem of testing whether the community memberships  of the two networks are the same. A  test statistic based on singular subspace distance is developed. 
Under the weighted stochastic block models with  dense graphs,  the limiting distribution of the proposed test statistic is developed.  Simulation results show that the test has
correct 	empirical type 1 errors under the dense graphs.   The test also behaves as expected in  empirical power, showing gradual changes   when the intra-block  and  inter-block distributions are close and achieving  1 when the two distributions are not so close, where the  closeness of the two distributions is characterized by Renyi divergence of order 1/2.  The Enron email networks are used to demonstrate the proposed test.  
		\end{abstract}
		
		\begin{keyword}[class=MSC]
			\kwd[Primary ]{62E20, 62F03, 62H15, 91D30}
			\kwd{}
			\kwd[Secondary ]{62F05}
		\end{keyword}
		
		\begin{keyword}
			\kwd{Moment matching, Networks, Procrustes transformation, Random matrix,   Singular subspace distance, Spectral clustering}
		\end{keyword}
		
	\end{frontmatter}

	\section{Introduction}
	Network data appear in many disciplines such as social science, neuroscience, and genetics. Many models have been proposed for network data, among which the  stochastic block models (SBMs) \citep{holland1983stochastic} 
	have emerged as a popular statistical framework for modeling network data with  community  structures.  SBMs are a class of generative models for describing the community structure in unweighted networks. The model assigns each of $n$ nodes to one of $K_n$ blocks, and each edge  exists with a probability specified by the block memberships of their endpoints.   To account for edge weights,   the observations are given in the form of a weighted adjacency matrix.  As extensions of unweighted SBMs, weighted 
	SBMs have been proposed, where the weight of each edge is generated independently from some  probability density determined by the community membership of its endpoints \citep{jog2015information,jog2015recovering,xu2017optimal}.

	Alternative to SBMs,  random dot product graph (RDPG)  models have been proposed where the  adjacency matrix of the nodes  is  generated from Bernoulli distributions with probabilities defined through  latent positions. The latent positions can be random and generated from some distribution. Such RDPG models are related to stochastic block model graphs  and degree-corrected stochastic block model graphs \citep{karrer2011stochastic}, as well as mixed membership block models \citep{airoldi2008mixed}.
	
	
	Community identification in a network is an important problem in network data analysis.  Spectral clustering is one of the mostly studied methods for community identification based on  SBMs \citep{von2008consistency, rohe2011spectral, mossel2012reconstruction, mcsherry2001spectral, lei2015consistency, lei2016goodness, zhang2016minimax, joseph2016impact, schiebinger2015geometry}.  \cite{lei2015consistency} showed that, under mild conditions, spectral clustering applied to the adjacency matrix of the network can consistently recover the hidden communities even when the order of the maximum expected degree is as small as $\log n$ where   $n$ is the number of nodes. \cite{xu2017optimal,jog2015information} established the optimal rates for community estimation in the weighted SBMs. 
	 \cite{lei2016goodness,bickel2016hypothesis} developed goodness of fit tests on number of clusters $\NumOfB$ for SBMs.

	This paper considers the problem of two-sample inference in the setting that two networks are observed for the same set of nodes, where each network is assumed to be generated from a weighted SBM. We specifically 
consider the problem of testing whether the community memberships  of the two networks are the same.  Such tests have many applications. For example, one might be interested in testing whether there is a change of community structures over time and whether a set of genes have different  network structures between disease and normal states.   This problem has not been studied in literature for weighted SBMs.  There are some related inference works developed  for the  RDPGs \citep{athreya2018statistical},  but these methods do not  treat the block  memberships as the parameters of interest.   \cite{tang2017nonparametric,tang2017semiparametric} considered the problem of testing whether two independent finite-dimensional 
random dot product graphs have the same generating latent positions or 
the respective generating latent positions  are scaled or diagonal transformations of  each other.  \cite{cape2017two,athreya2016limit,tang2018limit,cape2018signal,tang2017semiparametric,tang2017nonparametric} extend the discussion on an interesting asymptotic expansion of subspace distance in Frobenius norm and considered the two-sample test problem by upper bounds of subspace distance (or its  variants), but the limiting distribution for test statistic is unknown \citep{tang2017nonparametric,tang2017semiparametric}.  

\cite{GhoLux18} and  \cite{ghoshdastidar2017two} considered a different two-sample hypothesis testing problem, where one observes two random graphs, possibly of different sizes. Based on the two given network graphs, they are interested in testing  whether the underlying distributions that generate the graphs are same or different.  Their proposed test statistic is based on some summary statistics associated with the graphs. 

 Based on singular subspace distance in Frobenius norm, this paper derives a test statistic  of two-sample community memberships of weighted stochastic block models.  Different from the previous two-sample test statistics of \cite{tang2017nonparametric,tang2017semiparametric},  we derive the limiting distribution  of our proposed test statistic by moment matching method using random matrix theory for Gaussian ensembles.  Such results have not appeared in literature even for the  dense graphs.  A recent independent work of \cite{bao2018singular}  derived the normal distribution for singular subspace in Frobenius for low-rank matrices with  Wigner noises.  
The major difficulty to overcome is to prove that the asymptotic expansion  in \cite{tang2017nonparametric,tang2017semiparametric} still holds  in the dense graph region in order to derive mean and variance of our test statistic \eqref{eq:test statistic} (see Theorem \ref{thm:asymptotic expansion two views} in Section \ref{sec:asymp distri two views}).

The rest of the paper is organized as follows. Section \ref{sec:SBM} defines the homogeneous weighted SBMs and the conditions for dense graphs.  Section \ref{sec:two-sample}  presents the statistical definition of the null hypothesis that  two networks have the same community structures and presents our proposed test statistic.  Section \ref{sec:asymp distri two views} presents  its limiting distribution. Simulation results to evaluate the type I errors and the power of the proposed test are given in Section \ref{sec:simulation}. Section \ref{data}	 demonstrates the application of the proposed test to the  Enron email networks. Finally, Section \ref{sec:discuss} gives a brief discussion.  Detailed proofs can  be found in the Supplemental Materials. 

	\section{Homogeneous weighted SBM and  dense graph}\label{sec:SBM}
	\subsection{Homogeneous weighted SBM} Homogeneous weighted SBM of $n$ nodes with $\NumOfB$ underlying clusters is characterized by two set of parameters:  the underlying membership assignments $\bbZ_n \in \{0,1\}^{n \times \NumOfB}$ and  the  intra-, inter-edge distributions $\cP_n, \cQ_n$ \citep{xu2017optimal,lei2016goodness}.  For the sake of simplicity and similar to \cite{tang2018limit}, this paper assumes that $\NumOfB$ is fixed.

	The underlying membership assignments of a weighted SBM is characterized by $\bbZ_n $ where each row of $\bbZ_n \in \{0,1\}^{n \times \NumOfB}$ contains exactly one $1$, and  each column represents the assignments of a particular  membership. Here  $\bbZ_n$ is treated as fixed parameters for the model. Membership assignments can also be characterized by introducing  a mapping function $\cK$ \citep{jog2015recovering}, defined as 
	\begin{defn}
		\label{defn:true cluster assignment}
		Function $\cK: [n] \to \left[ \NumOfB \right]$ outputs the true membership assignment of each node $i$.
	\end{defn}

			Similar to \cite{gao2017achieving,xu2017optimal}, we make the following assumption on the size of each block: 
	\begin{assum}[Size of each block]
		\label{assum:size of each block}
		There exists $\beta \ge 1$ such that		$ \frac{n}{\beta \NumOfB} \le  \#\cC_i  \le \frac{\beta n}{ \NumOfB} $ for all $i \in [\NumOfB]$, which implies that $ \#\cC_i  \asymp  \frac{n}{ \NumOfB}$ for all $i \in [\NumOfB]$.
	\end{assum}

For the sake of simplicity of arguments in proofs, this paper considers that number of clusters $K$ is fixed and makes the following homogeneity assumption on the  intra-block, inter-block  edge distributions:
	\begin{assum}[Homogeneity] \label{assum:homogeneous weighted SBM}
		
The edge weight probability distributions $\cP_n, \cQ_n$ are supported on $S \subset \Reals^1$, where $S$ may be $[0,1]$, $[0, \infty)$ or $\Reals^1$.
		For $i\le j \in [n]$,
		$$w_{ij} \sim \left\{ \begin{array}{cc}		0, & i=j;\\ \cP_n , & \cK(i)=\cK (j), i <j;\\  \cQ_n , & \cK (i)\ne \cK (j). 
		\end{array}  \right. $$ where  $b_{\cP}, \sigma_{\cP}^2$ are mean and variance of intra-block distribution $\cP_n$ and  $b_{\cP}, \sigma_{\cQ}^2$ are mean and variance of inter-block  distribution $\cQ_n$;. We assume that $\sigma_{\cP}^2 \asymp  b_{\cQ} $, $\sigma_{\cQ}^2 \asymp  b_{\cQ} $, where 
$\cP,\cQ$ are  symbols for intra-block and inter-block distributions,  not the parameters  \citep{xu2017optimal}. While subscripts $~_{n},~_{(n)}$ emphasize the dependency on $n$, we ignore these subscripts in  $b_{\cP}, b_{\cQ}, \sigma_{\cP}^2, \sigma_{\cQ}^2$ for the sake of simplicity.
		
		
	\end{assum}

		As an example, consider the  unweighted SBM $\cG_{\NumOfB}\left( p_n,q_n\right)$, we have the  adjacency matrix with entity  $w_{ij}, (i \le j)$
		$$ w_{ij} \sim  \left\{ \begin{array}{lc}
		0, & i = j; \\
		\cP_n = \BernP \left( p_n \right), & \cK (i) = \cK (j), i <  j; \\ 	 \cQ_n =  \BernP \left( q_n \right), & \cK(i) \ne \cK(j).
		\end{array}\right. 
		$$ namely, for all $ \bbA_n \in  \{0,1\}^{n \times n}$ such that $ \bbA_n^T = \bbA_n, a_{ii}=0, i =1,\ldots, n$, the probability
		$$
		\bbP \left(  \bbW_n = \bbA_n \right)=  \prod_{i = 1}^{n-1} \left[\prod_{\substack{j: j>i \\\cK (j) = \cK(i)}} p_{n} ^{A_{ij} }\left(1-  p_n\right)^{1 - A_{ij}} \cdot \prod_{\substack{j: j>i \\\cK (j) \ne \cK(i)}} q_{n} ^{A_{ij} }\left( 1-  q_{n} \right)^{1 - A_{ij}}\right].
		$$
		In this case, Assumption \ref{assum:homogeneous weighted SBM} holds with means $b_{\cP} = p_{n}, b_{\cQ} = q_{n}$,  variances $\sigma_{\cP}^2 = p_{n}\left(1-p_{n} \right), \sigma_{\cQ}^2 = q_{n}\left(1-q_{n}\right)$.

For a given network, we observe a symmetric weight matrix $\bbW_n \in \Reals^{n \times n} = (w_{ij})_{ n \times n}$. For all $i < j\in[n]$
\ignore{\footnote{For the case with $i \le j$ should be easier to discuss \citep{rohe2011spectral}; concentration inequality \citep{oliveira2009concentration}, etc.  is actually discussing the case $i \le j$ -- \citep{tang2018limit,tang2017semiparametric} describe "of an undirected graph with all zeros ($\bbE( w_{ii} = 0 ) = 1$) on the diagonal"; however, concentration inequality they rely on refers to \citep{oliveira2009concentration}. On the other hand, it is true that some work focuses on adjacency matrix $\bbW_n$ with $\bbE\left[ \bbW_n | \bbZ_n\right] = \bbZ_n \BExpect \bbZ_n^T$ \citep{bickel2016hypothesis,lei2016goodness}, etc.}, }
the entry $w_{ij}$ are generated independently according to 
$w_{ij} \sim \cB_{\cK(i) \cK(j), (n) }$.
The	expectation of the weight matrix $\bbW_n$ is
\begin{equation}
\bbE_n \triangleq \E \bbW_n =   \bbZ_n \BExpect \bbZ_n^T - \diag\left( \bbZ_n \BExpect \bbZ_n^T \right)  \in \Reals^{ n \times n}, \label{eq:prob matrix SBM}
\end{equation}
where the symmetric matrix $\BExpect  = \left( b_{\cP} - b_{\cQ}\right) \bbI_{\NumOfB} + b_{\cQ} \1_{\NumOfB}\1_{\NumOfB}^T \in \Reals_{\ge 0}^{\NumOfB \times \NumOfB }$ represents expectation of intra-block  and  inter-block  distributions and $\diag(\bbM) = \diag  \{ m_{11}, \ldots, m_{ss}  \}  $ represents a diagonal matrix consisting of diagonal entries of $\bbM \in \Reals^{s \times s}$.

	\subsection{SBMs with dense graphs}
This paper focuses on SBMs with dense graphs and with the assumption on signal-to-noise ratio  defined by  Renyi divergence. 
As for sparsity of the graph, sparsity factor is analogous to \cite{tang2017semiparametric,  tang2018limit}.  The Renyi divergence and the dense graphs are  assumed to satisfy Assumption  \ref{assum:region of interest Renyi divergence}:
	\begin{assum}[Region of interest]
	\label{assum:region of interest Renyi divergence}

	\begin{equation}
	 b_{\cP} > b_{\cQ} \gtrsim  \SparsityFactor, \frac{n\Irenyi \left(\cP_n \| \cQ_n \right)}{\NumOfB \log n} \gtrsim 1, 
	\label{eq:region of interest Renyi divergence}
	\end{equation}
	where $\SparsityFactor = n^{-\frac12 + \epsilon}$ is the asymptotic lower bound for graph sparsity for some $\epsilon >0$, and the  Renyi divergence of order $\frac12$ is defined as 
	$$
	\Irenyi  \left( \cP_n, \cQ_n \right) \triangleq -2 \log \int  \left(\frac{d \cP_n }{d \cQ_n} \right)^{ \frac12 } d\cQ_n,
	$$
	where the lower threshold for Renyi divergence might not be tight.
\end{assum}

It is worth noting that sparsity factor threshold is consistent with \cite{tang2017asymptotically}.  For unweighted SBMs, Lemma B.1 in \cite{zhang2016minimax} provides relation between Renyi divergence of order $\frac12$ and SNR,  and Assumption \ref{assum:region of interest Renyi divergence} reduces to 
			\begin{equation}p_n>q_n \gtrsim \SparsityFactor , 
			SNR \triangleq  \frac{(p_n-q_n)^2}{p_n}\gtrsim \frac{\NumOfB \log n }n ,  
				\label{eq:SBM with lower bound for SNR}
			\end{equation}
			where the SNR is the signal-to-noise ratio frequently discussed in the literature of community recovery in SBM \cite{abbe2017community,athreya2018statistical}. Our SNR refers to summary table of exact recovery on page 18 of  \cite{abbe2017community}.
			\ignore{Furthermore, consider SBM with $p_{n} = p$, $q_{n} = p - n^{- \alpha}$, we have 
		\begin{eqnarray*}
			&& \Irenyi\left(\BernP (p_{n}) \| \BernP (q_{n} ) \right)\\
			 & = & \log \left( \sqrt{p_n q_n} + \sqrt{ (1- p_n) (1-q_n))} \right)\\
			& = & \log \left(\sqrt{p\left(p - n^{-\alpha}\right)} + \sqrt{(1-p)\left(1- p + n^{-\alpha}\right)} \right) \\
			&= & - \frac{n^{-2\alpha}}{8(1-p)p} + o\left(n^{-2\alpha}\right).
		\end{eqnarray*}
		While \cite{xu2017optimal} discussed the threshold around $nI_{Renyi} \to \infty$, this will lead to a threshold discussion near $\alpha = \frac12$ although \cite{xu2017optimal} does not discuss SVD/ spectral clustering.}

	\section{Procrustes Transformation and Property of Singular Subspace Distance for One Network}	\label{subsec:asymp distri one view}

Since  the spectral clustering is used to identify the community memberships  of the nodes of the two SBMs, we first provide  Definition \ref{defn:singular value decomposition}:
\begin{defn}
	\label{defn:singular value decomposition}
	For a symmetric $n \times n$ matrix $\bbG_n$, singular value decomposition is denoted as
	\begin{equation}
	\bbG_n= \sum_{i =1}^{n} \lambda_i \v_i \v_i^T= \bbV_{\bbG_n} \DiaEig_{\bbG_n}  \bbV_{\bbG_n}^T + \bbV_{\bbG_n}^{\perp} \DiaEig_{\bbG_n}^{\perp} \left( \bbV_{\bbG_n}^{\perp}\right)^T,|\lambda_1| \ge  \ldots \ge |\lambda_n| ,
	\label{eq:singular value decomposition}
	\end{equation}
	where $\DiaEig_{\bbG_n} = \diag\left\{  \lambda_1, \ldots, \lambda_{ \NumOfB} \right\}$ contains leading $\NumOfB$ singular components of $\bbG_n$ while $\DiaEig_{\bbG_n}^{\perp}$ contains the rest. $\bbV_{ \bbG_n } \in \Reals^{ n \times \NumOfB}$ may not be unique (due to multiple root) but just pick one collection.  
	
\end{defn}

In this paper, the singular value decomposition is applied to  the observed connection matrix  $\bbG_n$=$\bbW_n$ or its expected values  $\bbG_n=\bbE_n$ \citep{tang2017semiparametric,tang2018limit}. \ignore{ or (we do not discuss Laplacian in this paper) their corresponding Laplacian matrix  $\cL (\bbW_n)$, $\cL(\bbE_n)$ \citep{tang2018limit}.}
	
	To begin with, we define the  orthogonal Procrustes transformation from matrix $\bbV_{1} \in \Reals^{ n \times \NumOfB}$ to $\bbV_{2} \in \Reals^{ n \times \NumOfB}$:
	\begin{equation}
	\Procrustes\left( \bbV_{1} , \bbV_{2} \right) \in \arg \inf_{\bbU \in \cO (\NumOfB)} \left\| \bbV_{1} \bbU -  \bbV_{2} \right\|_F \subset \Reals^{\NumOfB \times \NumOfB},
	\label{eq:Procrustes transformation}
	\end{equation}
	where we do not need to specify the relationship  between $\Procrustes\left( \bbV_{1} , \bbV_{2} \right)$ and 	$\Procrustes \left(  \bbV_{2} ,  \bbV_{1}\right)$. We further define 
	\begin{equation}
	\left\|  \sin\Theta \left( \bbV_1, \bbV_2\right)  \right\|_F\triangleq	\left\| \bbV_{1}  \Procrustes\left( \bbV_{1} , \bbV_{2} \right)-  \bbV_{2} \right\|_F 
	\label{eq:sine theta distance in Frobenius norm}
	\end{equation}
as the $\sin\Theta$ distance in Frobenius norm.

We first establish the distance between singular vectors of $\bbV_{\bbW_n }$  and    $\bbV_{\bbE_n }$ after the Procrustes transformation $\bbT$. One natural definition is the Frobenius norm of two matrices
$\left\| \bbV_{\bbW_n } \bbT   - \bbV_{\bbE_n } \right\|_F$. However, the mean and variance of this distance is complicated; details of this argument  can be found  in Appendix \ref{sec:traditional sine theta distance}. Instead, we consider a modified and re-scaled quantity defined as 
$$\left\| \left(\bbV_{\bbW_n } \bbT   - \bbV_{\bbE_n } \right) \DiaEig_{ \bbE_n} \right\|_F,$$
which has a simpler mean and variance.  We have the following asymptotic expansion for the singular value decompositions of  the SBMs: 
	\begin{lem}
		\label{lem:tang asymptotic expansion}
	\begin{eqnarray}  
&& \frac1{\sqrt{\NumOfB n b_{\cP}	}} \left\| \left(\bbV_{\bbW_n } \bbT_n   - \bbV_{\bbE_n } \right) \DiaEig_{ \bbE_n} \right\|_F \label{eq:tang asymptotic expansion}\\
& = & 
	\frac1{\sqrt{ \NumOfB n b_{\cP}	 } } \left\|   \left( \bbW_n  - \bbE_n  \right) \bbV_{\bbE_n }   \right\|_F + O_P \left(  \sqrt{\frac{\NumOfB  }{nb_{\cP}} } \log (n)  \right) =	O_P(1) , \nonumber
	\end{eqnarray}
			where transformation matrix $\bbT$ is Procrustes transformation $ \Procrustes \left( \bbV_{\bbW_n } ,  \bbV_{\bbE_n }   \right) $.
\end{lem}
We provide a proof of this Lemma \ref{lem:tang asymptotic expansion} in section \ref{subsec:Proof of Lemma {lem:tang asymptotic expansion}} using the same technique as Theorem 2.1 of \cite{tang2018limit}. Lemma \ref{lem:tang asymptotic expansion} implies	
	\begin{eqnarray*}  
&& \frac1{ \NumOfB n b_{\cP} } \left\| \left(\bbV_{\bbW_n } \bbT_n   - \bbV_{\bbE_n } \right) \DiaEig_{ \bbE_n}  \right\|_F^2 = O_P(1)\\
&= & \frac1{\NumOfB n b_{\cP}	 }  \left\|   \left( \bbW_n  - \bbE_n  \right) \bbV_{\bbE_n }   \right\|_F^2 \\
		&  & 	+  	\frac{\left\|   \left( \bbW_n  - \bbE_n  \right) \bbV_{\bbE_n }   \right\|_F }{\sqrt{ \NumOfB n b_{\cP}	 } } \cdot O_P \left(  \sqrt{\frac{\NumOfB  }{nb_{\cP} } } \log (n)  \right)   + O_P \left(   \frac{\NumOfB   \left[\log (n)\right]^2   }{n b_{\cP}} \right) .
	\end{eqnarray*}

Theorem \ref{thm:asymptotic expansion one view}  shows  that the second term in the above asymptotic expansion can be removed.
	\begin{thm} \label{thm:asymptotic expansion one view}
		Under the Assumptions \ref{assum:size of each block}, \ref{assum:homogeneous weighted SBM}, \ref{assum:region of interest Renyi divergence}, we have
		\begin{eqnarray} \label{eq:simplification of Cos Sin matrix AJ}
		\frac1{\NumOfB n} \left\| \left(\bbV_{\bbW_n } \bbT   - \bbV_{\bbE_n } \right) \DiaEig_{ \bbE_n} \right\|_F^2  
& = &  \frac1{\NumOfB n} \left\|   \left( \bbW_n  - \bbE_n  \right) \bbV_{\bbE_n }   \right\|_F^2  + O_P \left( \frac{\NumOfB [\log (n)]^2 }{n} \right), \nonumber 
		\end{eqnarray}
		where transformation matrix $\bbT_n$ is Procrustes transformation $ \Procrustes \left( \bbV_{\bbW_n } ,  \bbV_{\bbE_n }   \right) $, we remove $b_{\cP}$ in the denominator in order to be consistent with later test statistic \eqref{eq:test statistic} in two-sample problem \eqref{eq:hypothesis test for SBM AJ}; we have
		$$\frac1{\NumOfB n} \left\| \left(\bbV_{\bbW_n} \bbT_n   - \bbV_{\bbE_n} \right) \DiaEig_{ \bbE_n} \right\|_F^2  = \Theta_P \left( \sigma_{\cQ}^2 +  \frac{\sigma_{\cP}^2  - \sigma_{\cQ}^2 }{\NumOfB} \right).$$
		
		Consequentially, variance of the linear representation dominates as well:
		\begin{eqnarray*}
			&& \Var \left[\frac1{ \NumOfB n} \left\| \left(\bbV_{\bbW_n} \bbT_n   - \bbV_{\bbE_n} \right) \DiaEig_{ \bbE_n} \right\|_F^2  \right] =  O \left(  \frac{b_{\cQ}^2}{nK}  \right) \\
			&= &  \Var\left[ \frac1{\NumOfB n}\left\|   \left( \bbW_n - \bbE_n \right) \bbV_{\bbE_n}   \right\|_F^2  \right]+ O \left(  \frac{\NumOfB^2 [\log (n)]^4 }{n^2}  \right).
		\end{eqnarray*}
	\end{thm}

	This asymptotic expansion lead to the  limiting distribution of 
	$ \frac1{ \NumOfB n} \left\| \left(\bbV_{\bbW} \bbT_n   - \bbV_{\bbE} \right) \DiaEig_{ \bbE} \right\|_F^2$ as stated in the  the following Theorem. 
	\begin{thm} 
		\label{thm:central limit theorem for frobenius norm}
		Under the assumption \ref{assum:size of each block}, \ref{assum:homogeneous weighted SBM}, \ref{assum:region of interest Renyi divergence},  we have 
		\begin{equation}
		\frac{ 	\frac1{\NumOfB n} \left\| \left(\bbV_{\bbW_n} \bbT_n   - \bbV_{\bbE_n} \right) \DiaEig_{ \bbE_n} \right\|_F^2  - \widetilde{\mu}_n}{ \sqrt{\widetilde{\Var_n} } } \to \cN(0,1),
		\end{equation}
				where transformation matrix $\bbT_n$ is Procrustes transformation $ \Procrustes \left( \bbV_{\bbW_n } ,  \bbV_{\bbE_n }   \right) $; the mean is $\widetilde{\mu}_n =   \sigma_{\cQ}^2 +  \frac{\sigma_{\cP}^2  - \sigma_{\cQ}^2 }{\NumOfB}$ and the variance is 
		$$\widetilde{\Var_n} =\frac2{n \NumOfB}  \left( \sigma_{\cQ}^4 + \frac{\sigma_{\cP}^4 - \sigma_{\cQ}^4 }{\NumOfB} \right)+ O\left( \frac{\NumOfB^2 [\log (n)]^4 }{n^2 } \right) .$$
		
		We use symbols $\widetilde{\mu}_n$, $\widetilde{\Var_n}$ just because symbols $\mu_n$, $\Var_n$ are reserved for the mean of variance of our test statistic $T_{n,\NumOfB}$ \eqref{eq:test statistic} in Theorem \ref{thm:asymptotic expansion two views}.
	\end{thm}
	
	\section{Two-sample Hypothesis Test  of Community Memberships Based on SBMs}\label{sec:two-sample}
	
	\subsection{A two-sample test problem}
	Consider the setting where we have two  independent networks with the same group of $n$ nodes and each is generated from a weighted SBM with underlying membership assignment $\bbZ_n^{(v)}, v=1,2$.  We are interested in testing  whether underlying block assignments are the same; in other words, testing 
	\begin{equation}
		H_0: \bbZ_n^{(1)} \upVdash \bbZ_n^{(2)}\mbox{ versus }H_1: \bbZ_n^{(1)} \not \upVdash \bbZ_n^{(2)}, \label{eq:original hypothesis test}
	\end{equation}
	where for two matrices $\bbM_1,\bbM_2 \in \Reals^{n \times \NumOfB}$, $\bbM_1 \upVdash  \bbM_2 $ means there exists $\bbU \in \cO(\NumOfB)$ such that $\bbM_1 = \bbM_2 \bbU$ and $\cO (\NumOfB)$ represents the set of $\NumOfB \times \NumOfB$ orthogonal matrices.

		For a weighted SBM, it is known that 
			\begin{equation}
			\bbV_{\bbE_n} \upVdash \bbZ_n ( \bbZ_n^T \bbZ_n)^{-\frac12}.
			\label{eq:normalized index parameter simeq AJ}
			\end{equation} 
In addition,   $\bbZ_n^{(1)} \upVdash \bbZ_n^{(2)}$ if and only if 
		$$\bbZ_n^{(1)} \left([\bbZ_n^{(1)}]^T \bbZ_n^{(1)} \right)^{-\frac12} \upVdash \bbZ_n^{(2)} \left([\bbZ_n^{(2)}]^T \bbZ_n^{(2)} \right)^{-\frac12}.$$  Since an orthogonal matrix is actually a permutation,  \eqref{eq:normalized index parameter simeq AJ}  implies that $H_0: \bbZ_n^{(1)} \upVdash \bbZ_n^{(2)}$  is equivalent to
	\begin{equation}
		H_0: \bbV_{\bbE_n^{(1)}} \upVdash \bbV_{\bbE_n^{(2)}} \mbox{ versus }H_1: \bbV_{\bbE_n^{(1)}} \not \upVdash \bbV_{\bbE_n^{(2)}}  . \label{eq:hypothesis test for SBM AJ}
	\end{equation}

	In order for this null hypothesis to be practically meaningful, we make an additional  Assumption  \ref{assum:underlying probability the same up to a scalar} on  the intra-block  and  inter-block  distributions:
	\begin{assum}
		\label{assum:underlying probability the same up to a scalar}
		For the  expectations of intra-block and  inter-block distributions stated in Assumption  \ref{assum:homogeneous weighted SBM}), we assume $\BExpect^{(2)} = \gamma \BExpect^{(1)}$, or equivalently, $\gamma =  \frac{b_{\cP}^{(1)}}{b_{\cP}^{(2)}} =  \frac{b_{\cQ}^{(1)}}{b_{\cQ}^{(2)}} $.
	\end{assum} 	
	The assumption may seem restrictive. However, if  the edge generating functions are different, the underlying network structures will be different and the the null is usually easy to reject.

	\subsection{Two-sample test statistic} \label{sec:test}
	
	Our proposed test statistic is  also based on the Procrustes transformation but we include an $\NumOfB \times \NumOfB$ matrix multiplier $\DiaEig_{\bbW_n^{(2)}}$ in order to simplify the calculations of the mean and variance (see Theorem \ref{theorem:mean for eigenspace distance}) of the test statistic:
	\begin{equation}
	T_{n,\NumOfB} = T_{n,\NumOfB} \left(\bbV_{\bbW_n^{(1)}}  ,\bbV_{ \bbW_n^{(2)} }\right) \triangleq  \frac1{n \NumOfB } \left\|\left(   \bbV_{\bbW_n^{(1)}}\bbT_n - \bbV_{\bbW_n^{(2)} }  \right)  \DiaEig_{\bbW_n^{(2)}} \right\|_F^2,  \label{eq:test statistic}
	\end{equation}
	where $\bbT_n$ is  the Procrustes transformation $\Procrustes \left( \bbV_{\bbW_n^{(1)}} , \bbV_{\bbW_n^{(2)} } \right)$.

It is important to point out the difference between our test statistic and the one in \cite{tang2017nonparametric,tang2017semiparametric}. 
First, our formulation of the null hypothesis test is  different from that of RDPGs since RDPGs are parametrized by latent position parameters.  \cite{tang2017nonparametric,tang2017semiparametric} developed a two-sample test on those latent position parameters.
Secondly, we  provide the  limiting distribution of our test statistic by using random matrix theory. In contrast,  \cite{tang2017nonparametric,tang2017semiparametric} proposed to apply bootstrap  to the  test statistic based on an upper bound estimation.

	\section{Asymptotic distribution of two-sample test statistic \eqref{eq:test statistic}}
		\label{sec:asymp distri two views}

\subsection{Asymptotic distribution of the proposed test}
Parallel  to the results in Theorem \ref{thm:asymptotic expansion one view}, we have the following asymptotic expansion for the two-sample test statistic $T_{n,\NumOfB}$:
	\begin{thm}
		\label{thm:asymptotic expansion two views}
		Assume that Assumptions \ref{assum:size of each block}, \ref{assum:homogeneous weighted SBM}, \ref{assum:region of interest Renyi divergence}, \ref{assum:underlying probability the same up to a scalar} hold. Under the null of \eqref{eq:hypothesis test for SBM AJ}, we have
		\begin{eqnarray}  
		T_{n,\NumOfB}&=&\frac1{ \NumOfB n}\left\| \left(  \bbV_{\bbW_n^{(1)}} \bbT_n - \bbV_{\bbW_n^{(2)} } \right) \DiaEig_{ \bbW_n^{(2)}}\right\|_F^2 
		\label{eq:asymptotic expansion in two laws} \\
		&= &    \frac1{\NumOfB n} \left\| \left( \gamma \bbW_n^{(1)} - \bbW_n^{(2)}  \right) \bbV_{\bbW_n^{(2)} }  \right\|_F^2 + O_P \left(  \frac{\NumOfB [\log (n)]^2 }n  \right), \nonumber 
		\end{eqnarray}
				where transformation matrix $\bbT_n$ is Procrustes transformation $ \Procrustes \left( \bbV_{\bbW_n^{(1)} } ,  \bbV_{\bbW_n^{(2)} }   \right) $.
		As a result, the  variance of the asymptotic expansion dominates as well,
		\begin{eqnarray*}
			\Var \left[T_{n,\NumOfB} \right] &  =  &  \Var\left[ \frac1{\NumOfB n} \left\| \left( \gamma  \bbW_n^{(1)} - \bbW_n^{(2)}  \right) \bbV_{\bbW^{(2)} }  \right\|_F^2 \right]+ O \left(  \frac{\NumOfB^2 [\log (n)]^4}{n^2}  \right) 
			=  O \left(  \frac{b_{\cQ}^2}{nK}  \right).
		\end{eqnarray*}
	\end{thm}

This asymptotic expansion  leads to the limiting distribution of our test statistic \eqref{eq:test statistic}, as stated in the Theorem \ref{thm:asym distri testing statistics AJ} below:
	
	\begin{thm} \label{thm:asym distri testing statistics AJ}
		Under the Assumptions \ref{assum:size of each block}, \ref{assum:homogeneous weighted SBM}, \ref{assum:region of interest Renyi divergence}, \ref{assum:underlying probability the same up to a scalar},   our proposed test statistic \eqref{eq:test statistic} have the following asymptotic distribution under the null of \eqref{eq:hypothesis test for SBM AJ}, 
		\begin{equation}  
		\frac{T_{n,\NumOfB}    - \mu_n }{\sqrt { \Var_n } } \to \cN(0,1), 
		\end{equation}
				where\ignore{ transformation matrix $\bbT_n$ is Procrustes transformation $ \Procrustes \left( \bbV_{\bbW_n^{(1)} } ,  \bbV_{\bbW_n^{(2)} }   \right) $;} the mean is $$ \mu_n =  \gamma^2  \sigma_{\cQ}^{2,(1)}  + \sigma_{\cQ}^{2,(2)} +  \frac{
			1  }{\NumOfB} \left[   \gamma \sigma_{\cP}^{2,(1)}+\sigma_{\cP}^{2,(2)}  - \left( \gamma^2 \sigma_{\cQ}^{2,(1)} + \sigma_{\cQ}^{2,(2)} \right) \right], $$ and the variance is 
\begin{eqnarray*} \Var_n
&=  & \frac{2}{n\NumOfB}\left[  \left(  \gamma^2 \sigma_{\cQ}^{2,(1)} + \sigma_{\cQ}^{2,(2)}\right)^2 +  \frac{\left( \gamma^2 \sigma_{\cP}^{2,(1)}  + \sigma_{\cP}^{2,(2)}\right)^2  - \left( \gamma^2 \sigma_{\cQ}^{2,(1)}  + \sigma_{\cQ}^{2,(2)}\right)^2}{\NumOfB}       \right]   \\
&& +O\left(\frac{ \NumOfB^2 [\log (n)]^4}{n^2  } \right).
\end{eqnarray*}
	\end{thm}

	In practice, $\mu_n$, $\Var_n$ have to be  estimated and their estimates need to be corresponding well-behaved estimators:
	\begin{defn}[Well-behaved estimators] \label{defn:well behaved estimators description	}
		Define  the  well-behaved estimators as those that 
		\begin{equation}
			\widehat{   \mu_n  } -  \mu_n 	= o_P \left( b_{\cP} \right) ,   \widehat{\Var_n}   -  \Var_n =  o_P\left( \frac{b_{\cP}^2}{n \NumOfB}\right).
			\label{eq:well behaved estimators description	}
		\end{equation}

	\end{defn}
	
Such well-behaved estimates can be obtained by  plugging well-behaved estimators $ \widehat{ \gamma^2 \sigma_{\cQ}^{2,(1)} + \sigma_{\cQ}^{2,(2)}}$, $ \widehat{ \gamma^2 \sigma_{\cP}^{2,(1)} + \sigma_{\cP}^{2, (2)} }$ for $  \gamma^2 \sigma_{\cQ}^{2,(1)} + \sigma_{\cQ}^{2, (2)}$, $  \gamma^2 \sigma_{P_n }^{2,(1)} + \sigma_{\cP} ^{2,(2)}$  \citep{jog2015information,jog2015recovering,xu2017optimal,mossel2012reconstruction,mcsherry2001spectral,tang2017asymptotically}. We  have the following corollary \ref{cor:asym distri testing statistics AJ -- well behaved estimator} when estimates of means and variances are used in the test statistic.
	\begin{cor}
		\label{cor:asym distri testing statistics AJ -- well behaved estimator}
		Suppose that Assumptions \ref{assum:size of each block},, \ref{assum:homogeneous weighted SBM}, \ref{assum:region of interest Renyi divergence}, \ref{assum:underlying probability the same up to a scalar} hold, our proposed test statistic \eqref{eq:test statistic} have the following asymptotic distribution under the null hypothesis of \eqref{eq:hypothesis test for SBM AJ},   we have
		\begin{equation}  
			\frac{ T_{n,\NumOfB}     - \widehat{   \mu_n  } }{\sqrt{ \widehat{  \Var_n} }} \to \cN(0,1), 
		\end{equation}
		where the mean is $$\widehat{   \mu_n }=  \widehat{\gamma^2  \sigma_{\cQ}^{2,(1)}  + \sigma_{\cQ}^{2,(2)} } +  \frac{
			1  }{\NumOfB} \left[   \widehat{ \gamma \sigma_{ \cP }^{2,(1)}+\sigma_{\cP}^{2,(2)} } -\widehat{  \gamma^2 \sigma_{\cP}^{2,(1)} + \sigma_{\cQ}^{2,(2)} } \right], $$
		and the variance is  
\begin{eqnarray*}  \widehat{ \Var_n} 
	& =  & \frac{2}{n \NumOfB}\left[  \widehat{ \gamma \sigma_{\cQ}^{2,(1)}+\sigma_{\cQ}^{2,(2)} }^2 +  \frac{ \widehat{ \gamma \sigma_{\cP }^{2,(1)}+\sigma_{\cP}^{2,(2)} }^2  -\widehat{ \gamma \sigma_{\cQ}^{2,(1)}+\sigma_{\cQ}^{2,(2)} }^2}{\NumOfB}     \right] \\
	&&  +O\left(\frac{\NumOfB^2 [\log (n)]^4}{n^2  } \right)  .
\end{eqnarray*}
	\end{cor}

\subsection{Asymptotic power of the proposed test}
\label{subsec:asmptotic power}
 
 	We evaluate the power of the proposed test by specifying the alternative using 
 the Hamming distance between the community memberships,  $\bbZ_n^{(1)}$ and $\bbZ_n^{(2)}$ of $n$ nodes 
 \begin{equation} \ell_n \left( \bbZ_n^{(1)}, \bbZ_n^{(2)} \right)  \triangleq \frac{1}{n} \min_{\Pi \in \cO\left( \NumOfB \right) }  d_H \left( \bbZ_n^{(1)}, \Pi \circ \bbZ_n^{(2)} \right),
 \label{eq:Hamming distance}
 \end{equation}
 where $d_H\left( \cdot , \cdot \right)$ denotes the Hamming distance, and  $\NumOfB$  is the permutation matrix.  We consider the following  hypothesis test with $\epsilon' >0$:
 \begin{equation} H_0': \ell_n = 0 \mbox{ v.s. } H_1': \ell_n \succeq  \frac{\NumOfB}{n^{1 - \epsilon'} \sqrt{ \mu_n} }, \label{eq:hypothesis test -- Hamming distance}
 \end{equation}
 where $\mu_n$ appears in Theorem \ref{thm:asym distri testing statistics AJ}. 
We further assume $\frac{n}{\NumOfB}$ is an integer and $\beta = 1 $ in Assumption \ref{assum:size of each block}. In this simple scenario with equal-size assumption,
  \begin{eqnarray*}
T_{n,\NumOfB}  	&=& \frac1{n\NumOfB}\left\|  \left[\bbV_{\bbW_n^{(1)}}  \Procrustes \left( \bbV_{\bbW_n^{(1)}}  , \bbV_{\bbW_n^{(2)}} \right)   - \bbV_{\bbW_n^{(2)}} \right] \Lambda_{\bbW_n^{(2)}}\right\|_F^2  \\
  	& \asymp_P & \frac1{n\NumOfB}\left\| \left[\bbZ_n^{(1)} \left( \left[ \bbZ_n^{(1)}\right]^T  \bbZ_n^{(1)}\right)^{-\frac12 }\Procrustes_{\bbZ}	-  \bbZ_n^{(2)} \left( \left[ \bbZ_n^{(2)}\right]^T  \bbZ_n^{(2)}\right)^{-\frac12 } \right] \Lambda_{\bbW_n^{(2)}} \right\|_F^2 \\
  	& = & \frac1{n\NumOfB}\left(\sqrt{\frac{\NumOfB}{n}} n\ell_n \cdot \frac{nb_{\cP}}{\NumOfB}\right)^2  = \frac{n^2 b_{\cP}^2 \ell_n^2}{\NumOfB^2} \succeq  n^{2 \epsilon'}\mu_{n}\succ  \mu_n,
  \end{eqnarray*}
where $$\Procrustes_{\bbZ} =  \Procrustes \left(\bbZ_n^{(1)} \left( \left[ \bbZ_n^{(1)}\right]^T  \bbZ_n^{(1)}\right)^{-\frac12 } , \bbZ_n^{(2)} \left( \left[ \bbZ_n^{(2)}\right]^T  \bbZ_n^{(2)}\right)^{-\frac12 } \right).$$
Consequentially, we have  the following results on the power of the proposed test:

	\begin{thm}[Asymptotic power guarantee] Assume that Assumptions \ref{assum:size of each block}, \ref{assum:homogeneous weighted SBM},  \ref{assum:region of interest Renyi divergence} and  \ref{assum:underlying probability the same up to a scalar} hold. In addition,  assume that $\frac{n}{\NumOfB}$ is an integer and $\beta = 1 $ in Assumption \ref{assum:size of each block}. Then under the alternative $H_1'$ of \eqref{eq:hypothesis test -- Hamming distance},
		$$
		T_{n,\NumOfB} \succeq_P n^{2\epsilon'} \mu_n \succ \mu_n,
		\label{eq:asymptotic power guarantee}
		$$
		where $\mu_n$ appears in Theorem \ref{thm:asym distri testing statistics AJ}; or equivalently,
				$$
\frac{		T_{n,\NumOfB} - \mu_n }{ \sqrt{\Var_n }}\succeq_P   n^{\frac12 + 2\epsilon'}\cdot \frac{\mu_n \sqrt{\NumOfB}}{b_{\cP}} \succeq n^{\frac12 + 2\epsilon'}\sqrt{\NumOfB}.
		$$
		
		Consequentially, for any two-sided $\alpha$ level test with $q_{\frac{\alpha}{2}}$,$ q_{\left(1-\frac{\alpha}{2}\right)}$ the $\frac{\alpha}{2}$-quantile and $\left(1-\frac{\alpha}{2}\right)$-quantile of Gaussian distribution, the probability under the alternative $H_1'$ of \eqref{eq:hypothesis test -- Hamming distance} satisfies
		
		$$
		\P_{H'_1} \left(q_{\frac{\alpha}{2}} < \frac{		T_{n,\NumOfB} - \mu_n }{ \sqrt{\Var_n }} < q_{\left(1-\frac{\alpha}{2}\right)}\right) \to 1.
		$$
		
		\label{thm:asymptotic power guarantee}
	\end{thm}

	\section{Simulation Studies}
	\label{sec:simulation}

	\subsection{Type I errors}
	We first evaluate the type I  errors of the proposed test. 
Tables \ref{tab:Type I error unweighted SBM Kn=2}
show the empirical type I errors of the proposed tests for different families of weighted SBMs based on 4,000 replications.

The first model considers  unweighted SBMs with $p_n =0.5,q_n=0.1$ ($p_n \asymp1$) for different sample sizes $n=500, 1000, 2000$ and $4000$ and different parameter values of $\lambda=1.5,1.3,1.0,1.8$ and $0.7$.    Overall, the typer I errors are under control, except that when the sample size is small and $\lambda=0.7$.

	The second model considers  a  family of weighted SBMs with $\cP_n = \chi^2 (5), \cQ_n = \chi^2 \left(1 \right)$. In this case $b_{\cP} \asymp 1$,  similar type 1 errors are observed as the unweighted SBMs for different values of $\lambda$ and different sample sizes. 
	
	 The third model considers the unweighted SBM  with $p = 1.8n^{ - \frac16 }, q=0.36n^{-\frac16}$. This problem is more complex: type I error is expected to converge when $n\to \infty$, while sparsity makes the convergence rate slower. Overall, the type I errors are under control.

	\begin{table}[htbp] \caption{Type I error of two sided test with significance level $\alpha = 5\%$ on unweighted SBM with $p = 0.5, q=0.1, \NumOfB = 2, \# \cC_{ 1} =2 \# \cC_{ 2}$. Run 4000 times for each data point.}
		\label{tab:Type I error unweighted SBM Kn=2}
		\begin{center}
			\begin{tabular}{cccccc} \hline 
				$n$ & $\gamma = 1.5$&$\gamma = 1.3 $ & $\gamma= 1$ & $\gamma = 0.8$ &  $\gamma = 0.7$ 
				\\ \hline
				\multicolumn{1}{c}{} &\multicolumn{5}{c}{unweighted SBM with $p = 0.5, q=0.1$} \\
			\ignore{	\multirow{3}{*}{500} & $\left( \bbV_{\bbW^{(2)}}^T \bbV_{\bbW^{(1)}}\right)^{-1}$ &{6.1\%}& 6.2\%  &{6.8\%} & 9.4\% &{10.2\%} 
				\\ 
				&$\bbV_{\bbW^{(1)}}^T \bbV_{\bbW^{(2)}}$ &{4.9\%}& 4.9\% &{5.0\%} & 6.2\% &{9.5\%} 
				\\  &$\Procrustes \left(\bbV_{\bbW^{(1)}} , \bbV_{\bbW^{(2)}}\right)$ }
500 &5.3 & 5.2 & 4.7 &5.0  &9.7 
				\\ 
				\ignore{\multirow{3}{*}{1000} & $\left( \bbV_{\bbW^{(2)}}^T \bbV_{\bbW^{(1)}}\right)^{-1}$ & 5.7 & 5.3  & 6.2 &7.5 & 9.3 
				\\ 
				& $\bbV_{\bbW^{(1)}}^T \bbV_{\bbW^{(2)}}$& 4.4  &  4.9 &4.8 & 5.7 & 6.4 
				\\ 
				& $\Procrustes \left( \bbV_{\bbW^{(1)}} , \bbV_{\bbW^{(2)}}\right)$}
1000				& 4.9& 5.2 & 4.8 & 4.8  &7.3 
				\\
			\ignore{	\multirow{3}{*}{2000} & $\left( \bbV_{\bbW^{(2)}}^T \bbV_{\bbW^{(1)}}\right)^{-1}$ & 5.3 & 5.6  & 5.6 & 6.7 &6.6 
				\\ & $\bbV_{\bbW^{(1)}}^T \bbV_{\bbW^{(2)}}$ & 5.1 &   5.0 &5.1 & 4.3 & 5.7 
				\\  &$ \Procrustes \left( \bbV_{\bbW^{(1)}} , \bbV_{\bbW^{(2)}}\right)$}
2000				&  4.9 &5.5 & 5.3& 5.3 & 5.6 
				\\ 
			\ignore{	\multirow{3}{*}{4000} & $\left( \bbV_{\bbW^{(2)}}^T \bbV_{\bbW^{(1)}}\right)^{-1}$ &5.5 & 5.8  & 5.7 & 5.8 &  6.6 
				\\  
				& $\bbV_{\bbW^{(1)}}^T \bbV_{\bbW^{(2)}}$ & 4.6 & {5.7} & {5.3} &  4.8 &  5.8  
				\\ 
				& $\Procrustes \left( \bbV_{\bbW^{(1)}} , \bbV_{\bbW^{(2)}} \right)$}
4000				& 4.5  & 5.1 & 5.2 &4.7 & 6.1  
				\\ 
		\multicolumn{1}{c}{} &\multicolumn{5}{c}{weighted SBM with $\cP_n = \chi^2 (5), \cQ_n = \chi^2 \left(1 \right)$} \\			
				
500				  &5.3 & 5.7  & 4.7  &  5.2 & 5.4   
\\ 
\ignore{				\multirow{3}{*}{1000} & $\left( \bbV_{\bbW^{(2)}}^T \bbV_{\bbW_n^{(1)}}\right)^{-1}$& 5.5 &5.5& 5.5 &5.0 & 4.9 & 5.5 
	\\ 
	& $\bbV_{\bbW_n^{(1)}}^T \bbV_{\bbW_n^{(2)}}$ &4.3 & 5.4 &  5.2 & 4.3   & 4.9 &  4.6 
	\\ 
	& $\Procrustes \left( \bbV_{\bbW_n^{(1)}} , \bbV_{\bbW_n^{(2)}} \right)$ }
1000		&5.2 & 4.6 &5.2 & 5.9 &5.0 
\\ 
\ignore{				\multirow{3}{*}{2000} & $ \left(\bbV_{\bbW_n^{(2)}}^T \bbV_{\bbW_n^{(1)}} \right)^{-1 }$ &5.2 & 4.5 & 4.7 & 5.4 & 5.0 & 5.4 
	\\ 
	& $\bbV_{\bbW_n^{(1)}}^T \bbV_{\bbW_n^{(2)}}$ & 4.8  & 5.4 & 5.5  & 5.7  &5.0 & 5.4  
	\\ 
	& $ \Procrustes \left( \bbV_{\bbW_n^{(1)}} , \bbV_{\bbW_n^{(2)}} \right)$}
2000		&  5.2 & 4.7 &  5.4 & 5.3	 & 5.6   
\\ 
\ignore{					\multirow{3}{*}{4000} & $ \left(\bbV_{\bbW_n^{(2)}}^T \bbV_{\bbW_n^{(1)}} \right)^{-1 }$ &5.0 &  5.1 & {5.0} & 5.3 &{5.0} & 5.1 
	\\ 
	& $\bbV_{\bbW_n^{(1)}}^T \bbV_{\bbW_n^{(2)}}$ & 4.9  & 5.1   &{4.9} &  5.3   & {5.2} & 5.0  
	\\ 
	& $ \Procrustes \left(\bbV_{\bbW_n^{(1)}} , \bbV_{\bbW_n^{(2)}} \right) $
}			
4000	&  5.2 & {5.1} & 5.1  & {5.0} & 5.3   
\\

	\multicolumn{1}{c}{} &\multicolumn{5}{c}{unweighted SBM with $p = 1.8n^{ - \frac16 }, q=0.36n^{-\frac16}$} \\			
			
500				&5.0 & {5.0} &4.9& {5.3} &6.1 
\\ \ignore{
	\multirow{3}{*}{1000} & $\left( \bbV_{\bbW_n^{(2)}}^T \bbV_{\bbW_n^{(1)}}\right)^{-1}$ & 5.1 &{5.3} &5.4& {7.0} &8.0 
	\\ 
	& $\bbV_{\bbW_n^{(1)}}^T \bbV_{\bbW_n^{(2)}}$ &5.5 & {5.3}&5.0& {6.5} &7.4 
	\\ & $\Procrustes \left( \bbV_{\bbW_n^{(1)}} , \bbV_{\bbW_n^{(2)}} \right)$ }
1000				& 5.3 & {5.4} &5.7 & {6.2} & 6.5 
\\
\ignore{	\\ \hline \multirow{3}{*}{2000} & $\left( \bbV_{\bbW_n^{(2)}}^T \bbV_{\bbW_n^{(1)}}\right)^{-1}$ & 4.6 & {5.0} &5.5 & {6.0} &6.9 
	\\ & $\bbV_{\bbW_n^{(1)}}^T \bbV_{\bbW_n^{(2)}}$ & 4.7   & {4.8} & 4.3 &{5.9}  & 6.9 
	\\ &$\Procrustes \left( \bbV_{\bbW_n^{(1)}} , \bbV_{\bbW_n^{(2)}} \right) $}
2000				&5.1  &{5.1}&5.3 &{6.0} &6.4 
\\
\ignore{\hline
	\multirow{3}{*}{4000} & $\left( \bbV_{\bbW_n^{(2)}}^T \bbV_{\bbW_n^{(1)}}\right)^{-1}$ & 5.3  &{4.9 }&5.1& {6.0} &6.4 
	\\  & $\bbV_{\bbW_n^{(1)}}^T \bbV_{\bbW_n^{(2)}}$ & 5.1 & {5.0} & 5.1  &  {5.7} &6.3 
	\\  & $\Procrustes \left( \bbV_{\bbW_n^{(1)}} , \bbV_{\bbW_n^{(2)}} \right) $}
4000				& 5.0 & {5.1} &5.4 & {5.7} &5.9 
\\ \hline				
			\end{tabular}
		\end{center}
		
	\end{table}

	\subsection{Empirical power}
	\label{sec:asymptotic power w.r.t. an interpretable alternative}

Table \ref{tab: empirical rejection rates} shows the empirical power for two different models. The first model  assumes that   $p = 0.5, q=0.5-200^{-\frac13} = 0.329$, which corresponds to  a SNR =0.0578.  The second model assumed that   $p = 0.5, q_n=p - n^{-\frac13}(\ge 0.329)$, which gives a SNR=$2n^{-\frac23}$. For each scenario, we fix Hamming distance $\ell_0$ and increase number of nodes $n$. As expected, as the Hamming distance $\ell_0$ between the two community memberships increases, we observe increased power of our proposed test. 
	
	\begin{table}
		\caption{Empirical  power  of the test with two-sided $\alpha =5\%$. Two unweighted SBMs with two blocks of equal sizes and  $\gamma = 1$. }
		\label{tab: empirical rejection rates}
		\begin{center}
			\begin{tabular}{cccccccccc} \hline  
				\multicolumn{1}{c}{} & \multicolumn{9}{c}{$\ell_n \equiv \ell_0$}\\
			$\ell_0$&0.0\%& 0.1\% &0.2\%& 0.4\% & 0.5\% & 1.0\% &1.6\%&  2.0\% & 5.0\% \\ 
			\cline{1-10}
				\multicolumn{10}{c}{$p = 0.5, q=0.5-200^{-\frac13} = 0.329$, $SNR =  0.0578$}\\ 
				200&5.7 & --& -- & -- &17.0 & 43.1& 70.7&87.0 & 100.0  \\ 
				500& 4.7 &--& 27.3 &  72.5 & 73.5 &  93.3 &  99.8 & 100.0 & 100.0  \\ 
				1000 &4.8 & 49.1 & 95.5 &  100.0& 100.0& 100.0& 100.0& 100.0& 100.0\\ 
				2000		& 5.3&99.8 & 100.0 &  100.0)& 100.0& 100.0& 100.0& 100.0& 100.0\\ 
				\multicolumn{10}{c}{$p = 0.5, q_n=p - n^{-\frac13}(\ge 0.329)$,  $SNR = 2n^{-\frac23}$}\\ 
				200&5.7 & --& -- & -- &17.0 & 43.1& 70.7&87.0 & 100.0  \\ 
				500& 5.1 & --& 13.6 &31.1 & 32.9&92.9&97.1 &100.0&100.0 \\ 
				1000 & 5.1 & 14.2 & 25.9 & 70.2 & 86.6 &100.0 & 100.0&100.0&100.0   \\		
				2000		&5.0& 20.4&60.8 & 99.2&100.0&100.0&100.0& 100.0&100.0\\ 
				\hline
			\end{tabular}
		\end{center}
	\end{table}
	%
	%
	%

	\section{Real Data Example  -- Enron Email dataset}\label{data}

	\begin{table}
		\caption{Enron dataset: persons' names and their positions for selected nodes number.}
		\label{tab:Enron node number correspondence}
		\begin{tabular}{ccc} \hline
			node number & name&  position \\ \hline
4& badeer-r& Director  \\  
16& causholli-m& Employee  \\  
18& cuilla-m& Manager \\  
32& forney-j& Manager  \\ 
34& gang-l& Employee  \\  
84& motley-m& Director  \\  
93& presto-k&  Vice  President   \\  
94 & quenet-j & Trader \\ 
107& salisbury-h& Employee  \\  
110& scholtes-d& Trader  \\ 
114& schwieger-j& Trader  \\  
120& slinger-r& Trader  \\  
122& solberg-g& Employee  \\ 
127& stepenovitch-j&  Vice  President   \\  
128& stokley-c& Employee \\  
132& tholt-j& Vice President  \\ 
138& ward-k& Employee  \\  
\hline

		\end{tabular}
	\end{table}
To demonstrate the proposed test, we analyzed the Enron email network data   (May 7th, 2015  version,  \href{https://www.cs.cmu.edu/~enron/}{https://www.cs.cmu.edu/$\sim$enron/}). The dataset includes email communication data of  150 users, mostly were in senior management positions,  including CEO (4), manager (8), trader (2), president (2), vice president (16), others (57). 
For each email, we have information on  sender, list of recipients and  the email date.  The email links were included  as long as they were sent to some of the 89 users. To construct the weights,  if A sent an email to B and C, both weights for edge (A,B) and edge (A,C) was increased by 1.  
Since the original Enron email network  were directed, we converted it into undirected network by setting the weight  $w_{new}^{(v)}(A,B)\leftarrow \min \left\{w_{old}^{(v)}(A,B) + w_{old}^{(v)}(B,C),  127 \right\}$.

There were a total of  11539 emails communications (without self-loops) among the 150 users    between 1998 and 2001, represented by  a directed graph with maximal weight $\max_{A,B} w_{old}^{all}(A,B) = 361$. We performed spectral clustering analysis  based on the Laplacian of the weight matrix  $\cL \left(\bbW_n^{(v)}\right)$ and applied k-means clustering methods.	Similar to  \cite{xu2013dynamic}, we set number of clusters $K=2$.

	\subsection{Comparing email networks before and after August 2000}
	We first compared the email networks before and after August 2000, where 
7534 emails and 4005 emails were observed, respectively.  Our test statistic \eqref{eq:test statistic} did not reject the the null $H_0$ \eqref{eq:hypothesis test for SBM AJ}, indicating no significant difference of the community memberships among the users before and after August 2000.

Figure \ref{fig:Enron dataset} shows  a visualization of the two email networks  with coordinates generated by Fruchterman-Reingold force-directed algorithm. 
	The weight matrix $\bbW_n^{(1)}$ from 1998 to Aug. 2000 results in two clusters with sizes 10 and 140. The  smaller cluster has nodes $[4,16,18, 34, 84, 107, 110, 114, 128,132]$\ignore{; nodes $[4, 16, 34, 107,  120]$ randomly appear in this cluster (since the clustering algorithm itself have randomness)}.
		Similarly, 
weight matrix $\bbW_n^{(2)}$ from Sept. 2000 to 2001 also resulted in two clusters with size $11$ and $139$, where the smaller cluster has nodes $[4, 16, 18, 34, 84, 93,  110, 114, 120,128,132]$\ignore{;  nodes [4, 16, 34, 107,  120] randomly appear in this cluster (since the clustering algorithm itself have randomness)}.
Nodes $[4,16,18,34, 84, 110, 114, 128,132]$ appeared in both small clusters. They include traders $[110, 114]$,  Manager [18],  Director [84] and a  Vice President [128] (see table \ref{tab:Enron node number correspondence}).
\ignore{The relative positions for nodes [4, 16, 34, 107, 120] in both networks are similar. }

	\begin{figure}[ht]
		\caption{Estimated Enron email networks before and after August 2000.}
		\label{fig:Enron dataset}
		\begin{center}
			\begin{tabular}{cc}
			\includegraphics[height=0.5\textheight, width=0.49\textwidth]{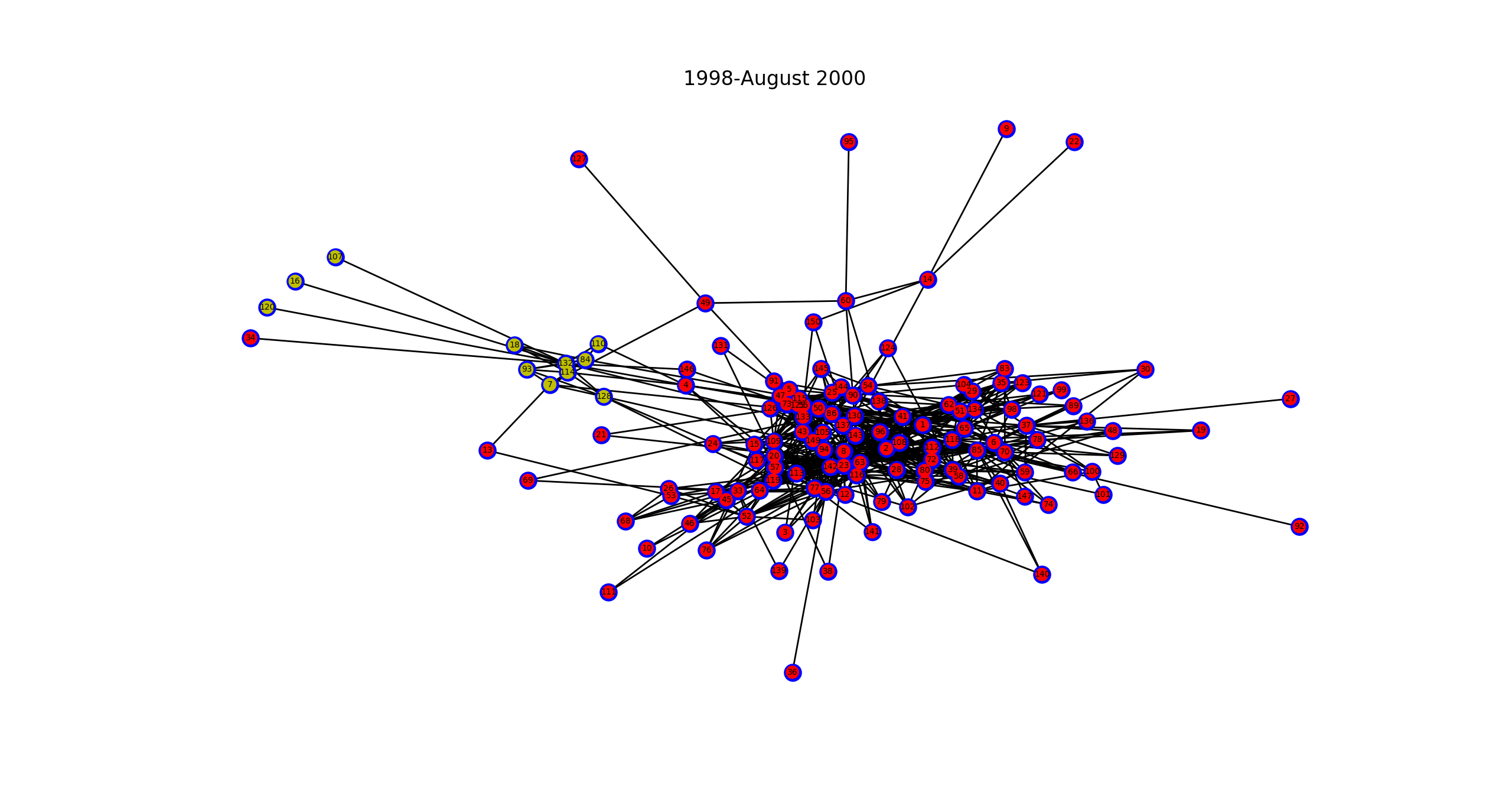} &
			\includegraphics[height=0.5\textheight, width=0.49\textwidth]{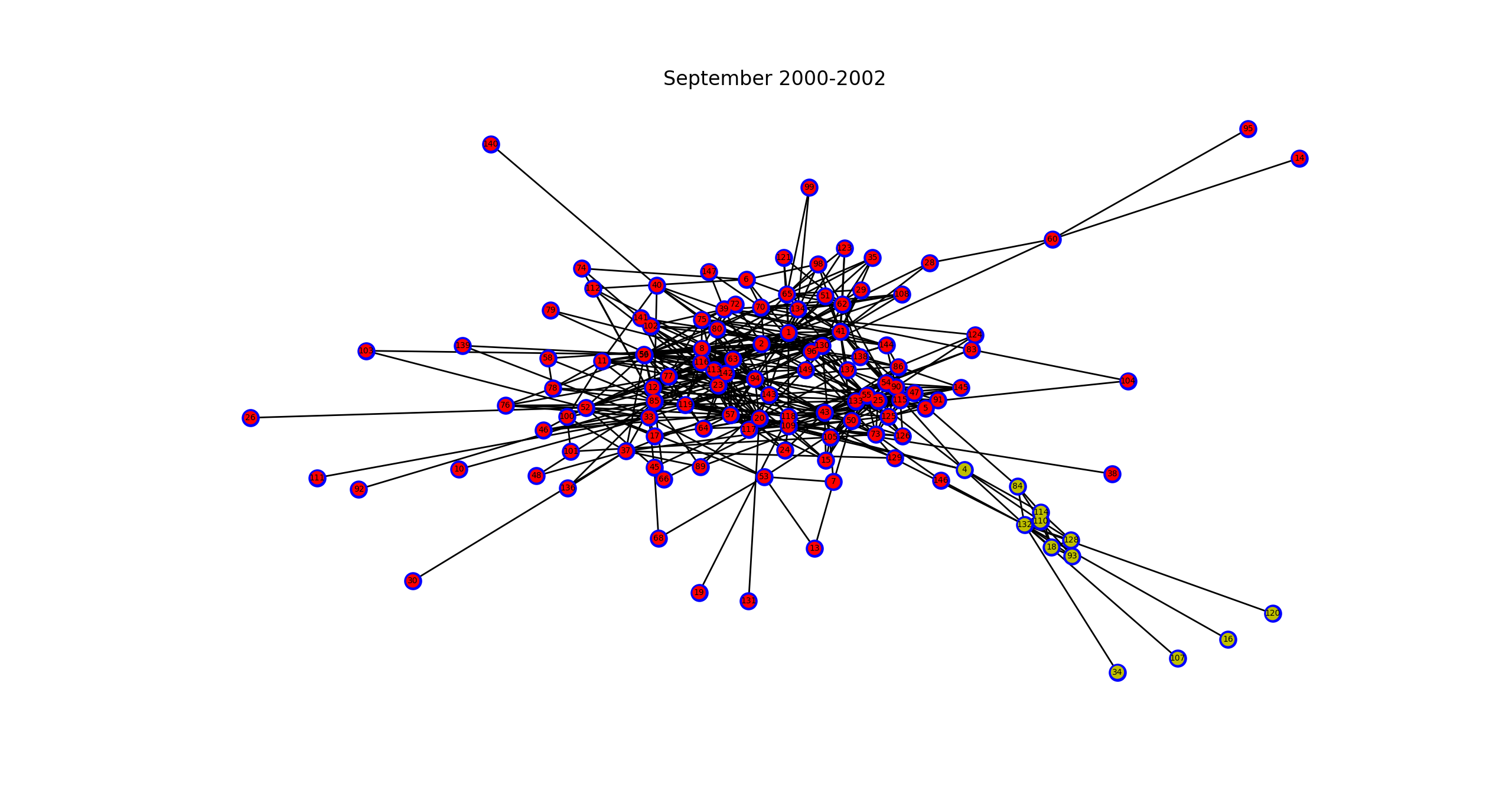}
			\end{tabular}
		\end{center}
		
	\end{figure}
	
	\subsection{Comparing email networks before and after December  2001}		
		We then compared the email networks before and after December 2001, where 
	7713 emails and 3826 emails were observed, respectively.  Our test statistic \eqref{eq:test statistic} rejected the null, indicating that the community memberships were different before and after December 2001.  The date was chosen  since CEO Jeffrey Skilling resigned on  Aug. 14th, 2001 and 
	the number of emails sent by week  revealed  peaks in email activity around Nov. 9th 2001 and  end of Dec 2001.

	\ignore{	A large increase in the probabilities of edges from CEO is found at week 89, in which CEO Jeffrey Skilling resigned on  Aug. 14th, 2001. 
	The number of emails sent by week  reveals peaks in email activity around Nov. 9th 2001, end of Dec 2001 but not around CEO Skilling’s resignation.}

Figure \ref{fig:Enron dataset2}  shows  the two estimated  email networks  with coordinates generated by Fruchterman-Reingold force-directed algorithm.  
	We observed that weight matrix $\bbW_n^{(1)}$ from 1998 to 2000 resulted in a smaller cluster with nodes $[4,16,18,34, 93,107,110,114,128,132]$\ignore{; nodes $[4, 16, 34, 107, 120]$ randomly appear in this cluster (since the clustering algorithm itself have randomness)}.
In contrast, weight matrix $ \bbW_n^{(2)}$ in 2001 results in two clusters: the smaller cluster has nodes \\$[3, 32, 93, 110, 114,121, 122, 132, 138]$\ignore{; some random nodes $[3, 121]$ appear}.
 Nodes [93, 110, 132]  were shared between the two smaller communities, which includes 2 Vice Presidents [93, 132]  and a  Trader [110] (see Table \ref{tab:Enron node number correspondence}).

%

	\begin{figure}[ht]
	\caption{Estimated  Enron email networks before and after December  2001.}
	\label{fig:Enron dataset2}
	\begin{center}
			\begin{tabular}{cc}
		\includegraphics[height=0.5\textheight, width=0.49\textwidth]{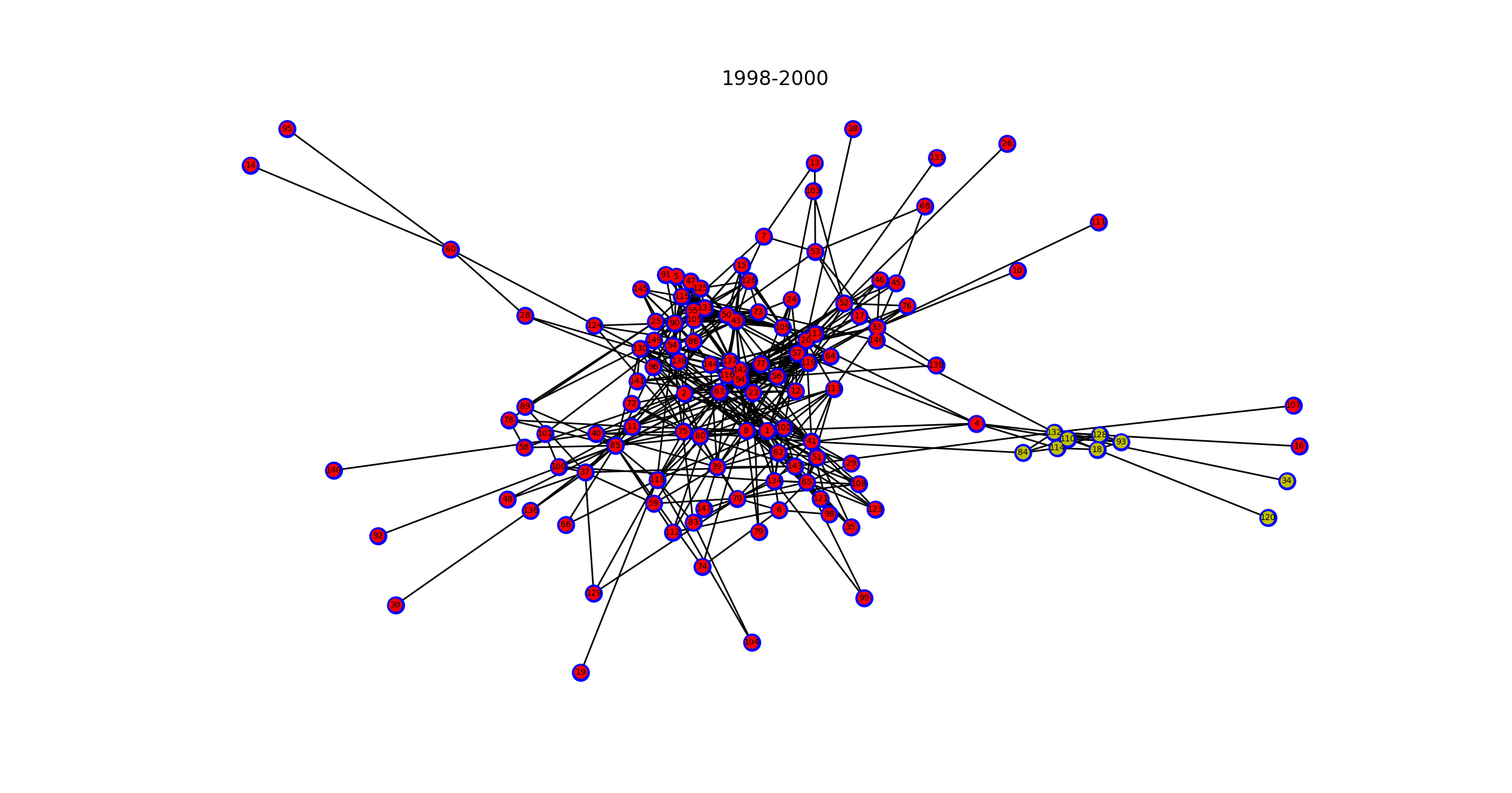}
		\includegraphics[height=0.5\textheight, width=0.49\textwidth]{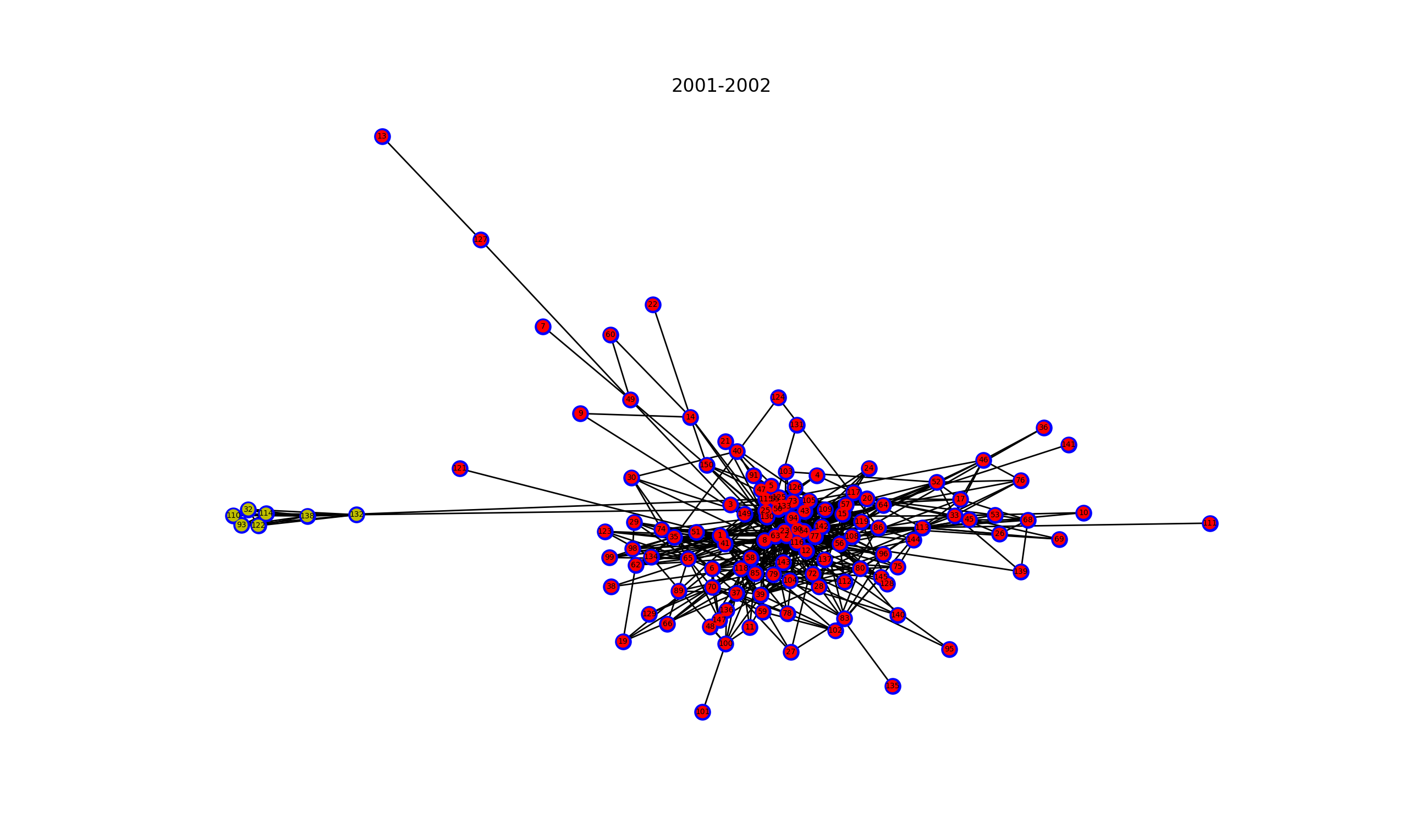}

	\end{tabular}	\end{center}
\end{figure}

	\section{Discussions} \label{sec:discuss}

	We have developed a statistical test for equivalent community memberships based on stochastic block models and derived its asymptotic null distribution under the dense graph assumption 
	(Assumption \ref{assum:region of interest Renyi divergence}).  This assumption is needed to obtain the  dominant representation for the subspace distance. 	 
		In order to detect the community structures,  we also require that the intra- and inter-cluster probability distributions are not  too close. While this assumption is reasonable,  it would be  interesting to further investigate the case when the intra- and inter- distributions $\cP_n,\cQ_n $ are close to each other as $n \to \infty$. 
	Like \cite{tang2017semiparametric}, we also assume that the distributions that generate the two weighted networks only differ by a scalar.  For the case when we have two communities for each network, our test statistic has correct type I errors and large power in detecting the difference in community memberships. 	When $\NumOfB >2$, estimation of the community memberships becomes more difficult, which can lead to  slower convergence rates (see Table \ref{tab:type I error unweighted SBM Kn>2} in Supplemental Materials), although  type I errors are still approximately under control.

The test procedure we developed is based on  community recovery from the observed weighted adjacency matrices \citep{bickel2016hypothesis,jog2015information,lei2015consistency,lei2016goodness,xu2017optimal}. Alternatively, one can also apply  spectral clustering method based on singular components of normalized Laplacian \citep{rohe2011spectral,sarkar2015role} of the corresponding network graphs. 	It is also interesting to consider  kernelized spectral clustering of samples from a finite mixture of nonparametric distributions \citep{schiebinger2015geometry}.
As a future research topic, it is interesting to investigate whether the asymptotic results still hold when these alternative clustering methods are applied. 

		Assumptions \ref{assum:homogeneous weighted SBM} and  \ref{assum:underlying probability the same up to a scalar} and including $\DiaEig_{\bbW_n^{(2)}}$ in the proposed test statistic \eqref{eq:test statistic} are all imposed to simplify the mean and variance of test statistic \eqref{eq:test statistic}.  If we  impose the ``equal-size" assumption (see section \ref{subsec:asmptotic power}),  where  $\frac{n}{\NumOfB}$ is assumed to be an integer and $\beta =1 $ in Assumption \ref{assum:size of each block} \citep{banerjee2018contiguity,banerjee2017optimal}, we may  relax   Assumptions \ref{assum:homogeneous weighted SBM} and  \ref{assum:underlying probability the same up to a scalar} and eliminate the  adjustment of multiplying $\DiaEig_{\bbW_n^{(2)}}$ in test statistic \eqref{eq:test statistic}. This assumption  may also possibly relax the requirement that the two  distributions that generate the  weighted networks differ only  by a scalar.

	%


\appendix

\section{Notations}

We clarify some notations to facilitate readers' understanding of statements and proofs of the theorems.

Besides Frobenius norm, we also use  2-norm in the proofs, which is a special case of induced norms \eqref{defn: induced norms}:
\begin{equation}
\|\bbM\|_2 \triangleq \sup_{\|\x\|_2 =1 ,\x \in \Reals^{d} } \|\bbM \x\|_2 = \sqrt{\rho \left(\bbM^T\bbM \right) }, \bbM \in \Reals^{m \times d}.
\label{eq:2-2 norm}
\end{equation}
where $\rho:\Reals^{m\times m} \to \Reals^1$ refers to spectral radius of a square matrix: \begin{equation} \rho(\bbM) \triangleq \max \left\{ \left| \lambda_{1} \left( \bbM \right)\right| , \ldots, \left| \lambda_{m} \left( \bbM \right)\right|  \right\}= \lim_{ k \to \infty} \| \bbM^k \|^{\frac1k},
\label{eq:spec radius}
\end{equation}
for any consistent matrix norm $\|\cdot \|$.  In general,  $\rho(\bbM) \le \| \bbM \|_2$, for symmetric matrix $\bbM = \bbM^T \in \Reals^{m \times m}$, $\| \bbM \|_2 = \rho(\bbM)$. 

When writing $O_P(\cdot ), O(\cdot ), o_P(\cdot ),o(\cdot )$ with respect to a $m \times d $ matrix, we generally refers (if no any other special instructions) to the order with respect to its Frobenius norm. $f(n) = \Theta(n^{- \alpha})$ or $f(n) \asymp n^{ - \alpha}$ mean that there exists constant $C>1$ such that $\frac{n^{- \alpha}}{C} \le f(n) \le C n^{- \alpha}$. The equivalent symbols  used in this work are summarized in Table \ref{tab:summary equivalent symbols}.

\begin{table}
	\caption{Equivalent symbols used in different scenarios.}
	\label{tab:summary equivalent symbols}
	\begin{tabular}{cccccc} \hline
		$O_P(\cdot)$ & $O(\cdot)$ & $o_P(\cdot)$ & $o(\cdot)$ &  $\Theta_P(n^{ - \alpha})$ & $\Theta(n^{ - \alpha})$ \\ \hline
		$\preceq_P$ &$\preceq $ & -- & -- & $\asymp_P$ & $\asymp$\\  \hline 
	\end{tabular}
\end{table}

\section{Proof sketch for the asymptotic expansions in Theorem 2 and Theorem 4}
\label{sec:proof sketches of Minh's dominant term}
We present in this section a sketch of the main ideas in the proofs of Theorem \ref{thm:asymptotic expansion one view} and Theorem \ref{thm:asymptotic expansion two views}.

Compared to the proof of Lemma \ref{lem:tang asymptotic expansion} that uses the same technique as  in \cite{tang2017semiparametric,tang2018limit},
it suffices to provide upper bounds for the difference of the squares of Frobenous norms. 

For the case of one-network in Theorem \ref{thm:asymptotic expansion one view}, we obtain the following bound
\begin{equation}
\frac{ \left\| \left(\bbV_{\bbW_n}\bbT_n- \bbV_{\bbE_n} \right) \DiaEig_{ \bbE_n} \right\|_F^2  }{\NumOfB n}  - \frac{ \left\|   \left( \bbW_n - \bbE_n \right) \bbV_{\bbE_n}   \right\|_F^2 }{\NumOfB n} = O_P \left( \frac{\NumOfB^{2} [\log(n)]^2 }{n} \right), \label{eq:upper bound one view}
\end{equation} 
 where $\bbT_n = \Procrustes \left(  \bbV_{\bbW_n}, \bbV_{\bbE_n}  \right), \bbT_n = \bbV_{\bbW_n}^T\bbV_{\bbE_n}$ or $\left(\bbV_{\bbE_n}^T\bbV_{\bbW_n}\right)^{-1}$. We consider three different $\bbT_n$s for two reasons. First, the proof of \eqref{eq:upper bound one view} can be simplified by focusing on $\bbT_n =\left(\bbV_{\bbE_n}^T\bbV_{\bbW_n}\right)^{-1}$. Second,  result for \eqref{eq:upper bound one view} with $\bbT_n = \bbV_{\bbW_n}^T\bbV_{\bbE_n}$ has a direct Corollary \ref{cor:term hard to handle by second singular value of Erdos Renyi graph} that can simplify the proof of Theorem \ref{thm:asymptotic expansion two views} since it  simplifies the proof for \eqref{eq:upper bound two views}.
 
For the two-sample case in Theorem \ref{thm:asymptotic expansion two views}, under the null hypothesis \eqref{eq:hypothesis test for SBM AJ}, we have the bound for  the  difference of the squares of Frobenous norms  
\begin{eqnarray}
&&	\frac1{ \NumOfB n}\left\| \left(  \bbV_{\bbW_n^{(1)}} \bbT_n - \bbV_{\bbW_n^{(2)} } \right) \DiaEig_{ \bbW_n^{(2)}}\right\|_F^2 
-   \frac1{\NumOfB n} \left\| \left( \gamma \bbW_n^{(1)} - \bbW_n^{(2)}  \right) \bbV_{\bbE_n^{(2)} }  \right\|_F^2 \label{eq:upper bound two views}\\
&= & O_P \left( \frac{\NumOfB^{2} [\log(n)]^2 }{n} \right), \nonumber 
\end{eqnarray}	
where $$\bbT_n = \Procrustes \left( \bbV_{\bbW_n^{(1)}} , \bbV_{\bbW_n^{(2)}} \right), \bbV_{\bbW_n^{(1)}}^T\bbV_{\bbW_n^{(2)}}, \mbox{ or } \left(\bbV_{\bbW_n^{(2)}}^T\bbV_{\bbW_n^{(1)}}\right)^{-1}.$$

To conclude, these two bounds are used to prove asymptotic expansions in Theorem \ref{thm:asymptotic expansion one view} and Theorem \ref{thm:asymptotic expansion two views}, respectively.

 \ignore{For simplicity, we ignore the subscript $n$ in matrices $\bbW_n, \bbE_n, \bbW_n^{(1)},\bbW_n^{(2)}$.}

\subsection{Proof sketch for the asymptotic expansion in Theorem \ref{thm:asymptotic expansion one view}}
 To prove  \eqref{eq:upper bound one view}, we first focus on proving the result  for $\bbT_n = \left(\bbV_{\bbW_n}^T\bbV_{\bbE_n}\right)^{-1}$ in Section \ref{subsec:bbT VE VW inverse}.  
 Its  proof is briefly  sketched as the following:
 
 {\allowdisplaybreaks
 	\begin{eqnarray*}
 		&&\frac1{\NumOfB n} \left\| \left(\bbV_{\bbW_n} \left( \bbV_{ \bbE_n}^T \bbV_{\bbW_n} \right)^{-1 }  - \bbV_{\bbE_n} \right) \DiaEig_{ \bbE_n} \right\|_F^2
 		\\
 		&\xlongequal{\text{\eqref{eq: part using the fact that E = VEVET E}}} &  \frac1{\NumOfB n} \left\| \left( \bbI_n   - \bbV_{\bbE_n} \bbV_{ \bbE_n}^T \right) \left( \bbW_n - \bbE_n \right) \bbV_{\bbW_n}  \left( \bbV_{ \bbE_n}^T \bbV_{\bbW_n} \right)^{-1 }  \right\|_F^2
 		 + O_P \left( \frac{\NumOfB^2\log (n)}{n^2b_{\cP}^{\frac32}} \right)  \nonumber \\
 		& = & \frac1{\NumOfB n} \left\|   \left( \bbW_n - \bbE_n \right) \bbV_{\bbW_n}  \left( \bbV_{ \bbE_n}^T \bbV_{\bbW_n} \right)^{-1 }  \right\|_F^2  \nonumber
 		\\
 		&&  - \frac1{\NumOfB n} \left\|  \bbV_{ \bbE}^T \left( \bbW_n - \bbE_n \right) \bbV_{\bbW_n}  \left( \bbV_{ \bbE_n}^T \bbV_{\bbW_n} \right)^{-1 }  \right\|_F^2 +O_P\left( \frac{\NumOfB^2\log (n)}{n^2b_{\cP}^{\frac32}} \right) \nonumber\\
 		& \xlongequal{\text{\eqref{eq:term Wigner order}}} &  \frac1{\NumOfB n} \left\|   \left( \bbW_n - \bbE_n \right) \bbV_{\bbW_n}  \left( \bbV_{ \bbE_n}^T \bbV_{\bbW_n} \right)^{-1 }  \right\|_F^2 + O_P \left( \frac{\NumOfB [\log (n) ]^2  }n\right) \nonumber\\
 		& \xlongequal{\text{\eqref{eq:my B8 AJ}}} &   \frac1{\NumOfB n} \left\|   \left( \bbW_n - \bbE_n \right) \bbV_{\bbE_n}   \right\|_F^2 + O_P \left( \frac{\NumOfB [\log (n)]^2    }{n} \right)  \nonumber.
 	\end{eqnarray*}
 }
 
For $\bbT_n = \Procrustes \left( \bbV_{\bbW_n} , \bbV_{\bbE_n} \right)$,  we need to consider the difference of two squares of Frobenius norm,
$$
\frac1{ \NumOfB n}\left\| \left(\bbV_{\bbW_n} \bbV_{ \bbW_n}^T \bbV_{ \bbE_n}- \bbV_{\bbE_n} \right) \DiaEig_{ \bbE_n}\right\|_F^2 -   \frac1{\NumOfB n} \left\| \left(\bbV_{\bbW_n} \tilde \bbT_n - \bbV_{\bbE_n} \right) \DiaEig_{ \bbE_n}  \right\|_F^2 ,
$$
 that is, left-hand side of \eqref{eq:upper bound one view} with $\bbT_n =  \bbV_{ \bbW_n}^T \bbV_{ \bbE_n}$ and $\bbT_n = \tilde \bbT_n$ where \eqref{eq:requirement for tilde Tn} holds:
$$\left\| \bbV_{ \bbW_n}^T \bbV_{ \bbE_n} - \tilde \bbT_n\right\|_F = O_P\left(\frac{\NumOfB^{2 }}{ nb_{\cP}}\log (n) \right).
$$
The difference can be upper-bounded as \eqref{eq:differences of two frobenius norm equivalence of 3T's}
 {\allowdisplaybreaks
\begin{eqnarray*}
	&& \frac1{ \NumOfB n}\left\| \left(\bbV_{\bbW_n} \bbV_{ \bbW_n}^T \bbV_{ \bbE_n}- \bbV_{\bbE_n} \right) \DiaEig_{ \bbE_n}\right\|_F^2 -   \frac1{\NumOfB n} \left\| \left(\bbV_{\bbW_n} \tilde \bbT_n - \bbV_{\bbE_n} \right) \DiaEig_{ \bbE_n}  \right\|_F^2 \nonumber \\
	& \xlongequal{\text{\eqref{eq:requirement for tilde Tn}}} &  \frac2{\NumOfB n} \left\langle  \bbV_{\bbW_n} \left(  \bbV_{\bbW_n}^T \bbV_{ \bbE_n} - \tilde  \bbT_n  \right) \DiaEig_{ \bbE_n}   ,  \left(\bbV_{\bbW_n}  \bbV_{\bbW_n}^T \bbV_{ \bbE_n} - \bbV_{\bbE_n} \right) \DiaEig_{ \bbE_n}  \right\rangle \nonumber \\
	&& + O_P \left(  \frac{\NumOfB [\log (n)]^2}n \right) \nonumber \\
	& &\xlongequal{\bbV_{\bbW_n}^T \bbV_{\bbW_n}  \bbV_{\bbW_n}^T= \bbV_{\bbW_n}^T }  0+ O_P \left(   \frac1n  \right) = O_P \left(   \frac{\NumOfB [\log (n)]^2}n  \right).  
\end{eqnarray*}

}
With these results and the proof of  \eqref{eq:upper bound one view} for $ \left(\bbV_{\bbW_n}^T\bbV_{\bbE_n}\right)^{-1} $, we take $\tilde \bbT_n = \left(\bbV_{\bbW_n}^T\bbV_{\bbE_n}\right)^{-1}$ in \eqref{eq:differences of two frobenius norm equivalence of 3T's} and we prove that  \eqref{eq:upper bound one view} holds for $ \bbT_n=\bbV_{\bbE_n}^T\bbV_{\bbW_n}$. With this result,  we take $ \bbT_n = \Procrustes \left( \bbV_{\bbW_n} , \bbV_{\bbE_n} \right)$ in \eqref{eq:differences of two frobenius norm equivalence of 3T's} and  we prove that \eqref{eq:upper bound one view} holds for $ \bbT_n = \Procrustes \left( \bbV_{\bbW_n} , \bbV_{\bbE_n} \right)$.
\subsection{Proof sketch for  asymptotic expansion in Theorem \ref{thm:asymptotic expansion two views}}

For the two-sample results stated in Theorem \ref{thm:asymptotic expansion two views}, it is worth mentioning that an essential difficulty that makes the two-sample problem more difficult than the problem with one network  is that  $\bbW_n^{(v)} \ne \bbV_{\bbW_n^{(v)} } \bbV_{\bbW_n^{(v)} }^T\bbW_n^{(v)}$.  However,  equality $\bbE_n = \bbV_{\bbE_n} \bbV_{\bbE_n}^T\bbE_n$ is used  in the above proof sketch for \eqref{eq:upper bound one view}  in step \eqref{eq: part using the fact that E = VEVET E}. In contrast, for the two-sample problem,  we do not have such an equality. 

To prove  \eqref{eq:upper bound two views},  we first consider  $\bbT
_n = \bbV_{\bbW_n^{(1)}}^T\bbV_{\bbW_n^{(2)}}$. With details given in Section \ref{subsec:proof theorem asymptotic expansion two views}, steps at the beginning are sketched as the following:
{\allowdisplaybreaks
\begin{eqnarray*}
	&& \frac{ \left\| \left( \bbV_{\bbW_n^{(1)}} \bbV_{ \bbW_n^{(1)}}^T\bbV_{\bbW_n^{(2)}} -  \bbV_{\bbW_n^{(2)}}\right)  \DiaEig_{\bbW_n^{(2)} } \right\|_F^2}{ \NumOfB n} -   \frac{\left\| \left(\gamma \bbW^{(1)} - \bbW_n^{(2)}  \right) \bbV_{\bbW_n^{(2)} }  \right\|_F^2 }{\NumOfB n} \nonumber \\
	&=& \frac{\left\| \left( \bbV_{\bbW_n^{(1)}} \bbV_{ \bbW_n^{(1)}}^T -  \bbI_n    \right) \bbW_n^{(2)} \bbV_{\bbW_n^{(2)} } \right\|_F^2 }{ \NumOfB n}-   \frac{ \left\| \left(\gamma \bbW_n^{(1)} - \bbW_n^{(2)}  \right) \bbV_{\bbW_n^{(2)} }  \right\|_F^2 }{\NumOfB n}\nonumber \\
	& = &  \frac1{\NumOfB n} \left\langle \left( \bbV_{\bbW_n^{(1)} } \bbV_{ \bbW_n^{(1)} }^T -  \bbI_n    \right) \bbW^{(2)} \bbV_{\bbW_n^{(2)} } , \right. \nonumber \\
	&& \left.  \left( \bbV_{\bbW_n^{(1)} } \bbV_{ \bbW_n^{(1)} }^T -  \bbI_n    \right) \bbW_n^{(2)} \bbV_{\bbW_n^{(2)}} -  \left( \gamma \bbW_n^{(1)} - \bbW_n^{(2)} \right) \bbV_{\bbW_n^{(2)}}  \right\rangle \nonumber \\
	&& + O_P \left(  \frac{ \NumOfB [\log (n)]^2 }{ n } \right)  \nonumber \\
	& = &  \frac1{\NumOfB n} \left\langle \left( \bbV_{\bbW_n^{(1)} } \bbV_{ \bbW_n^{(1)} }^T -  \bbI_n    \right) \bbW_n^{(2)} \bbV_{\bbW_n^{(2)}} ,   \left( \bbV_{\bbW_n^{(1)} } \bbV_{ \bbW_n^{(1)} }^T\bbW_n^{(2)} - \gamma  \bbW_n^{(1)}   \right)  \bbV_{\bbW_n^{(2)}} \right\rangle  \nonumber \\
	&& + O_P \left(   \frac{ \NumOfB [\log (n)]^2  }{ n }  \right) \nonumber \\
& = &  - \frac{\gamma}{\NumOfB n}\tr \left[   \bbV_{\bbW_n^{(2)}}^T \bbW_n^{(1)}   \left( \bbV_{\bbW_n^{(1)}} \bbV_{ \bbW_n^{(1)} }^T -  \bbI_n    \right) \bbW_n^{(2)} \bbV_{\bbW_n^{(2)}}  \right]  + O_P \left(  \frac{ \NumOfB [\log (n)]^2 }{ n } \right)  \nonumber \\
	& = &  - \frac{\gamma}{\NumOfB n}\tr \left[  \bbV_{\bbW_n^{(2)}}^T \bbV_{\bbW_n^{(1)}}^{\perp }  \DiaEig_{\bbW_n^{(1)}}^{\perp }   \left[ \bbV_{ \bbW_n^{(1)}}^{\perp } \right]^T  \bbV_{\bbW_n^{(2)}}  \DiaEig_{\bbW_n^{(2)}}  \right]  +O_P\left( \frac{ \NumOfB[\log (n)]^2 }{ n } \right) \nonumber.
\end{eqnarray*}
}

To continue proving \eqref{eq:upper bound two views} with $\bbT
_n = \bbV_{\bbW_n^{(1)}}^T\bbV_{\bbW_n^{(2)}}$,  we  need to show that 
$$\frac{\gamma}{\NumOfB n}\tr \left[  \bbV_{\bbW_n^{(2)}}^T \bbV_{\bbW_n^{(1)}}^{\perp }  \DiaEig_{\bbW_n^{(1)}}^{\perp }   \left[ \bbV_{ \bbW_n^{(1)}}^{\perp } \right]^T  \bbV_{\bbW_n^{(2)}}  \DiaEig_{\bbW_n^{(2)}}  \right]  =O_P\left( \frac{ \NumOfB[\log (n)]^2 }{ n } \right),$$  and 
the same result holds when   $$\frac{\left\| \left(\gamma \bbW_n^{(1)} - \bbW_n^{(2)}  \right) \bbV_{\bbW_n^{(2)} }  \right\|_F^2}{\NumOfB n}$$ is replaced with 
$$- \frac{\left\| \left(\gamma \bbW_n^{(1)} - \bbW_n^{(2)}  \right) \bbV_{\bbE^{(2)} }  \right\|_F^2}{\NumOfB n}. $$
The proofs of these two steps are supported  by Corollary \ref{cor:term hard to handle by second singular value of Erdos Renyi graph} and  Corollary \ref{cor:bbV bbW also OK rather than bbV bbP} in Section \ref{subsubsec:Two useful corollaries for two-sample problem}. These two corollaries are essentially derived from the proof for the problem  of one network in Section \ref{appendix:proof of one view dominant term}. 

We  outline the proof of  Corollary \ref{cor:term hard to handle by second singular value of Erdos Renyi graph}  to demonstrate how we  overcome the essential difficulty mentioned above.  This Corollary  is derived from the result for  $\bbT_n = \bbV_{\bbE_n}^T\bbV_{\bbW_n}$ in \eqref{eq:upper bound one view} rather than a direct corollary of the discussion in Section \ref{subsec:bbT VE VW inverse} of $\bbT = \left(\bbV_{\bbW_n}^T\bbV_{\bbE_n}\right)^{-1}$ in \eqref{eq:upper bound one view}. Corollary \ref{cor:term hard to handle by second singular value of Erdos Renyi graph} states that \eqref{eq:cor:term hard to handle by second singular value of Erdos Renyi graph} holds, that is 
$$\tr \left[  \bbV_{\bbE_n}^T \bbV_{\bbW_n}^{\perp }  \DiaEig_{\bbW_n}^{\perp }   \left[ \bbV_{ \bbW_n}^{\perp } \right]^T  \bbV_{\bbE_n}   \right]	= O_P\left( \NumOfB^2 [\log (n)]^2 \right),
$$
which implies 
{ \allowdisplaybreaks
	\begin{eqnarray*}
		&&		O_P\left( \NumOfB \log(n)\right)\\
		&  = & \sqrt{\tr \left[  \bbV_{\bbE_n^{(1)}}^T \bbV_{\bbW_n^{(1)}  }^{\perp }  \DiaEig_{\bbW_n^{(1)}}^{\perp }   \left[ \bbV_{ \bbW_n^{(1)}}^{\perp } \right]^T  \bbV_{\bbE_n^{(1)}}   \right] }   \\
		& = & \left\| \left[ \DiaEig_{\bbW_n^{(1)}}^{\perp }\right]^{\frac12}   \left[ \bbV_{ \bbW^{(1)}}^{\perp } \right]^T  \bbV_{\bbE_n^{(1)}}   \right\|_F  = \left\| \left[ \DiaEig_{\bbW_n^{(1)}}^{\perp }\right]^{\frac12}   \left[ \bbV_{ \bbW_n^{(1)}}^{\perp } \right]^T  \bbV_{\bbE_n^{(2)}}   \right\|_F \\
		& = &   \left\| \left[ \DiaEig_{\bbW_n^{(1)}}^{\perp }\right]^{\frac12}   \left[ \bbV_{ \bbW_n^{(1)}}^{\perp } \right]^T \cdot \right.
		\left. \left[\bbV_{\bbE_n^{(2)}} - \bbV_{\bbW_n^{(2)} } \Procrustes  \left(   \bbV_{\bbW_n^{(2)} } , \bbV_{\bbE_n^{(2)}}  \right)+ \bbV_{\bbW_n^{(2)} } \Procrustes  \left( \bbV_{\bbW_n^{(2)} } , \bbV_{\bbE_n^{(2)}} \right) \right]\right\|_F \\
		& = &  \left\| \left[ \DiaEig_{\bbW_n^{(1)}}^{\perp }\right]^{\frac12}   \left[ \bbV_{ \bbW_n^{(1)}}^{\perp } \right]^T   \bbV_{\bbW_n^{(2)} } \Procrustes \left(  \bbV_{\bbW_n^{(2)} } , \bbV_{\bbE_n^{(2)}} \right) \right\|_F + O_P\left(  \sqrt{\frac{ \NumOfB }{ n b_{\cP} }} \cdot \sqrt{\frac{n b_{\cP}}{\NumOfB}} \right) \\
		& = &  \left\| \left[ \DiaEig_{\bbW_n^{(1)}}^{\perp }\right]^{\frac12}   \left[ \bbV_{ \bbW_n^{(1)}}^{\perp } \right]^T   \bbV_{\bbW_n^{(2)} }  \right\|_F + O_P\left( \sqrt{  \NumOfB b_{\cP} } \right) ,
	\end{eqnarray*}
}
where $ \left[ \DiaEig_{\bbW_n^{(1)}}^{\perp }\right]^{\frac12}  \triangleq \diag \left\{ \left| \sigma_{\NumOfB+1} \left( \bbW_n^{(1)}\right) \right|^{\frac12} , \ldots,  \left| \sigma_{ n } \left( \bbW_n^{(1)}\right) \right|^{\frac12} \right\}$. This  equivalently implies  that 
\begin{eqnarray*}
&& - \frac{\gamma}{\NumOfB n}\tr \left[  \bbV_{\bbW_n^{(2)}}^T \bbV_{\bbW_n^{(1)}}^{\perp }  \DiaEig_{\bbW_n^{(1)}}^{\perp }   \left[ \bbV_{ \bbW_n^{(1)}}^{\perp } \right]^T  \bbV_{\bbW_n^{(2)}}  \DiaEig_{\bbW_n^{(2)}}  \right] \\
& = &  - \frac{\gamma}{\NumOfB n}  \left\| \left[ \DiaEig_{\bbW_n^{(1)}}^{\perp }\right]^{\frac12}   \left[ \bbV_{ \bbW_n^{(1)}}^{\perp } \right]^T   \bbV_{\bbW_n^{(2)} }  \right\|_F^2= O_P \left(  \frac{\NumOfB[\log (n)]^2}n\right).
\end{eqnarray*}
With this result,  we arrives at Lemma \ref{lem:just one step away from final result two sample}:
	$$ \frac{\left\| \left( \bbV_{\bbW_n^{(1)}} \bbT -  \bbV_{\bbW_n^{(2)} }  \right)  \DiaEig_{\bbW_n^{(2)} } \right\|_F^2 }{ \NumOfB n}=  \frac{ \left\| \left(\gamma \bbW_n^{(1)} - \bbW_n^{(2)}  \right) \bbV_{\bbW_n^{(2)} }  \right\|_F^2}{\NumOfB n}+ O_P \left(  \frac{\NumOfB [\log (n)]^2 }n  \right).
	$$
 Finally, together with further argument using Corollary \ref{cor:bbV bbW also OK rather than bbV bbP}, we obtains \eqref{eq:upper bound two views}.

\section*{Acknowledgements}\label{acknowledge}
This research was supported in part by the National Institutes of Health Grants GM123056 and GM129781.

\section*{Supplementary Material}\label{supp_mat}         
\textbf{Supplement to "Two-sample test of community memberships of weighted stochastic block models''}
(.pdf file). The supplement includes: (i) proofs of all theoretical results in the main paper, (ii) additional technical tools and supporting lemmas, and (iii) additional numerical results when the number of communities is greater than 2.

\bibliographystyle{imsart-nameyear}
\bibliography{statistics}

\newpage

\input{TestSBM-AppendixB}

\end{document}

%% file: TestSBM-AppendixB.tex
\setcounter{section}{2}
\section{Proofs of asymptotic expansions of the singular subspace distances}
\label{append:proofs of asymptotic expansions of the singular subspace distances}
In this Section, we present proof of Lemma \ref{lem:tang asymptotic expansion}, Theorem \ref{thm:asymptotic expansion one view} and  Theorem \ref{thm:asymptotic expansion two views} for the asymptotic expansions of the singular subspace distances. 	In the following proofs, we may ignore the  subscript $n$ when no  confusion exists. 
These asymptotic expansions are needed to derive the asymptotic distribution of the proposed test statistic.

\subsection{Proof of Lemma \ref{lem:tang asymptotic expansion}}
\label{subsec:Proof of Lemma {lem:tang asymptotic expansion}}
The techniques to prove Lemma \ref{lem:tang asymptotic expansion} are similar to (2.5) of Theorem 3.1 in \cite{tang2018limit}. We briefly outline the proof here.

First notice that \eqref{eq:expectation dominant term}, which we derive later on, says
$$ \frac{  \left\| \left( \bbW_n - \bbE_n\right) \bbV_{\bbE_n}   \right\|_F^2}{ \NumOfB n b_{\cP}}
= 	\frac{(\NumOfB -1) n \sigma_Q^2  +  n\sigma_P^2 + o\left(n \sigma_{\cP}^2\right)}{\NumOfB n b_{\cP}} =	O_P(1),$$
which implies 
\begin{equation} \frac{ \left\| \left( \bbW_n - \bbE_n\right) \bbV_{\bbE_n}   \right\|_F }{\sqrt{\NumOfB n b_{\cP}}} 
 = 		O_P(1).
 \label{eq:normalized Frob norm}
 \end{equation}

Second,  for the asymptotic expansion, instead of using spectral embeddings as in \cite{tang2017semiparametric,tang2018limit}, we have singular vector matrices.
For convenience,   instead of writing transformation $\bbT$ on the left $n\times\NumOfB$ matrix as it appeared in \eqref{eq:tang asymptotic expansion}, we write it on the right $n\times\NumOfB$ matrix -- this follows the style that although (2.5)  in Theorem 2.1 and (3.1) in Theorem 3.1 of \cite{tang2018limit} has orthogonal matrix multiplying on the left one, Section (B.19-22) in B.2 of \cite{tang2018limit} has it mutiplying on the right one.  

\ignore{Before  the proof, we state a high-probability fact,
	\begin{lem}[Full rank]	$\rank  \bbW_n \ge  \NumOfB$ with high probability. $\rank \bbV_{\bbW_n^{(1)}} = \rank \bbV_{bbW_n^{(2)}} = \NumOfB$, or alternatively, $\rank \bbV_{\bbW_n^{(2)}}^T\bbV_{\bbW_n^{(1)}} = \NumOfB$, with high probability. \label{claim:full rank assumption on BernP}
	\end{lem}
	\cite{tang2017semiparametric,tang2018limit} implicitly utilize this lemma/ fact in their proof of the asymptotic expansions. 
}
Before our derivation, we present Corollary \ref{cor:improved Davis Kahan spec}, which is a consequence of Davis-Kahan $\sin \Theta$ theorem. Although some classical forms are in \cite{stewart1990matrix,davis1970rotation}, we present Davis-Kahan $\sin \Theta$ theorem in the context  of our setting,  which is analogous to \cite{hsu2016COMS4772}:
\begin{lem}[Davis-Kahan $\sin \Theta$] \label{lem:improved Davis-Kahan}
	Denote singular value decomposition of $\bbW_n$ as  $\bbW_n  = \bbV_{\bbW_n} \DiaEig_{\bbW_n} \bbV_{\bbW_n}^T + \bbV_{\bbW_n}^{\perp} \DiaEig_{\bbW_n}^{\perp} \left(\bbV_{\bbW_n}^{\perp} \right)^T$, and similarly for  $\bbE_n$. Suppose $ \left\| \DiaEig_{\bbW_n}^{\perp}  \right\|_2  < \left\|  \DiaEig_{\bbE_n}^{-1}\right\|_2^{-1}$, where $  \left\| \DiaEig_{\bbE_n}^{- 1} \right\|_2^{-1 }    $ is the $\NumOfB$th (absolutely) largest eigenvalue of $\bbE_n$, $     \left\| \DiaEig_{\bbW_n} \right\|_2 $ is the $(\NumOfB+ 1 )$th (absolutely) largest eigenvalue of $\bbW_n$. Then for any unitarily-invariant norm $\| \cdot  \|_{\cU}$ (and we focus on $\|\cdot \|_{\cU} = \| \cdot \|_2, \| \cdot \|_{F}$),
	$$
	\left\| \left( \bbV_{\bbW_n }^{\perp} \right)^T  \bbV_{\bbE_n}   \right\|_{\cU}  \le \frac{ \left\| \left(\bbW_n - \bbE_n \right) \bbV_{ \bbE_n } \right\|_{\cU}  }{ \left\| \DiaEig_{\bbE_n}^{- 1} \right\|_2^{-1}  -  \left\| \DiaEig_{\bbW_n}^{\perp} \right\|_2}, $$
\end{lem}

\begin{proof} For any unitarily invariant norm $\|\cdot \|_{\cU} $, we have
	\begin{eqnarray*}
		&& \left\| \left(\bbV_{\bbW_n}^{\perp}\right)^T \left(\bbW_n - \bbE_n \right) \bbV_{\bbE} \right\|_{ \cU} \\
		& = &\left\|  \left(\bbV_{\bbW_n}^{\perp}\right)^T  \bbW_n \bbV_{\bbE_n} - \left(\bbV_{\bbW_n}'\right)^T \bbE_n   \bbV_{\bbE_n} \right\|_{\cU} \\
		& = &  \left\| \DiaEig_{\bbW_n}^{\perp} \left(\bbV_{\bbW_n}^{\perp} \right)^T \bbV_{\bbE_n} -  \DiaEig_{\bbE_n} \left(\bbV_{\bbW_n}^{\perp} \right)^T \bbV_{\bbE_n  } \right\|_{ \cU} \\	
		& = & 	 \left\| \left( \DiaEig_{\bbW_n}^{\perp} -    c \bbI_{ \NumOfB }  \right) \left(\bbV_{\bbW_n}' \right)^T \bbV_{\bbE_n}  -  \left( \DiaEig_{\bbE_n} - c \bbI_{ \NumOfB }  \right)  \left(\bbV_{\bbW_n}^{\perp} \right)^T \bbV_{\bbE_n} - \left(\DiaEig_{\bbW_n}^{\perp} \right) \right\|_{ \cU} \\
		& \ge &\left\| \left(  \DiaEig_{\bbW_n}^{\perp} -    c \bbI_{ \NumOfB }  \right) \left(\bbV_{\bbW_n}^{\perp} \right)^T \bbV_{\bbE_n } \right)  - \rho \left(  \left( \DiaEig_{\bbE_n} - c \bbI_{ \NumOfB }  \right)  \left(\bbV_{\bbW_n}^{\perp} \right)^T \bbV_{\bbE_n } - \left(\DiaEig_{\bbW_n}^{\perp} \right)  \right\|_{ \cU} ,
	\end{eqnarray*}
	while $c$ can be chosen arbitrarily, we pick $c = \frac{ \left\| \DiaEig_{\bbE_n} \right\|_2 +  \left\| \DiaEig_{\bbE_n}^{- 1} \right\|_2^{-1 } }2$; 
	thus
	\begin{eqnarray*}
		&& \left\|  \left(\bbV_{\bbW_n}^{\perp} \right)^T \left(\bbW_n - \bbE_n \right) \bbV_{\bbE} \right\|_{\cU} \\
		& \ge &  	\left[ \frac1{  \left\| \left( \DiaEig_{\bbW_n}^{\perp} - c \bbI_{ \NumOfB } \right)^{-1} \right\|_2 } -  \left\|  \DiaEig_{\bbE_n} - c \bbI_{ \NumOfB } \right\|_2 \right] \cdot \left\|  \left(\bbV_{\bbW_n}^{\perp} \right)^T \bbV_{\bbE_n } \right\|_{ \cU} \\
		& = & \left[ \frac1{  \left\| \left(\DiaEig_{\bbW_n}^{\perp} - c \bbI_{ \NumOfB } \right)^{-1} \right\|_2 } -  \frac{  \left\| \DiaEig_{\bbE_n} \right\|_2 -  \left\| \DiaEig_{\bbE_n}^{- 1} \right\|_2^{-1 } }2 \right] \cdot \left\|  \left(\bbV_{\bbW_n}^{\perp} \right)^T \bbV_{\bbE_n } \right\|_{ \cU},
	\end{eqnarray*}
	where under assumption $\left\|  \DiaEig_{\bbW_n}^{\perp} \right\|_2 < \left\|  \DiaEig_{\bbE_n}^{-1}\right\|_2^{-1}$, we have 
	$$\frac1{ \left\| \left(\DiaEig_{\bbW_n}^{\perp} - c \bbI_{ \NumOfB } \right)^{-1} \right\|_2 } \ge \frac{   \left\| \DiaEig_{\bbE_n} \right\|_2 +  \left\| \DiaEig_{\bbE_n}^{- 1} \right\|_2^{-1 } }2 -  \left\| \DiaEig_{\bbW_n}^{\perp} \right\|_2, $$ thus
	$$
	\left\| \left(\bbV_{\bbW_n}^{\perp} \right)^T \left(\bbW_n - \bbE_n \right) \bbV_{\bbE_n } \right\|_{\cU} \ge \left[   \left\| \DiaEig_{\bbE_n}^{- 1} \right\|_2^{-1 }  -   \left\| \DiaEig_{\bbW_n}^{\perp} \right\|_2 \right] \cdot  \left\|  \left(\bbV_{\bbW_n}' \right)^T \bbV_{\bbE_n } \right\|_{\cU} .
	$$
	
	While the proof above is from \cite{hsu2016COMS4772}, we particularly write  $\bbV_{\bbE}$ on the right to elaborate the proof better,  although this is not quite important due  to rigidity of eigenvalues \citep{alma9977075773903681}. 
	
	On the other hand, \eqref{eq:matrix norm inequality}, \eqref{eq:induced norm submultiplicative} for $\|\cdot \|_{\cU} = \| \cdot \|_2, \| \cdot \|_{F}$ imply
	\begin{eqnarray*}
		\left\| \left(\bbV_{\bbW_n}^{\perp}\right)^T \left(\bbW_n - \bbE_n \right) \bbV_{\bbE} \right\|_{\cU} & \le &  \left\| \left( \bbV_{\bbW_n}^{\perp} \right)^T \right\|_2 \cdot \left\| \left(\bbW_n - \bbE_n \right) \bbV_{\bbE_n } \right\|_{\cU}  \\
		\xlongequal{ } && \left\| \left(\bbW_n - \bbE_n \right) \bbV_{\bbE_n } \right\|_{\cU} , 
	\end{eqnarray*}
	because 	 $\left\| \left( \bbV_{\bbW_n}^{\perp} \right)^T \right\|_2 = \left\|  \bbV_{\bbW_n}^{\perp} \right\|_2 = 1$. This 	
	implies $$ \left\|  \left(\bbV_{\bbW_n}^{\perp} \right)^T \bbV_{\bbE_n } \right\|_{\cU} \le \frac{ \left\| \left(\bbW_n - \bbE_n \right) \bbV_{\bbE_n } \right\|_{\cU}}{  \left\| \DiaEig_{\bbE_n}^{- 1} \right\|_{2}^{-1 }  -   \left\| \DiaEig_{\bbW_n}^{\perp} \right\|_{2 } }. $$

\end{proof}
\begin{cor}
	\label{cor:improved Davis Kahan spec}
	In case $\|\cdot \|_{\cU} = \| \cdot \|_2, \| \cdot \|_{F}$, combining Lemma \ref{lem:improved Davis-Kahan} with probabilistic upper bounds for spectra of edge-independent random graphs \citep{lu2013spectra}, we obtain
	\begin{eqnarray}
	\left\| \left( \bbV_{\bbW_n }^{\perp} \right)^T  \bbV_{\bbE_n}   \right\|_2 &= & O_P\left( \sqrt{\frac{\NumOfB}{nb_{\cP}} } \right) , \label{eq:improved Davis Kahan spec} \\
	\left\| \left( \bbV_{\bbW_n }^{\perp} \right)^T  \bbV_{\bbE_n}  \right\|_F  &= & O_P\left( \frac{\NumOfB}{\sqrt{nb_{\cP}} } \right). \label{eq:improved Davis Kahan Frob}
	\end{eqnarray}
\end{cor}
\begin{proof}
	From  Lemma \ref{lem:improved Davis-Kahan}, rigidity of eigenvalues (1.5) in \cite{alma9977075773903681}  or spectra of eigenvalues \cite{lu2013spectra} imply that \eqref{eq:improved Davis Kahan spec} holds with high probability;   $ \left\| \DiaEig_{\bbE_n}^{- 1}  \right\|_2^{-1 } \asymp \frac{nb_{\cP}}{\NumOfB}$. The only difference between $\| \cdot \|_2, \| \cdot \|_F$ in \eqref{eq:improved Davis Kahan spec}, \eqref{eq:improved Davis Kahan Frob} is that $\left\| \left(\bbW_n - \bbE_n \right) \bbV_{\bbE_n } \right\|_{F}=O_P \left(\sqrt{n b_{P_n }}\right) $ while under  the  Assumptions \ref{assum:region of interest Renyi divergence} and \ref{assum:size of each block}.
		\begin{equation}
		\left\| \left(\bbW_n - \bbE_n\right) \bbV_{\bbE_n} \right\|_2 = O_P\left( \sqrt{\frac{n b_{P_n }}{\NumOfB} }\right).
		\label{eq:V'(W-P)V}
		\end{equation} We also refer to Theorem 5 of \cite{lu2013spectra}.
		
\end{proof}

Recall Proposition A.3 in \cite{tang2017semiparametric}:
\begin{equation}
\left\| \Procrustes \left(  \bbV_{\bbW_n } , \bbV_{\bbE_n} \right)  - \left( \bbV_{\bbW_n } \right)^T  \bbV_{\bbE_n}  \right\|_2  = O_P\left( \left[\frac{\NumOfB}{nb_{\cP}} \right]^3 \right). \label{eq:Procrustes to UAWUE}
\end{equation}

Now we are ready to prove the rest of lemma \ref{lem:tang asymptotic expansion}. In order to be consistent with proof strategy for Theorem 3.1 in \cite{tang2018limit}, we start with $ \frac{\sqrt{n b_{\cP}}}{\NumOfB} \left[ \bbV_{\bbW_n}  - \bbV_{\bbE_n}{\tiny } \Procrustes \left(\bbV_{\bbE_n} , \bbV_{\bbW_n} \right) \right] $ which is $O_P(1)$ by heuristically referring to \eqref{eq:normalized Frob norm}, \eqref{eq:improved Davis Kahan Frob}:
{\allowdisplaybreaks
	\begin{eqnarray*}
	&& \frac{\sqrt{n b_{\cP}}}{\NumOfB} \left[ \bbV_{\bbW_n}  - \bbV_{\bbE_n}{\tiny } \Procrustes \left(\bbV_{\bbE_n} , \bbV_{\bbW_n} \right) \right]  \\
	& \xlongequal{\text{\eqref{eq:Procrustes to UAWUE}}} & \frac{\sqrt{n b_{\cP}}}{\NumOfB} \left[\bbI_{n} - \bbV_{\bbE_n}  \bbV_{\bbE_n}^T\right] \bbV_{\bbW_n}+ O_P\left( \left[\frac{\NumOfB}{nb_{\cP}} \right]^{-\frac52 } \right) \\
	& \xlongequal[= \bbV_{\bbW_n}\DiaEig_{\bbW_n}]{\bbW_n \bbV_{\bbW_n} } &  \frac{\sqrt{n b_{\cP}}}{\NumOfB} \left[\bbI_{n} - \bbV_{\bbE_n}  \bbV_{\bbE_n}^T\right] \bbW_n \bbV_{\bbW_n} \DiaEig_{\bbW_n}^{-1} +O_{P}\left(\frac{\sqrt{\NumOfB}}{n b_{\cP}} \right) \\
	& \xlongequal{\bbE_n = \bbV_{\bbE_n}  \bbV_{\bbE_n}^T\bbE_n } 
	& \frac{\sqrt{n b_{\cP}}}{\NumOfB} \left[\bbI_{n} - \bbV_{\bbE_n}  \bbV_{\bbE_n}^T\right] \left( \bbW_n - \bbE_n\right) \bbV_{\bbW_n} \DiaEig_{\bbW_n}^{-1} + O_{P}\left(\frac{\sqrt{\NumOfB}}{n b_{\cP}} \right)   \\
	& \xlongequal{\text{\eqref{eq:term Wigner order}}} &  \frac{\sqrt{n b_{\cP}}}{\NumOfB}  \left( \bbW_n - \bbE_n\right) \bbV_{\bbW_n} \DiaEig_{\bbW_n}^{-1} + O_{P}\left(\frac{\sqrt{\NumOfB}}{n b_{\cP}} \log (n)\right)\\
	& = & \frac{\sqrt{n b_{\cP}}}{\NumOfB}  \left( \bbW_n - \bbE_n\right) \left( \bbI_{n} - \bbV_{ \bbE_n }\bbV_{ \bbE_n }^T\right)\bbV_{\bbW_n} \DiaEig_{\bbW_n}^{-1}\\
	&& \frac{\sqrt{n b_{\cP}}}{\NumOfB} \left( \bbW_n - \bbE_n\right) \bbV_{ \bbE_n }\bbV_{ \bbE_n }^T\bbV_{\bbW_n} \DiaEig_{\bbW_n}^{-1} + O_{P}\left(\frac{\sqrt{\NumOfB}}{n b_{\cP}} \log (n)\right)\\
	& \xlongequal{\text{\eqref{eq:improved Davis Kahan spec} or \eqref{eq:improved Davis Kahan Frob}}}& \frac{\sqrt{n b_{\cP}}}{\NumOfB} \left( \bbW_n - \bbE_n\right) \bbV_{ \bbE_n }\bbV_{ \bbE_n }^T\bbV_{\bbW_n} \DiaEig_{\bbW_n}^{-1} + O_{P}\left(\frac{\sqrt{\NumOfB}}{n b_{\cP}} \log (n)\right).
\end{eqnarray*} 

By an  argument similar to \eqref{eq:swapping in my quantity} and (B.8-10)'s contribution to (B.20) in \cite{tang2018limit}, we have
\begin{eqnarray*}
&& \frac{\sqrt{n b_{\cP}}}{\NumOfB} \left[ \bbV_{\bbW_n}  - \bbV_{\bbE_n}{\tiny } \Procrustes \left(\bbV_{\bbE_n} , \bbV_{\bbW_n} \right) \right]  \\
&= &\frac{\sqrt{n b_{\cP}}}{\NumOfB}  \left( \bbW_n - \bbE_n\right) \bbV_{ \bbE_n }\bbV_{ \bbE_n }^T\bbV_{\bbW_n} \DiaEig_{\bbW_n}^{-1} + O_{P}\left(\frac{\sqrt{\NumOfB}}{n b_{\cP}} \log (n)\right) \\
 & = & \frac{\sqrt{n b_{\cP}}}{\NumOfB}  \left( \bbW_n - \bbE_n\right) \bbV_{ \bbE_n }\DiaEig_{\bbE_n}^{-1} \bbV_{ \bbE_n }^T\bbV_{\bbW_n}  + O_{P}\left(\frac{\sqrt{\NumOfB}}{n b_{\cP}} \log (n)\right) \\
	&\xlongequal{\text{\eqref{eq:Procrustes to UAWUE}}}&	\frac{\sqrt{n b_{\cP}}}{\NumOfB} \left( \bbW_n - \bbE_n\right) \bbV_{\bbE_n} \DiaEig_{\bbE_n}^{-1} \Procrustes \left(\bbV_{\bbE_n} , \bbV_{\bbW_n} \right) \\
	&& + O_{P}\left(\frac{\sqrt{\NumOfB}}{n b_{\cP}} \log (n)\right),
	\end{eqnarray*}
}
which implies 
{ \allowdisplaybreaks
	\begin{eqnarray*}
		&& 	\frac{\sqrt{n b_{\cP}}}{\NumOfB}  \left\| \bbV_{\bbW_n}   \Procrustes \left(\bbE_{\bbW_n} , \bbV_{\bbW_n} \right) - \bbV_{\bbE_n}\right\|_F \\
		&=& 	\frac{\sqrt{n b_{\cP}}}{\NumOfB} \left\| \left( \bbW_n - \bbE_n\right) \bbV_{\bbE_n} \DiaEig_{\bbE_n}^{-1}  \right\|_F + O_{P}\left(\frac{\sqrt{\NumOfB}}{n b_{\cP}} \log (n)\right),
\end{eqnarray*}}
and therefore, 
{\allowdisplaybreaks
	\begin{eqnarray*}
		&& \frac1{\sqrt{\NumOfB n b_{\cP}}} \left\| \left(\bbV_{\bbW_n}   \Procrustes \left(\bbE_{\bbW_n} , \bbV_{\bbW_n} \right) - \bbV_{\bbE_n} \right) \DiaEig_{\bbE_n} \right\|_F \\
		&=&  \frac1{\sqrt{\NumOfB n b_{\cP}}} \left\| \left( \bbW_n - \bbE_n\right) \bbV_{\bbE_n}   \right\|_F + O_{P}\left(\sqrt{\frac{\NumOfB}{n b_{\cP}}} \log (n)\right),
	\end{eqnarray*}
}
which is exactly \eqref{eq:tang asymptotic expansion}.
\subsection{Proof of Theorem \ref{thm:asymptotic expansion one view} on asymptotic expansion of the singular subspace distance}
\label{appendix:proof of one view dominant term}

This section proves \eqref{eq:upper bound one view}:
$$
\frac{ \left\| \left(\bbV_{\bbW_n}\bbT- \bbV_{\bbE_n} \right) \DiaEig_{ \bbE_n} \right\|_F^2  }{\NumOfB n}  - \frac{ \left\|   \left( \bbW_n - \bbE_n \right) \bbV_{\bbE_n}   \right\|_F^2 }{\NumOfB n} = O_P \left( \frac{\NumOfB^{2} [\log(n)]^2 }{n} \right),
$$
where $\bbT= \bbT \left( \bbV_{\bbE_n}, \bbV_{\bbW_n} \right) = \left( \bbV_{ \bbE}^T \bbV_{ \bbW_n}  \right)^{-1 },  \bbV_{ \bbW}^T \bbV_{\bbE_n} , \Procrustes \left( \bbV_{\bbW_n} ,  \bbV_{\bbE_n}  \right)  $.

\subsubsection{$ \bbT = \left( \bbV_{ \bbE}^T \bbV_{\bbW} \right)^{-1 } $ }
\label{subsec:bbT VE VW inverse}

In the case $\bbT = \bbT = \left( \bbV_{ \bbE}^T \bbV_{\bbW} \right)^{-1 }$,  we first simplify $$\dps \frac{\left\| \left(\bbV_{\bbW} \left( \bbV_{ \bbE}^T \bbV_{\bbW} \right)^{-1 }  - \bbV_{\bbE} \right) \DiaEig_{ \bbE} \right\|_F^2}{\NumOfB n}$$ by Lemma \ref{lem:lems used for one view}:

\begin{lem}
	\label{lem:lems used for one view}
	\begin{equation}
	\left\| \bbV_{\bbE}^T \left(   \bbW  - \bbE  \right)  \bbV_{\bbW } \right\|_F = O_P \left(  \NumOfB  \log (n) \right). 
	\label{eq:term Wigner order}
	\end{equation}
	
	\begin{equation}
	\Lambda_{\bbW}^{-1}\left( \bbV_{ \bbE}^T \bbV_{\bbW} \right)^{-1 }\DiaEig_{ \bbE}  = \left( \bbV_{ \bbE}^T \bbV_{\bbW} \right)^{-1 } + O_P\left( \frac{\NumOfB^3\log(n)}{n^2 b_{\cP}^2}  \right). \label{eq:swapping in my quantity}
	\end{equation}	
	
	\begin{equation}
	\left\|  \left(  \bbW - \bbE \right)  \left(\bbV_{\bbW}    \left( \bbV_{ \bbE}^T \bbV_{\bbW} \right)^{-1 }  - \bbV_{\bbE} \right)   \right\|_F =  O_P\left( \NumOfB \right)  . \label{eq:my B8 AJ}
	\end{equation}
\end{lem}

\begin{proof} As of \eqref{eq:term Wigner order}, Lemma A.4 of \cite{tang2017semiparametric} provides  its version in RDPG since $\bbV_{\bbE}^T \left(   \bbW  - \bbE  \right)  \bbV_{\bbW } = \DiaEig_{\bbE } \bbV_{\bbE}^T \bbV_{\bbW }  - \bbV_{\bbE}^T \bbV_{\bbW } \DiaEig_{\bbW}$. 
	
	Heuristically, 
	$$	\left\| \bbV_{\bbE}^T \left(   \bbW  - \bbE  \right)  \bbV_{\bbW } \right\|_F \simeq \left\| \bbV_{\bbE}^T \left(   \bbW  - \bbE  \right)  \bbV_{\bbE } \right\|_F =  \left\| \WignerLikeRM \right\|_F \le O_P \left( \NumOfB  \log (n)\right), $$
	where the  last bound is due to the fact that 
	$$
	\WignerLikeRM \triangleq \begin{bmatrix}
	\dps \sum_{i,j:\cK(i) = s,\cK(j) = t}\frac{w_{ij} - \E w_{ij}}{\sqrt{\# \cC_s \# \cC_t}}
	\end{bmatrix}_{s,t \in [\NumOfB]}
	$$
	and notice that for the square of each entry
	$$\E  \left[			\sum_{i,j:\cK(i) = s,\cK(j) = t}\frac{(w_{ij} - \E w_{ij})^2 }{\sqrt{\# \cC_s \# \cC_t}}\right] =      			\sum_{i,j:\cK(i) = s,\cK(j) = t}\frac{\Var w_{ij} }{\# \cC_s \# \cC_t} = O_P\left(b_{\cP}\right).
	$$
	
	Strictly speaking,
	\begin{eqnarray*}
		&& \left\| \bbV_{\bbE}^T \left(   \bbW  - \bbE  \right)  \bbV_{\bbW } \Procrustes-  \bbV_{\bbE}^T \left(   \bbW  - \bbE  \right)  \bbV_{\bbE } \right\|_F \\
		& =  &  \left\| \bbV_{\bbE}^T \left(   \bbW  - \bbE  \right)  \left(\bbV_{\bbW } \Procrustes  - \bbV_{\bbE }\right) \right\|_F \le \left\| \bbV_{\bbE}^T \left(   \bbW  - \bbE  \right) \right\|_2 \cdot  \left\|   \bbV_{\bbW } \Procrustes  - \bbV_{\bbE }  \right\|_F,
	\end{eqnarray*}
	while the first part can be controlled by \eqref{eq:V'(W-P)V}, the second part can be
	controlled by Davis-Kahan theorem with upper bound of order $O_P\left( \frac{K}{ \sqrt{nb_{\cP}}}\right)$; hence,
	\begin{eqnarray*}
		&& \left\| \bbV_{\bbE}^T \left(   \bbW  - \bbE  \right)  \bbV_{\bbW } \Procrustes-  \bbV_{\bbE}^T \left(   \bbW  - \bbE  \right)  \bbV_{\bbE } \right\|_F \\
		& \le & O_P\left( \sqrt{\frac{n b_{\cP} }{\NumOfB}}  \right)\cdot  O_P \left( \frac{K}{ \sqrt{nb_{\cP}}} \right) = O_P \left( \sqrt{\NumOfB } \right),
	\end{eqnarray*}
	
	as well as
	{\allowdisplaybreaks
		\begin{eqnarray*} 
			&& \left\| \bbV_{\bbE}^T \left(   \bbW  - \bbE  \right)  \bbV_{\bbW } \Procrustes-  \bbV_{\bbE}^T \left(   \bbW  - \bbE  \right)  \bbV_{\bbE } \right\|_F + \left\| \bbV_{\bbE}^T \left(   \bbW  - \bbE  \right)  \bbV_{\bbE } \right\|_F  \\
			&		\ge  & \left\| \bbV_{\bbE}^T \left(   \bbW  - \bbE  \right)  \bbV_{\bbW }  \right\|_F, \\
			&& \left\| \bbV_{\bbE}^T \left(   \bbW  - \bbE  \right)  \bbV_{\bbW } \Procrustes -  \bbV_{\bbE}^T \left(   \bbW  - \bbE  \right)  \bbV_{\bbE } \right\|_F + \left\| \bbV_{\bbW}^T \left(   \bbW  - \bbE  \right)  \bbV_{\bbE } \right\|_F \\
			&		\ge  & \left\| \bbV_{\bbE}^T \left(   \bbW  - \bbE  \right)  \bbV_{\bbE }  \right\|_F.
		\end{eqnarray*}
	}
	As for \eqref{eq:swapping in my quantity}, in addition to $\left\| \DiaEig_{\bbE_n}\right\|_2 = O \left( \frac{n b_{\cP}}{\NumOfB}\right)$ according to the Assumption \ref{assum:homogeneous weighted SBM}, \ref{assum:size of each block}, we have 
	\begin{eqnarray*}
		&& 	 \DiaEig_{\bbW}^{-1} \left(\bbV_{\bbE}^T \bbV_{\bbW}\right)^{-1} -  \left(\bbV_{\bbE}^T \bbV_{\bbW}  \right)^{ -1 } \DiaEig_{\bbE}^{-1} \\   & = & 	 \DiaEig_{\bbW}^{-1} \left(\bbV_{\bbE}^T \bbV_{\bbW}\right)^{-1} \left[ \DiaEig_{\bbE} \bbV_{\bbE}^T \bbV_{\bbW} - \bbV_{\bbE}^T \bbV_{\bbW} \DiaEig_{\bbW} \right]  \left(\bbV_{\bbE}^T \bbV_{\bbW}\right)^{-1} \DiaEig_{\bbE }^{-1} \\
		& = &  \DiaEig_{\bbW }^{-1}   \left(\bbV_{\bbE}^T \bbV_{\bbW}\right)^{-1}  \bbV_{\bbE}^T  ( \bbW-\bbE )\bbV_{\bbW}  \left(\bbV_{\bbE}^T \bbV_{\bbW}\right)^{-1}  \DiaEig_{\bbE}^{-1}  \xlongequal{}  O_P\left( \frac{\NumOfB^3 \log (n) }{n^2 b_{\cP}^2} \right)
	\end{eqnarray*}
	based on \text{\eqref{eq:term Wigner order}}.
	As for  \eqref{eq:my B8 AJ},
	\begin{eqnarray*}
		&& \left\| 	   \left(\bbW  -  \bbE \right)\left(\bbV_{\bbW}    \left( \bbV_{ \bbE}^T \bbV_{\bbW} \right)^{-1 }  - \bbV_{\bbE} \right)   \right\|_F \\ 
		& \le &    \left\| \bbW  -  \bbE \right\|_2 \cdot \left\| \bbV_{\bbW}    \left( \bbV_{ \bbE}^T \bbV_{\bbW} \right)^{-1 }  - \bbV_{\bbE} \right\|_F \\
		& = & O_P\left( \sqrt{ \frac{n b_{\cP}}{\NumOfB } } \right) \cdot O_P\left( \frac{\NumOfB}{\sqrt{nb_{\cP}}} \right) =  O_P\left(  \sqrt{\NumOfB} \right),
	\end{eqnarray*}
	where we refers to Theorem 3.1 of \cite{oliveira2009concentration} and \cite{lu2013spectra}.
\end{proof}
 We now have
{\allowdisplaybreaks
	\begin{eqnarray*}
	&& \frac1{\sqrt{\NumOfB n}} \left(\bbV_{\bbW} \left( \bbV_{ \bbE}^T \bbV_{\bbW} \right)^{-1 }  - \bbV_{\bbE} \right) \DiaEig_{ \bbE}\\
	& = &  \frac1{\sqrt{\NumOfB n}} \left( \bbI_n   - \bbV_{\bbE} \bbV_{ \bbE}^T \right) \bbV_{\bbW} \left( \bbV_{ \bbE}^T \bbV_{\bbW} \right)^{-1 }\DiaEig_{ \bbE}\\ 
	&= & \frac1{\sqrt{\NumOfB n}} \left( \bbI_n   - \bbV_{\bbE} \bbV_{ \bbE}^T \right) \bbW \bbV_{\bbW} \Lambda_{\bbW}^{-1}\left( \bbV_{ \bbE}^T \bbV_{\bbW} \right)^{-1 }\DiaEig_{ \bbE}\\ 
	& = & \frac1{\sqrt{\NumOfB n}} \left( \bbI_n   - \bbV_{\bbE} \bbV_{ \bbE}^T \right) \bbW  \bbV_{\bbW} \Lambda_{\bbW}^{-1}\left( \bbV_{ \bbE}^T \bbV_{\bbW} \right)^{-1 }\DiaEig_{ \bbE} \\
	& \xlongequal[{\text{\eqref{eq:improved Davis Kahan spec}}}]{\text{\eqref{eq:swapping in my quantity}}}& \frac1{\sqrt{\NumOfB n}} \left( \bbI_n   - \bbV_{\bbE} \bbV_{ \bbE}^T \right) \bbW \bbV_{\bbW}  \left( \bbV_{ \bbE}^T \bbV_{\bbW} \right)^{-1 } +O_P \left( \frac{\NumOfB^2 \log (n ) }{n^2b_{\cP}^{\frac32}} \right),
\end{eqnarray*}
}where the last step uses the fact that 
\begin{eqnarray*}
	&& \frac1{\sqrt{\NumOfB n}} \left\| \left( \bbI_n   - \bbV_{\bbE} \bbV_{ \bbE}^T \right) \bbW  \bbV_{\bbW} \right\|_2    = \frac1{\sqrt{\NumOfB n}} \left\|  \bbV_{\bbE}^{\perp} \left(\bbV_{ \bbE}^{\perp}\right)^T  \bbV_{\bbW} \DiaEig_{\bbW_n} \right\|_2   \\
	& \le  & \frac1{\sqrt{\NumOfB n}} \cdot  \left\|\left(\bbV_{ \bbE}^{\perp}\right)^T  \bbV_{\bbW}  \right\|_2 \cdot \left\|   \DiaEig_{\bbW_n} \right\|_2 =  O_P \left( \frac{\sqrt{b_{\cP}}}{\NumOfB} \right);
\end{eqnarray*}
hence,
\begin{eqnarray}
	&& \frac1{\sqrt{\NumOfB n}} \left(\bbV_{\bbW} \left( \bbV_{ \bbE}^T \bbV_{\bbW} \right)^{-1 }  - \bbV_{\bbE} \right) \DiaEig_{ \bbE} \label{eq: part using the fact that E = VEVET E} \\
	& \xlongequal{\bbE = \bbV_{\bbE} \bbV_{\bbE}^T\bbE } &\frac1{\sqrt{\NumOfB n}} \left( \bbI_n   - \bbV_{\bbE} \bbV_{ \bbE}^T \right)\left( \bbW - \bbE \right) \bbV_{\bbW}  \left( \bbV_{ \bbE}^T \bbV_{\bbW} \right)^{-1 } +O_P \left( \frac{\NumOfB^2\log (n)}{n^2b_{\cP}^{\frac32}} \right),\nonumber
\end{eqnarray}
because $\bbE =  \bbV_{\bbE} \bbV_{ \bbE}^T \bbE$. 
Consequentially,
{\allowdisplaybreaks
\begin{eqnarray}
&&\frac1{\NumOfB n} \left\| \left(\bbV_{\bbW} \left( \bbV_{ \bbE}^T \bbV_{\bbW} \right)^{-1 }  - \bbV_{\bbE} \right) \DiaEig_{ \bbE} \right\|_F^2 \nonumber \\
&\xlongequal{\text{\eqref{eq: part using the fact that E = VEVET E}}} &  \frac1{\NumOfB n} \left\| \left( \bbI_n   - \bbV_{\bbE} \bbV_{ \bbE}^T \right) \left( \bbW - \bbE \right) \bbV_{\bbW}  \left( \bbV_{ \bbE}^T \bbV_{\bbW} \right)^{-1 }  \right\|_F^2 + O_P \left( \frac{\NumOfB^2\log (n)}{n^2b_{\cP}^{\frac32}} \right)  \nonumber \\
& = & \frac1{\NumOfB n} \left\|   \left( \bbW - \bbE \right) \bbV_{\bbW}  \left( \bbV_{ \bbE}^T \bbV_{\bbW} \right)^{-1 }  \right\|_F^2  \label{eq:for cor: term hard to handle by second singular value of Erdos Renyi graph}\\
&&  - \frac1{\NumOfB n} \left\|  \bbV_{ \bbE}^T \left( \bbW - \bbE \right) \bbV_{\bbW}  \left( \bbV_{ \bbE}^T \bbV_{\bbW} \right)^{-1 }  \right\|_F^2 +O_P\left( \frac{\NumOfB^2\log (n)}{n^2b_{\cP}^{\frac32}} \right) \nonumber\\
& \xlongequal{\text{\eqref{eq:term Wigner order}}} &  \frac1{\NumOfB n} \left\|   \left( \bbW - \bbE \right) \bbV_{\bbW}  \left( \bbV_{ \bbE}^T \bbV_{\bbW} \right)^{-1 }  \right\|_F^2 + O_P \left( \frac{\NumOfB [\log (n) ]^2  }n\right) \nonumber\\
& \xlongequal{\text{\eqref{eq:my B8 AJ}}} &   \frac1{\NumOfB n} \left\|   \left( \bbW - \bbE \right) \bbV_{\bbE}   \right\|_F^2 + O_P \left( \frac{\NumOfB [\log (n)]^2    }{n} \right) \nonumber.
\end{eqnarray}
}
\subsubsection{$  \bbT  = \bbV_{ \bbW}^T \bbV_{\bbE}  , \Procrustes  \left( \bbV_{\bbW} , \bbV_{ \bbE} \right)$}
\label{subsubsec: orthogonal subspaces square of Frobenius term}
For $\bbT_n = \Procrustes \left( \bbV_{\bbW_n} , \bbV_{\bbE_n} \right)$,  we need two final steps that are similar to each other, both based on same fact about the difference of two squares of Frobenius norm,
$$
\frac1{ \NumOfB n}\left\| \left(\bbV_{\bbW_n} \bbV_{ \bbW_n}^T \bbV_{ \bbE_n}- \bbV_{\bbE_n} \right) \DiaEig_{ \bbE_n}\right\|_F^2 -   \frac1{\NumOfB n} \left\| \left(\bbV_{\bbW_n} \tilde \bbT_n - \bbV_{\bbE_n} \right) \DiaEig_{ \bbE_n}  \right\|_F^2 ,
$$
that is, left-hand side of \eqref{eq:upper bound one view} with $\bbT_n =  \bbV_{ \bbW_n}^T \bbV_{ \bbE_n}$ and $\bbT_n = \tilde \bbT_n$ where 
\begin{equation}\left\| \bbV_{ \bbW_n}^T \bbV_{ \bbE_n} - \tilde \bbT_n\right\|_F = O_P\left(\frac{\NumOfB^{2 }}{ nb_{\cP}}\log (n) \right).
\label{eq:requirement for tilde Tn}
\end{equation}
The difference can be upper bounded by 
{\allowdisplaybreaks
	\begin{eqnarray}
	&& \frac1{ \NumOfB n}\left\| \left(\bbV_{\bbW_n} \bbV_{ \bbW_n}^T \bbV_{ \bbE_n}- \bbV_{\bbE_n} \right) \DiaEig_{ \bbE_n}\right\|_F^2 -   \frac1{\NumOfB n} \left\| \left(\bbV_{\bbW_n} \tilde \bbT_n - \bbV_{\bbE_n} \right) \DiaEig_{ \bbE_n}  \right\|_F^2 \nonumber \\
	& \xlongequal{\text{\eqref{eq:requirement for tilde Tn}}} &  \frac2{\NumOfB n} \left\langle  \bbV_{\bbW_n} \left(  \bbV_{\bbW_n}^T \bbV_{ \bbE_n} - \tilde  \bbT_n  \right) \DiaEig_{ \bbE_n}   ,  \left(\bbV_{\bbW_n}  \bbV_{\bbW_n}^T \bbV_{ \bbE_n} - \bbV_{\bbE_n} \right) \DiaEig_{ \bbE_n}  \right\rangle \nonumber \\
	&& + O_P \left(  \frac{\NumOfB [\log (n)]^2}n \right) \nonumber \\
	& &\xlongequal{\bbV_{\bbW_n}^T \bbV_{\bbW_n}  \bbV_{\bbW_n}^T= \bbV_{\bbW_n}^T }  0+ O_P \left(   \frac1n  \right) = O_P \left(   \frac{\NumOfB [\log (n)]^2}n  \right).  \label{eq:differences of two frobenius norm equivalence of 3T's}
	\end{eqnarray}
}
Then we are ready to present final two steps in a unified proof strategy: equipped with Section \ref{subsec:bbT VE VW inverse}, that is, the proof of \eqref{eq:upper bound one view} for $ \left(\bbV_{\bbW_n}^T\bbV_{\bbE_n}\right)^{-1} $, we take $\tilde \bbT_n = \left(\bbV_{\bbW_n}^T\bbV_{\bbE_n}\right)^{-1}$ in \eqref{eq:differences of two frobenius norm equivalence of 3T's} and then we prove that  \eqref{eq:upper bound one view} holds with $ \bbV_{\bbE_n}^T\bbV_{\bbW_n}$; equipped with the proof of \eqref{eq:upper bound one view} for $ \bbV_{\bbE_n}^T\bbV_{\bbW_n}$, we take $\tilde \bbT_n = \Procrustes \left( \bbV_{\bbW_n} , \bbV_{\bbE_n} \right)$ in \eqref{eq:differences of two frobenius norm equivalence of 3T's} and then we prove that \eqref{eq:upper bound one view} holds with $\tilde \bbT_n = \Procrustes \left( \bbV_{\bbW_n} , \bbV_{\bbE_n} \right)$.

\subsubsection{Variance}

Now we finish proving \eqref{eq:simplification of Cos Sin matrix AJ} and \eqref{eq:upper bound one view}.  From \eqref{eq:variance SBM}, we have 
$$\Var\left[ \frac1{\NumOfB n}\left\|   \left( \bbW_n - \bbE_n \right) \bbV_{\bbE_n}   \right\|_F^2  \right] =O \left(  \frac{b_{\cQ}^2}{nK}  \right).$$
{\allowdisplaybreaks\begin{eqnarray*}
	&& \Var \left[\frac1{ \NumOfB n} \left\| \left(\bbV_{\bbW_n} \bbT_n   - \bbV_{\bbE_n} \right) \DiaEig_{ \bbE_n} \right\|_F^2  \right]   \\
	&= &  \frac1{ \NumOfB n}  \expc\left[ \left\| \left(\bbV_{\bbW_n} \bbT_n   - \bbV_{\bbE_n} \right) \DiaEig_{ \bbE_n} \right\|_F^2 - \expc  \left\| \left(\bbV_{\bbW_n} \bbT_n   - \bbV_{\bbE_n} \right) \DiaEig_{ \bbE_n} \right\|_F^2\right]^2 \\
	& = & \expc\left[ \frac{\left\|   \left( \bbW_n - \bbE_n \right) \bbV_{\bbE_n}   \right\|_F^2 - \expc\left\|   \left( \bbW_n - \bbE_n \right) \bbV_{\bbE_n}   \right\|_F^2}{ \NumOfB n} +  O_P \left(   \frac{\NumOfB [\log (n)]^2}n  \right)\right]^2  \\
	& = &\Var \left[ \frac1{\NumOfB n}\left\|   \left( \bbW_n - \bbE_n \right) \bbV_{\bbE_n}   \right\|_F^2\right]  \\
	&&+ \expc\left[ \frac{\left\|   \left( \bbW_n - \bbE_n \right) \bbV_{\bbE_n}   \right\|_F^2 - \expc\left\|   \left( \bbW_n - \bbE_n \right) \bbV_{\bbE_n}   \right\|_F^2}{ \NumOfB n} \cdot  O_P \left(   \frac{\NumOfB [\log (n)]^2}n  \right)\right] \\
	&& + O \left(   \frac{\NumOfB^2 [\log (n)]^4}{n^2}  \right) \\
	& = &\Var \left[ \frac1{\NumOfB n}\left\|   \left( \bbW_n - \bbE_n \right) \bbV_{\bbE_n}   \right\|_F^2\right]  \\
	&& +O \left( \frac{b_{\cQ}^2}{nK}  \cdot    \frac{\NumOfB [\log (n)]^2}n  \right)  + O \left(   \frac{\NumOfB^2 [\log (n)]^4}{n^2}  \right) \\
		& = &\Var \left[ \frac1{\NumOfB n}\left\|   \left( \bbW_n - \bbE_n \right) \bbV_{\bbE_n}   \right\|_F^2\right]   + O \left(   \frac{\NumOfB^2 [\log (n)]^4}{n^2}  \right) .
\end{eqnarray*}
}
\subsection{Proof of Theorem \ref{thm:asymptotic expansion two views} on asymptotic expansion of the singular subspace distance}
\label{subsec:proof theorem asymptotic expansion two views}

This section proves that under the null of \eqref{eq:hypothesis test for SBM AJ}, upper bound \eqref{eq:upper bound two views} holds:
\begin{eqnarray*}
&&	\frac1{ \NumOfB n}\left\| \left(  \bbV_{\bbW_n^{(1)}} \bbT_n - \bbV_{\bbW_n^{(2)} } \right) \DiaEig_{ \bbW_n^{(2)}}\right\|_F^2 
-   \frac1{\NumOfB n} \left\| \left( \gamma \bbW_n^{(1)} - \bbW_n^{(2)}  \right) \bbV_{\bbW_n^{(2)} }  \right\|_F^2 \\
&= & O_P \left( \frac{\NumOfB^{2} [\log(n)]^2 }{n} \right), \nonumber 
\end{eqnarray*}
 for  $\bbT_n = \Procrustes \left( \bbV_{\bbW_n^{(1)}} , \bbV_{\bbW_n^{(2)}} \right), \bbV_{\bbW_n^{(1)}}^T\bbV_{\bbW_n^{(2)}}, \mbox{ and } \left(\bbV_{\bbW_n^{(2)}}^T\bbV_{\bbW_n^{(1)}}\right)^{-1}$.

An essential difficulty that make two-sample problem different from problem with one network is that  $\bbW_n^{(v)} \ne \bbV_{\bbW_n^{(v)} } \bbV_{\bbW_n^{(v)} }^T\bbW_n^{(v)} $ in two-sample problem. However, the key equality  $\bbE = \bbV_{\bbE} \bbV_{\bbE}^T\bbE$ in one network  problem is used in step \eqref{eq: part using the fact that E = VEVET E} in  the  proof sketches for \eqref{eq:upper bound one view}.   In contrast, for the two-sample problem,  a different approach has to be taken.

Different from one-sample problem, we focus on $\bbT_n = \bbV_{\bbW_n^{(1)}}^T\bbV_{\bbW_n^{(2)}}$ and finish the proof using  Corollary \ref{cor:term hard to handle by second singular value of Erdos Renyi graph} and  Corollary \ref{cor:bbV bbW also OK rather than bbV bbP}, which can be derived  from the proof of one-sample problem in Section \ref{appendix:proof of one view dominant term}.
\subsubsection{Two useful corollaries for two-sample problem}
\label{subsubsec:Two useful corollaries for two-sample problem}
Lemma \ref{lem:lems used for one view} implies  Corollary \ref{cor:term hard to handle by second singular value of Erdos Renyi graph} that is useful  in proving the  dominant term in the two-sample problem.
\begin{cor}
	\begin{equation}\tr \left[  \bbV_{\bbE}^T \bbV_{\bbW}^{\perp }  \DiaEig_{\bbW}^{\perp }   \left[ \bbV_{ \bbW}^{\perp } \right]^T  \bbV_{\bbE}   \right]	= O_P\left( \NumOfB^2 [\log (n)]^2 \right). \label{eq:cor:term hard to handle by second singular value of Erdos Renyi graph}
	\end{equation} 
	\label{cor:term hard to handle by second singular value of Erdos Renyi graph}
\end{cor}
\begin{proof}
	Since { \allowdisplaybreaks
		\begin{eqnarray*}
			&&\frac1{ \sqrt{\NumOfB n}} \left\| \left( \bbV_{\bbW} \bbV_{ \bbW}^T -  \bbI_n    \right) \bbE \bbV_{\bbE} -  \left( \bbW- \bbE \right) \bbV_{\bbE} \right\|_F \\
			& = & \frac1{ \sqrt{\NumOfB n}} \left\| \left( \bbV_{\bbW} \bbV_{ \bbW}^T \bbE -  \bbW    \right)  \bbV_{\bbE}  \right\|_F \\
			& = & \frac1{ \sqrt{\NumOfB n}} \left\| \left[  \bbV_{\bbW} \bbV_{ \bbW}^T \left(\bbE - \bbW \right) -   \bbV_{\bbW}^{ \perp } \DiaEig_{\bbW}^{ \perp } 
			\left(\bbV_{ \bbW}^{ \perp }\right)^T     \right]  \bbV_{\bbE}   \right\|_F \\
			& \le & \frac1{ \sqrt{\NumOfB n}} \left\|  \bbV_{\bbW} \bbV_{ \bbW}^T \left(\bbE - \bbW \right) \bbV_{\bbE}  \right\|_F + \frac1{ \sqrt{\NumOfB n}}  \left\| \bbV_{\bbW}^{ \perp } \DiaEig_{\bbW}^{ \perp } \right\|_2 \cdot \left\| 
			\left(\bbV_{ \bbW}^{ \perp }\right)^T      \bbV_{\bbE}   \right\|_F \\
			& \xlongequal[\text{\eqref{eq:improved Davis Kahan Frob}}]{\text{\eqref{eq:term Wigner order}}}  & O_P \left( \sqrt{ \frac{ \NumOfB }{n}}\log (n) \right) +  \frac1{ \sqrt{\NumOfB n}} \cdot O_P \left( \sqrt{\frac{n b_{\cP}   }{\NumOfB}}\cdot \frac{\NumOfB}{ \sqrt{ n b_{\cP} }}  \right) \\
			& = &  O_P \left( \sqrt{ \frac{ \NumOfB }{n}}\log (n) \right),
		\end{eqnarray*}
	} by recalling $\bbT = \bbV_{ \bbW}^T \bbV_{\bbE}$ for \eqref{eq:upper bound one view}, { \allowdisplaybreaks
		\begin{eqnarray*}
			&& O_P \left(  \frac{ \NumOfB[\log (n)]^2 }{ n } \right) \\
			& = & \frac1{ \NumOfB n}\left\| \left( \bbV_{\bbW} \bbV_{ \bbW}^T -  \bbI_n    \right) \bbV_{\bbE}\DiaEig_{\bbE} \right\|_F^2 -   \frac1{\NumOfB n} \left\| \left( \bbW- \bbE \right) \bbV_{\bbE}  \right\|_F^2 \\
			& = & \frac1{ \NumOfB n}\left\| \left( \bbV_{\bbW} \bbV_{ \bbW}^T -  \bbI_n    \right) \bbE \bbV_{\bbE} \right\|_F^2 -   \frac1{\NumOfB n} \left\| \left( \bbW- \bbE \right) \bbV_{\bbE}  \right\|_F^2\\
			& = &  \frac1{\NumOfB n} \left\langle \left( \bbV_{\bbW} \bbV_{ \bbW}^T -  \bbI_n    \right) \bbE \bbV_{\bbE} ,   \left( \bbV_{\bbW} \bbV_{ \bbW}^T -  \bbI_n    \right) \bbE \bbV_{\bbE} -  \left( \bbW- \bbE \right) \bbV_{\bbE}  \right\rangle \\
			&& + O_P \left(  \frac{ \NumOfB[\log (n)]^2 }n \right) \\
			& = &  \frac1{\NumOfB n} \left\langle \left( \bbV_{\bbW} \bbV_{ \bbW}^T -  \bbI_n    \right) \bbE \bbV_{\bbE} ,   \left( \bbV_{\bbW} \bbV_{ \bbW}^T\bbE-  \bbW   \right)  \bbV_{\bbE} \right\rangle + O_P \left(   \frac{ \NumOfB[\log (n)]^2 }n \right) \\
			& = &   - \frac1{\NumOfB n}\tr \left[   \bbV_{\bbE}^T \bbW   \left( \bbV_{\bbW} \bbV_{ \bbW}^T -  \bbI_n    \right) \bbE \bbV_{\bbE}  \right] + O_P \left(  \frac{ \NumOfB[\log (n)]^2 }n \right) \\
			& = &  - \frac1{\NumOfB n}\tr \left[  \bbV_{\bbE}^T \bbV_{\bbW}^{\perp }  \DiaEig_{\bbW}^{\perp }   \left[ \bbV_{ \bbW}^{\perp } \right]^T  \bbV_{\bbE}  \DiaEig_{\bbE}  \right]  +O_P\left( \frac{ \NumOfB[\log (n)]^2 }n \right),
		\end{eqnarray*}
	}
	and hence we obtain our result.
\end{proof}

It is  worth noticing that it is not easy to achieve such a good upper bound using the bounds on second largest singular value of the adjacency matrix of Erdos Renyi graph \citep{feige2005spectral,oliveira2009concentration,lu2013spectra} as well as  \eqref{eq:improved Davis Kahan spec} and \eqref{eq:improved Davis Kahan Frob}. This implies possible improvement in those fundamental work in (dense) Erdos Renyi model.

Corollary \ref{cor:bbV bbW also OK rather than bbV bbP} is also useful  in substituting $\bbV_{ \bbW}$ by $\bbV_{ \bbE}$ in the results for both one-sample problem (Theorem \ref{thm:asymptotic expansion one view}) and the two-sample problem (Theorem \ref{thm:asymptotic expansion two views}).
\begin{cor}
	\label{cor:bbV bbW also OK rather than bbV bbP}
	\begin{equation}
	\frac1{\NumOfB n} \left\|   \left( \bbW - \bbE \right) \bbV_{\bbW}   \right\|_F^2 =  \frac1{\NumOfB n} \left\|   \left( \bbW - \bbE \right) \bbV_{\bbE}   \right\|_F^2 + O_P \left( \frac{\NumOfB[\log (n)]^{2}  }{n} \right)  . \end{equation}
\end{cor}
\begin{proof}
	While proving Theorem \ref{thm:asymptotic expansion one view} for $ \bbT = \left( \bbV_{ \bbE}^T \bbV_{\bbW} \right)^{-1 } $ in Section \ref{subsec:bbT VE VW inverse}, \eqref{eq:for cor: term hard to handle by second singular value of Erdos Renyi graph} implies 
	$$
	\frac{ \left\|   \left( \bbW - \bbE \right) \bbV_{\bbW}  \left( \bbV_{ \bbE}^T \bbV_{\bbW} \right)^{-1 }  \right\|_F^2}{\NumOfB n} \xlongequal[\text{\eqref{eq:my B8 AJ}}]{\text{\eqref{eq:term Wigner order}}}  \frac{  \left\|   \left( \bbW - \bbE \right) \bbV_{\bbE}   \right\|_F^2}{\NumOfB n} + O_P \left( \frac{\NumOfB[ \log (n) ]^2 }{n} \right),
	$$
	or alternatively,
	\begin{eqnarray}
	&&\frac{\tr \left[\left( \bbV_{ \bbW}^T \bbV_{\bbE} \right)^{-1 }     \bbV_{\bbW}^T   \left( \bbW - \bbE \right)^2 \bbV_{\bbW}  \left( \bbV_{ \bbE}^T \bbV_{\bbW} \right)^{-1 }  \right]}{\NumOfB n} \label{eq:first step of cor: term hard to handle by second singular vlaue of ER graph} \\
	& \xlongequal[\text{\eqref{eq:my B8 AJ}}]{\text{\eqref{eq:term Wigner order}}} &  \frac{\tr  \left[   \bbV_{\bbE}^T \left( \bbW - \bbE \right)^2  \bbV_{\bbE}   \right]}{\NumOfB n} + O_P \left( \frac{\NumOfB [ \log (n) ]^2 }{n} \right). \nonumber
	\end{eqnarray}
The step  above can be directly verified.
	
	Multiplying  on the left by $ \bbV_{ \bbW}^T \bbV_{\bbE}$ and right by $ \bbV_{ \bbE}^T \bbV_{\bbW}$, we obtain 
	{ \allowdisplaybreaks
	\begin{eqnarray*}
		&& \frac{ \left\|   \left( \bbW - \bbE \right) \bbV_{\bbW}   \right\|_F^2}{\NumOfB n}  =\frac{\tr \left[     \bbV_{\bbW}^T   \left( \bbW - \bbE \right)^2 \bbV_{\bbW}   \right]}{\NumOfB n} \\
		& \xlongequal{\text{\eqref{eq:first step of cor: term hard to handle by second singular vlaue of ER graph}}} &  \frac{\tr  \left[ \bbV_{ \bbW}^T \bbV_{\bbE}  \bbV_{\bbE}^T \left( \bbW - \bbE \right)^2  \bbV_{\bbE}  \bbV_{ \bbE}^T \bbV_{\bbW}  \right]}{\NumOfB n} + O_P \left( \frac{\NumOfB[\log (n)]^{2}  }{n} \right) \\
		& = &  \frac{  \left\|   \left( \bbW - \bbE \right) \bbV_{\bbE} \bbV_{ \bbE}^T \bbV_{\bbW}  \right\|_F^2}{\NumOfB n} + O_P \left( \frac{\NumOfB[\log (n)]^{2}  }{n} \right)  \\
		& = &  \frac{  \left\|   \left( \bbW - \bbE \right) \bbV_{\bbE} \right\|_F^2}{\NumOfB n}  + O_P \left( \frac{\NumOfB[\log (n)]^{2}  }{n} \right) \\
		&& +  \frac{ \tr   \left\{    \left( \bbW - \bbE \right) \left[  \bbV_{\bbE} \bbV_{ \bbE}^T \bbV_{\bbW}  \bbV_{\bbW}^T  \bbV_{\bbE} \bbV_{ \bbE}^T  - \bbV_{\bbE}\bbV_{\bbE}^T \right]  \left( \bbW - \bbE \right)  \right\}}{\NumOfB n}, 
	\end{eqnarray*}
}	where notice 
{\allowdisplaybreaks	\begin{eqnarray*}
		&& \left\|  \bbV_{\bbE} \bbV_{ \bbE}^T \bbV_{\bbW}  \bbV_{\bbW}^T  \bbV_{\bbE} \bbV_{ \bbE}^T  - \bbV_{\bbE}\bbV_{\bbE}^T \right\|_F \\
		&= & \left\|  \bbV_{\bbE} \bbV_{ \bbE}^T\left( \bbV_{\bbW}  \bbV_{\bbW}^T - \bbI_n\right)  \bbV_{\bbE} \bbV_{ \bbE}^T   \right\|_F \\
		& = &  \left\|  \bbV_{\bbE} \bbV_{ \bbE}^T\bbV_{\bbW}^{\perp}  \left[\bbV_{\bbW}^{\perp }\right]^T   \bbV_{\bbE} \bbV_{ \bbE}^T   \right\|_F  \xlongequal{\text{\eqref{eq:improved Davis Kahan spec}}}  O_P \left( \frac{\NumOfB}n \right),
	\end{eqnarray*}
}	together with the fact that $ \left\| \bbW - \bbE \right\|_2 =O_P \left( \sqrt{ \frac{n b_{\cP}}{\NumOfB}} \right)$, we obtain 
	\begin{eqnarray*}
		\frac{ \left\|   \left( \bbW - \bbE \right) \bbV_{\bbW}   \right\|_F^2}{\NumOfB n} 	& = &  \frac{  \left\|   \left( \bbW - \bbE \right) \bbV_{\bbE} \right\|_F^2}{\NumOfB n}  + O_P \left( \frac{\NumOfB[\log (n)]^{2}  }{n} \right) + O_P\left(\frac{b_{\cP}}{\NumOfB n}\right)  \\
		& = &  \frac{  \left\|   \left( \bbW - \bbE \right) \bbV_{\bbE} \right\|_F^2}{\NumOfB n}  + O_P \left( \frac{\NumOfB[\log (n)]^{2}  }{n} \right) .
	\end{eqnarray*}

\end{proof}

It is worth noting that Corollary \ref{cor:bbV bbW also OK rather than bbV bbP} is not easy to prove by \cite{feige2005spectral,oliveira2009concentration,lu2013spectra} as well as \eqref{eq:improved Davis Kahan spec} and \eqref{eq:improved Davis Kahan Frob}. We will use this technique  overcome a similar difficulty \eqref{eq:last obstacle in two views} in proving Theorem \ref{thm:asymptotic expansion two views} for the two-sample problem.

A similar lemma to Lemma \ref{lem:lems used for one view} for the one-sample problem also holds for the  two-sample problem.  
\begin{lem}
	\label{lem:lems used for two views}
	\begin{equation}
	\left\| \bbV_{\bbW^{(2)}}^T \left(  \gamma  \bbW^{(1)}  - \bbW^{(2)}  \right)  \bbV_{\bbW^{(1)} } \right\|_F = O_P \left(  \NumOfB  \log (n) \right),
	\label{eq:term Wigner order 2 view}
	\end{equation}
	
	\begin{equation}
	\Lambda_{\bbW^{(2)}}^{-1}\left( \bbV_{ \bbW^{(2)} }^T \bbV_{\bbW^{(1)}} \right)^{-1 }\DiaEig_{ \bbW^{(2)}}  = \left( \bbV_{ \bbW^{(2)}}^T \bbV_{\bbW^{(1)} } \right)^{-1 } +  O_P\left( \frac{ \NumOfB^3 \log (n)}{ n b_{\cP}^2 } \right),\label{eq:swapping in my quantity 2 view}
	\end{equation}	
	
	\begin{equation}
	\left\|  \left( \gamma  \bbW^{(1)} - \bbW^{(2)} \right)  \left(\bbV_{\bbW^{(1)} }    \left( \bbV_{ \bbW^{(2)}}^T \bbV_{\bbW^{(2)} } \right)^{-1 }  - \bbV_{\bbW^{(2)}} \right)   \right\|_F =  O_P\left(    \sqrt{\NumOfB} \right)  . \label{eq:my B8 AJ 2 view}
	\end{equation}
\end{lem}

\begin{proof} \eqref{eq:term Wigner order 2 view} is similar to the argument for \eqref{eq:term Wigner order}. As for \eqref{eq:swapping in my quantity 2 view}, in addition to $\rho\left( \DiaEig_{\bbE_n}\right) = O \left( \frac{n}{\NumOfB}\right)$ according to Assumptions \ref{assum:homogeneous weighted SBM} and \ref{assum:size of each block},  notice
	\begin{eqnarray*}
		&& 	 \DiaEig_{\bbW^{(1)} }^{-1} \left(\bbV_{\bbW^{(2)} }^T \bbV_{\bbW^{(1)} }\right)^{-1} -  \left(\bbV_{\bbW^{(2)} }^T \bbV_{\bbW^{(1)}}  \right)^{ -1 } \DiaEig_{\bbW^{(2)}}^{-1} \\   & = & 	 \DiaEig_{\bbW^{(1)} }^{-1} \left(\bbV_{\bbW^{(2)} }^T \bbV_{\bbW^{(1)} }\right)^{-1} \left[ \DiaEig_{\bbW^{(2)} } \bbV_{\bbW^{(2)} }^T \bbV_{\bbW^{(1)} } - \bbV_{\bbW^{(2)} }^T \bbV_{\bbW^{(1)}} \DiaEig_{\bbW^{(1)} } \right]  \left(\bbV_{\bbW^{(2)} }^T \bbV_{\bbW^{(1)} }\right)^{-1} \DiaEig_{\bbW^{(2)} }^{-1} \\
		& = &  \DiaEig_{\bbW^{(1)}}^{-1}   \left(\bbV_{\bbW^{(2)}}^T \bbV_{\bbW^{(1)}}\right)^{-1}  \bbV_{\bbW^{(2)}}^T  ( \bbW^{(1)} -\bbW^{(2)})\bbV_{\bbW^{(1)} }  \left(\bbV_{\bbW^{(2)}}^T \bbV_{\bbW^{(1)} }\right)^{-1}  \DiaEig_{\bbW^{(2)} }^{-1}  \\
		&  \xlongequal{\text{\eqref{eq:term Wigner order 2 view}}} &  O_P\left( \frac{\NumOfB^3 \log (n)}{n^2 b_{\cP}^2 }  \right).
	\end{eqnarray*}
	
As of \eqref{eq:my B8 AJ},
	\begin{eqnarray*}
		&& \left\| 	   \left( \gamma \bbW^{(1)}  -  \bbW^{(2)}  \right)\left(\bbV_{\bbW^{(1)} }    \left( \bbV_{ \bbW^{(2)}}^T \bbV_{\bbW^{(1)} } \right)^{-1 }  - \bbV_{\bbW^{(2)} } \right)   \right\|_F \\ 
		& \le &   \rho \left(\gamma  \bbW^{(1)}   -  \bbW^{(2)}  \right) \cdot \left\| \bbV_{\bbW^{(1)} }    \left( \bbV_{ \bbW^{(2)} }^T \bbV_{\bbW^{(1)} } \right)^{-1 }  - \bbV_{\bbW^{(2)}} \right\|_F \\
		& = & O_P\left(  \sqrt{\frac{n b_{\cP}}{\NumOfB}} \right) \cdot O_P\left( \sqrt{ \frac{\NumOfB}{n b_{\cP}} } \right) =  O_P\left(  \sqrt{\NumOfB} \right).
	\end{eqnarray*}
	where we refers to Theorem 3.1 of \cite{oliveira2009concentration} and \cite{lu2013spectra}.
\end{proof}

Since { \allowdisplaybreaks
	\begin{eqnarray*}
		&&\frac1{ \sqrt{\NumOfB n}} \left\| \left( \bbV_{\bbW^{(1)} } \bbV_{ \bbW^{(1)}  }^T \bbV_{\bbW^{(2)}} - \bbV_{\bbW^{(2)}}\right) \DiaEig_{\bbW^{(2)}} -  \left( \gamma \bbW^{(1)}- \bbW^{(2)} \right) \bbV_{\bbW^{(2)}} \right\|_F \\
		&&\xlongequal{ \bbV_{\bbW^{(2)}} \DiaEig_{\bbW^{(2)}} =  \bbW^{(2)} \bbV_{\bbW^{(2)}}} \\ &&\frac1{ \sqrt{\NumOfB n}} \left\| \left( \bbV_{\bbW^{(1)} } \bbV_{ \bbW^{(1)}  }^T -  \bbI_n    \right) \bbW^{(2)} \bbV_{\bbW^{(2)}} -  \left( \gamma \bbW^{(1)}- \bbW^{(2)} \right) \bbV_{\bbW^{(2)}} \right\|_F \\
		& = & \frac1{ \sqrt{\NumOfB n}} \left\| \left( \bbV_{\bbW^{(1)}} \bbV_{ \bbW^{(1)} }^T \bbW^{(2)} - \gamma  \bbW^{(1)}    \right)  \bbV_{\bbW^{(2)}}  \right\|_F \\
		& = & \frac1{ \sqrt{\NumOfB n}} \left\| \left[  \bbV_{\bbW^{(1)} } \bbV_{ \bbW^{(1)} }^T \left(\bbW^{(2)} -  \gamma \bbW^{(1)} \right) -   \bbV_{\bbW^{(1)}}^{ \perp } \DiaEig_{\bbW^{(1)}}^{ \perp } 
		\left(\bbV_{ \bbW^{(1)}}^{ \perp }\right)^T     \right]  \bbV_{\bbW^{(2)}}   \right\|_F \\
		& \le & \frac{\left\|  \bbV_{\bbW^{(1)}} \bbV_{ \bbW^{(1)}}^T \left(\bbW^{(2)} -\gamma  \bbW^{(1)} \right) \bbV_{\bbW^{(2)}}  \right\|_F}{ \sqrt{\NumOfB n}}  + \frac{  \left\| \bbV_{\bbW^{(1)} }^{ \perp } \DiaEig_{\bbW^{(1)}}^{ \perp } 
			\left(\bbV_{ \bbW}^{ \perp }\right)^T      \bbV_{\bbW^{(2)}}   \right\|_F }{ \sqrt{\NumOfB n}}\\
		& \xlongequal{\text{\eqref{eq:term Wigner order 2 view}}}  & O_P \left( \frac{ \NumOfB \log (n) }{\sqrt{n}} \right) +  \frac1{ \sqrt{\NumOfB n}} \cdot O_P \left(\sqrt{ \frac{n b_{\cP} }{ \NumOfB} } \right) \cdot O_P\left( \sqrt{\frac{\NumOfB}{ n b_{\cP}}} \right)\\
		& =& O_P \left( \frac{ \NumOfB \log (n) }{\sqrt{n}} \right),
	\end{eqnarray*} 
}by taking the difference of two squares of Frobenius norm,
{\allowdisplaybreaks\begin{eqnarray*}
&& \frac{ \left\| \left( \bbV_{\bbW^{(1)}} \bbV_{ \bbW^{(1)}}^T\bbV_{\bbW^{(2)}} -  \bbV_{\bbW^{(2)}}\right)  \DiaEig_{\bbW^{(2)} } \right\|_F^2}{ \NumOfB n} -   \frac{\left\| \left(\gamma \bbW^{(1)} - \bbW^{(2)}  \right) \bbV_{\bbW^{(2)} }  \right\|_F^2 }{\NumOfB n} \nonumber \\
&=& \frac{\left\| \left( \bbV_{\bbW^{(1)}} \bbV_{ \bbW^{(1)}}^T -  \bbI_n    \right) \bbW^{(2)} \bbV_{\bbW^{(2)} } \right\|_F^2 }{ \NumOfB n}-   \frac{ \left\| \left(\gamma \bbW^{(1)} - \bbW^{(2)}  \right) \bbV_{\bbW^{(2)} }  \right\|_F^2 }{\NumOfB n}\nonumber \\
& = &  \frac1{\NumOfB n} \left\langle \left( \bbV_{\bbW^{(1)} } \bbV_{ \bbW^{(1)} }^T -  \bbI_n    \right) \bbW^{(2)} \bbV_{\bbW^{(2)} } , \right. \nonumber \\
&& \left.  \left( \bbV_{\bbW^{(1)} } \bbV_{ \bbW^{(1)} }^T -  \bbI_n    \right) \bbW^{(2)} \bbV_{\bbW^{(2)}} -  \left( \gamma \bbW^{(1)} - \bbW^{(2)} \right) \bbV_{\bbW^{(2)}}  \right\rangle \nonumber \\
&& + O_P \left(  \frac{ \NumOfB [\log (n)]^2 }{ n } \right)  \nonumber \\
& = &  \frac1{\NumOfB n} \left\langle \left( \bbV_{\bbW^{(1)} } \bbV_{ \bbW^{(1)} }^T -  \bbI_n    \right) \bbW^{(2)} \bbV_{\bbW^{(2)}} ,   \left( \bbV_{\bbW^{(1)} } \bbV_{ \bbW^{(1)} }^T\bbW^{(2)} - \gamma  \bbW^{(1)}   \right)  \bbV_{\bbW^{(2)}} \right\rangle  \nonumber \\
&& + O_P \left(   \frac{ \NumOfB [\log (n)]^2  }{ n }  \right) \nonumber \\
& = &  - \frac{\gamma}{\NumOfB n}\tr \left[   \bbV_{\bbW^{(2)}}^T \bbW^{(1)}   \left( \bbV_{\bbW^{(1)}} \bbV_{ \bbW^{(1)} }^T -  \bbI_n    \right) \bbW^{(2)} \bbV_{\bbW^{(2)}}  \right]  + O_P \left(  \frac{ \NumOfB [\log (n)]^2 }{ n } \right)  \nonumber \\
& = &  - \frac{\gamma}{\NumOfB n}\tr \left[  \bbV_{\bbW^{(2)}}^T \bbV_{\bbW^{(1)}}^{\perp }  \DiaEig_{\bbW^{(1)}}^{\perp }   \left[ \bbV_{ \bbW^{(1)}}^{\perp } \right]^T  \bbV_{\bbW^{(2)}}  \DiaEig_{\bbW^{(2)}}  \right]  +O_P\left( \frac{ \NumOfB[\log (n)]^2 }{ n } \right) \nonumber.
\end{eqnarray*}
}
On the other hand, Corollary \ref{cor:term hard to handle by second singular value of Erdos Renyi graph} implies 
{ \allowdisplaybreaks
	\begin{eqnarray*}
		&&		O_P\left( \NumOfB \log(n)\right)\\
		&  = & \sqrt{\tr \left[  \bbV_{\bbE^{(1)}}^T \bbV_{\bbW^{(1)}  }^{\perp }  \DiaEig_{\bbW^{(1)}}^{\perp }   \left[ \bbV_{ \bbW^{(1)}}^{\perp } \right]^T  \bbV_{\bbE^{(1)}}   \right] }  \\
		& = &  \left\| \left[ \DiaEig_{\bbW^{(1)}}^{\perp }\right]^{\frac12}   \left[ \bbV_{ \bbW^{(1)}}^{\perp } \right]^T  \bbV_{\bbE^{(1)}}   \right\|_F  = \left\| \left[ \DiaEig_{\bbW^{(1)}}^{\perp }\right]^{\frac12}   \left[ \bbV_{ \bbW^{(1)}}^{\perp } \right]^T  \bbV_{\bbE^{(2)}}   \right\|_F \\
		& = &   \left\| \left[ \DiaEig_{\bbW^{(1)}}^{\perp }\right]^{\frac12}   \left[ \bbV_{ \bbW^{(1)}}^{\perp } \right]^T \cdot \right.\\
		&& \left. \left[\bbV_{\bbE^{(2)}} - \bbV_{\bbW^{(2)} } \Procrustes  \left(   \bbV_{\bbW^{(2)} } , \bbV_{\bbE^{(2)}}  \right)+ \bbV_{\bbW^{(2)} } \Procrustes  \left( \bbV_{\bbW^{(2)} } , \bbV_{\bbE^{(2)}} \right) \right]\right\|_F \\
		& = &  \left\| \left[ \DiaEig_{\bbW^{(1)}}^{\perp }\right]^{\frac12}   \left[ \bbV_{ \bbW^{(1)}}^{\perp } \right]^T   \bbV_{\bbW^{(2)} } \Procrustes \left(  \bbV_{\bbW^{(2)} } , \bbV_{\bbE^{(2)}} \right) \right\|_F + O_P\left(  \sqrt{\frac{ \NumOfB }{ n b_{\cP} }} \cdot \sqrt{\frac{n b_{\cP}}{\NumOfB}} \right) \\
		& = &  \left\| \left[ \DiaEig_{\bbW^{(1)}}^{\perp }\right]^{\frac12}   \left[ \bbV_{ \bbW^{(1)}}^{\perp } \right]^T   \bbV_{\bbW^{(2)} }  \right\|_F + O_P\left( \sqrt{  \NumOfB b_{\cP} } \right) ,
	\end{eqnarray*}
}
where $ \left[ \DiaEig_{\bbW^{(1)}}^{\perp }\right]^{\frac12}  \triangleq \diag \left\{ \left| \sigma_{\NumOfB+1} \left( \bbW^{(1)}\right) \right|^{\frac12} , \ldots,  \left| \sigma_{ n } \left( \bbW^{(1)}\right) \right|^{\frac12} \right\}$; or equivalently,
\begin{eqnarray*}
	&& - \frac{\gamma}{\NumOfB n}\tr \left[  \bbV_{\bbW^{(2)}}^T \bbV_{\bbW^{(1)}}^{\perp }  \DiaEig_{\bbW^{(1)}}^{\perp }   \left[ \bbV_{ \bbW^{(1)}}^{\perp } \right]^T  \bbV_{\bbW^{(2)}}  \DiaEig_{\bbW^{(2)}}  \right] \\
	& = &  - \frac{\gamma}{\NumOfB n}  \left\| \left[ \DiaEig_{\bbW^{(1)}}^{\perp }\right]^{\frac12}   \left[ \bbV_{ \bbW^{(1)}}^{\perp } \right]^T   \bbV_{\bbW^{(2)} }  \right\|_F^2= O_P \left(  \frac{\NumOfB[\log (n)]^2}n\right),
\end{eqnarray*}
hence  by using the same arguments as in Section \ref{subsubsec: orthogonal subspaces square of Frobenius term}, we also finish proof for $\bbT = \left(\bbW^{(2)} \bbW^{(1)} \right)^{ -1 }, \Procrustes  \left( \bbW^{(1)}, \bbW^{(2)}\right)$. The results are summarized in  Lemma \ref{lem:just one step away from final result two sample} in a similar form as  Theorem \ref{thm:asymptotic expansion one view}. 
\begin{lem} \label{lem:just one step away from final result two sample}
	$$ \frac{\left\| \left( \bbV_{\bbW^{(1)}} \bbT -  \bbV_{\bbW^{(2)} }  \right)  \DiaEig_{\bbW^{(2)} } \right\|_F^2 }{ \NumOfB n}=  \frac{ \left\| \left(\gamma \bbW^{(1)} - \bbW^{(2)}  \right) \bbV_{\bbW^{(2)} }  \right\|_F^2}{\NumOfB n}+ O_P \left(  \frac{\NumOfB [\log (n)]^2 }n  \right).
	$$
\end{lem}

Lastly,  we need $\bbV_{\bbE^{(2)}}$ instead of $\bbV_{\bbW^{(2)}}$ and we take advantage of Corollary \ref{cor:bbV bbW also OK rather than bbV bbP}:
{\allowdisplaybreaks
\begin{eqnarray}
&& \frac{ \left\| \left(\gamma \bbW^{(1)} - \bbW^{(2)}  \right) \bbV_{\bbW^{(2)} }  \right\|_F^2}{\NumOfB n}  \label{eq:last obstacle in two views}
\\  & =&  \frac{ \left\| \left(\gamma \bbW^{(1)} -  \bbE^{(2)} + \bbE^{(2)} - \bbW^{(2)}  \right) \bbV_{\bbW^{(2)} }  \right\|_F^2}{\NumOfB n}  \nonumber \\
& = &  \frac{ \left\| \left(\gamma \bbW^{(1)} -  \bbE^{(2)} \right) \bbV_{\bbW^{(2)} }  \right\|_F^2}{\NumOfB n} +  \frac{ \left\| \left(  \bbE^{(2)} - \bbW^{(2)}  \right) \bbV_{\bbW^{(2)} }  \right\|_F^2}{\NumOfB n}  \nonumber \\
&& + \frac2{n\NumOfB}\tr  \left[ \bbV_{\bbW^{(2)} } \left(\gamma \bbW^{(1)} -  \bbE^{(2)} \right)  \left(  \bbE^{(2)} - \bbW^{(2)}  \right) \bbV_{\bbW^{(2)} }  \right], \nonumber 
\end{eqnarray}
}
where in the cross term, $\gamma \bbW^{(1)} -  \bbE^{(2)}$ has  zero mean, and is independent of the rest and the cross term is a linear function of $\gamma \bbW^{(1)} -  \bbE^{(2)}.$
The central limit theorem implies $$ \frac2{n\NumOfB}\tr  \left[ \bbV_{\bbW^{(2)} } \left(\gamma \bbW^{(1)} -  \bbE^{(2)} \right)  \left(  \bbE^{(2)}- \bbW^{(2)}  \right) \bbV_{\bbW^{(2)} }  \right] = O_P \left( \frac1{n\NumOfB}\right).$$ Hence,  using a similar argument as in the proof of Corollary \ref{cor:bbV bbW also OK rather than bbV bbP}, we obtain
\begin{eqnarray*}
	&& \frac{ \left\| \left(\gamma \bbW^{(1)} - \bbW^{(2)}  \right) \bbV_{\bbW^{(2)} }  \right\|_F^2}{\NumOfB n} \\  
	& = &	\frac{ \left\| \left(\gamma \bbW^{(1)} -  \bbE \right) \bbV_{\bbE^{(2)} }  \right\|_F^2}{\NumOfB n} +  \frac{ \left\| \left(  \bbE^{(2)}- \bbW^{(2)}  \right) \bbV_{\bbE^{(2)} }  \right\|_F^2}{\NumOfB n}  +O_P \left( \frac{\NumOfB [\log (n)]^2}{n}\right) \\
	& = &  \frac{ \left\| \left(\gamma \bbW^{(1)} - \bbW^{(2)}  \right) \bbV_{\bbE^{(2)}}   \right\|_F^2}{\NumOfB n} +O_P \left( \frac{\NumOfB [\log (n)]^2}{n}\right),
\end{eqnarray*}
 we  then  finish  the  proof of Theorem \ref{thm:asymptotic expansion two views}.

\section{Proofs of central limit theorems via moment matching}

\subsection{Proof of Theorem \ref{thm:central limit theorem for frobenius norm} on the asymptotic distribution of the singular subspace distance} 
\label{subsec: proof of CLT for customized normalization}

Using the same techniques as (2.1.46) of \cite{anderson2010introduction}, that is, Section 3.3.5 ``the moment problem" in \cite{durrett2010probability}, it suffices to verify that 
\begin{equation}
\lim_{ n \to \infty}\expc \left( \frac{ W_n - \expc W_n }{\sqrt{ \Var W_n}} \right)^j = \left\{ \begin{array}{cc} 0 , & \text{ if }j\text{ is odd};\\ (j-1)!! , & \text{ if }j\text{ is even}.\\ 
\end{array}\right.
\label{eq:moment method verification}
\end{equation}
where right hand side of \eqref{eq:moment method verification} coincides with the moments of the Gaussian distribution $\Phi$.

Same as Theorem 2.1.31 in \cite{anderson2010introduction},  the first step is to evaluate the  mean and variance of $$W_{n} \triangleq \left\| \left(\bbW_n - \bbE_n \right) \bbZ_n \left( \bbZ_n^T \bbZ_n\right)^{- \frac12} \right\|_F^2 .$$ For the sake of convenience, denote $\bbX_n \triangleq \bbW_n - \bbE_n$ with each entry with zero mean.  The $ik$-th entry of $\left(\bbW_n - \bbE_n \right) \bbZ_n \left( \bbZ_n^T \bbZ_n\right)^{- \frac12}$ is 
\begin{equation}
\frac1{ \sqrt{\# \cC_{k} } }  \sum_{t \in \cC_k} x_{it}, i \in [n],k\in \left[\NumOfB\right]
\label{eq:ikth entry W-P Z inv(sqrt(ZtZ))}
\end{equation}

Hence,{ \allowdisplaybreaks
	\begin{eqnarray}
&& \expc W_n \label{eq:expectation dominant term} \\		& = & \sum_{i =1}^n \sum_{k =1 }^{\NumOfB} \frac1{\# \cC_{k} } \expc  \left( \sum_{t \in \cC_k} x_{it} \right)^2 = \sum_{i , k =1 }^{\NumOfB}  \frac1{\# \cC_{k} } \sum_{r \in \cC_i}  \expc  \left( \sum_{t \in \cC_k} x_{rt}^2  \right) \nonumber\\
		& = & \sum_{i , k =1 }^{\NumOfB}  \frac1{\# \cC_{k} } \sum_{r \in \cC_i} \sum_{t \in \cC_k} \sigma_{ik}^2 - \sum_{  i = 1}^{\NumOfB} \frac{\sigma_{ii}^2}{\# \cC_i}\cdot \# \cC_i  =  \sum_{i , k =1 }^{\NumOfB}  \# \cC_i  \sigma_{ik}^2 \sum_{  i = 1}^{\NumOfB}  \sigma_{ii}^2 \nonumber  \\
		&\xlongequal{} &  (\NumOfB -1) n \sigma_Q^2  +  n\sigma_P^2 - K \sigma_{\cP}^2 = (\NumOfB -1) n \sigma_Q^2  +  n\sigma_P^2 + o\left(n \sigma_{\cP}^2\right) , \nonumber
	\end{eqnarray}
	based on \mbox{Assumption \ref{assum:homogeneous weighted SBM}}}
where  the negligible term $o\left(n \sigma_{\cP}^2\right)$ is  due to the fact that $x_{ii} = 0$ rather than $x_{ii} \sim \cP$. Due to the same reason, we may treat $x_{ii} \sim \cP$ in later calculations for the sake of simplicity.

In term of the  variance, only the terms with each edge appearing at least twice are relevant, 
{ \allowdisplaybreaks
	\begin{eqnarray}
		&& \Var W_n = \Var \left\| \left(\bbW_n - \bbE_n \right) \bbZ_n \left( \bbZ_n^T \bbZ_n\right)^{- \frac12}\right\|_F^2 \label{eq:variance SBM}  \\
		& = & \expc \left[ \sum_{i , k =1 }^{\NumOfB}  \frac1{\# \cC_{k} } \sum_{r \in \cC_i}    \left( \sum_{t \in \cC_k} x_{rt} \right)^2 - \expc\left[ \sum_{i , k =1 }^{\NumOfB}  \frac1{\# \cC_{k} } \sum_{r \in \cC_i}    \left( \sum_{t \in \cC_k} x_{rt} \right)^2\right]\right]^2 \nonumber \\
		& = & 	\expc \left[ \sum_{i , k =1 }^{\NumOfB}  \frac1{\# \cC_{k} } \sum_{r \in \cC_i}  \left( \sum_{ t \in \cC_k} \left[ x_{r t }^2 - \expc \left(x_{r t }^2 \right) \right] + 2 \sum_{s<t \in \cC_k} x_{rs}x_{ r t} \right)  \right]^2 \nonumber \\
		& = & 	\expc \left\{  \sum_{i , k =1 }^{\NumOfB}  \frac1{\# \cC_{k} } \sum_{r \in \cC_i}  \sum_{ t \in \cC_k} \left[ x_{r t }^2 - \expc \left(x_{r t }^2 \right) \right] \right\}^2  	+ 4 \expc \left[ \sum_{i , k =1 }^{\NumOfB}  \frac1{\# \cC_{k} } \sum_{r \in \cC_i}   \sum_{s<t \in \cC_k} x_{rs}x_{ r t}   \right]^2 \nonumber \\
		& = & 	4 \expc \left\{  \sum_{i < k =1 }^{\NumOfB}  \frac1{\# \cC_{k} } \sum_{r \in \cC_i}  \sum_{ t \in \cC_k} \left[ x_{r t }^2 - \expc \left(x_{r t }^2 \right) \right] \right\}^2  \nonumber \\
		&&+ 	\expc \left\{  \sum_{i =1 }^{\NumOfB}  \frac1{\# \cC_{i} } \sum_{r, t \in \cC_i} \left[ x_{r t }^2 - \expc \left(x_{r t }^2 \right) \right] \right\}^2 \nonumber   \\
		&& + \sum_{i , k = 1}^{\NumOfB }  \frac4{ \left(\# \cC_k\right)^2 } \sum_{r \in \cC_i }  \sum_{s<t \in \cC_k} \left[x_{rs}^2  - \expc (x_{rs}^2 )\right] \left[x_{rt}^2 - \expc  (x_{rt}^2    )\right] \nonumber \\
		& = & 	 \expc \left\{  \sum_{i < k =1 }^{\NumOfB} \left( \frac1{\# \cC_{i} }+\frac1{\# \cC_{k} } \right) \sum_{r \in \cC_i}  \sum_{ t \in \cC_k} \left[ x_{r t }^2 - \expc \left(x_{r t }^2 \right) \right] \right\}^2 \nonumber  \\
		&& + 	4 \expc \left\{  \sum_{i =1 }^{\NumOfB}  \frac1{\# \cC_{i} } \sum_{r< t \in \cC_i} \left[ x_{r t }^2 - \expc \left(x_{r t }^2 \right) \right] \right\}^2 \nonumber  \\
		&&  + \sum_{i  = 1}^{\NumOfB }  \# \cC_i   \left[ 2 \left( \NumOfB - \sum_{k=1}^{ \NumOfB} \frac1{\# \cC_k} \right) \sigma_Q^4 +  \frac{ 2 (\# \cC_i -1)}{\# \cC_i} (\sigma_P^4 - \sigma_Q^4) \right] \nonumber \\
		& = & 	\sum_{i < k =1 }^{\NumOfB}  \left( \frac1{\# \cC_{i} }+\frac1{\# \cC_{k} } \right)^2 \sum_{r \in \cC_i}  \sum_{ t \in \cC_k}\expc \left[ x_{r t }^2 - \expc \left(x_{r t }^2 \right) \right]^2    \nonumber \\
		&&+ 	   \sum_{i =1 }^{\NumOfB}  \frac4{ (\# \cC_{i})^2 } \sum_{r< t \in \cC_i} \expc\left[ x_{r t }^2 - \expc \left(x_{r t }^2 \right) \right]^2  + 2n\left(\NumOfB  -  \sum_{k=1}^{ \NumOfB} \frac{ 1  }{\# \cC_k} \right) \sigma_Q^4 \nonumber \\
		&& + 2\left(n - \NumOfB \right) (\sigma_P^4 - \sigma_Q^4)  	\sum_{i < k =1 }^{\NumOfB}  \left(2 + \frac{\# \cC_{i}^2 + \# \cC_{k}^2}{\# \cC_{i} \# \cC_{k} } \right) r_Q^4 + 	   \sum_{i =1 }^{\NumOfB}  \left(2- \frac2{ \# \cC_{i} }\right) r_P^4 \nonumber \\
		&&  + 2n \left(\NumOfB -  \sum_{k=1}^{ \NumOfB} \frac{ 1  }{\# \cC_k}\right) \sigma_Q^4 + 2\left(n - \NumOfB \right) (\sigma_P^4 - \sigma_Q^4)\nonumber  \\
		&= & n r_Q^4	\sum_{k =1 }^{\NumOfB} \frac{1}{\# \cC_{k} }  + 	   \sum_{i =1 }^{\NumOfB}  \left(2- \frac2{ \# \cC_{i} }\right) r_P^4   + 2 n \left(\NumOfB -  \sum_{k=1}^{ \NumOfB} \frac{ 1  }{\# \cC_k}\right) \sigma_Q^4  \nonumber \\ && + 2\left(n - \NumOfB \right) (\sigma_P^4 - \sigma_Q^4) \nonumber \\
		& = & 2 n \NumOfB  \left( \sigma_Q^4 + \frac{\sigma_P^4 - \sigma_Q^4 }{\NumOfB}  \right) + O\left( \NumOfB^2 \right),\nonumber
	\end{eqnarray}
	based on \mbox{Assumption \ref{assum:size of each block}}.	}

To conclude the proof, we need to  show
{ \allowdisplaybreaks
	\begin{eqnarray*}
		&& \lim_{ n \to \infty}  \expc \left( \frac{ W_n - \expc W_n }{\Var W_n} \right)^j   =  	\lim_{ n \to \infty}\expc \left( \frac{ W_n - \expc W_n }{ \sqrt{ 2 \NumOfB n \sigma_Q^4 + 2  n \left( \sigma_P^4 - \sigma_Q^4 \right) }  } \right)^j  \\
		& = & 	\lim_{ n \to \infty} \expc \left[\frac{ \dps   \sum_{i, k =1 }^{\NumOfB} \frac1{\# \cC_{k} }  \sum_{ r \in \cC_i} \left( \sum_{ t \in \cC_k} x_{rt }\right)^2 - \expc W_n}{\sqrt{ 2 \NumOfB n \sigma_Q^4 + 2  n \left( \sigma_P^4 - \sigma_Q^4 \right) }   }\right]^j = \left\{ \begin{array}{cc} 0 , & \text{ if }j\text{ is odd};\\ (j-1)!! , & \text{ if }j\text{ is even}.\\ 
		\end{array}\right.. 
	\end{eqnarray*}
}

\subsubsection{Limit calculations} 
Consider the enumerator,
{\allowdisplaybreaks
	\begin{eqnarray}
	&& \E\left[  \left(W_n - \E W_n\right)^j \right] \nonumber \\
	& = &  \E \left[   \sum_{i, k =1 }^{\NumOfB} \frac1{\# \cC_{k} }  \sum_{ r \in \cC_i} \left( \sum_{ t \in \cC_k} x_{rt }^2 + 2  \sum_{ s< t \in \cC_k} x_{rs }x_{rt}\right) -\expc  W_n  \right]^j \nonumber  \\
	&  = & \E \left[     \sum_{i =1 }^{\NumOfB} \frac1{\# \cC_{i} }  \sum_{ r \in \cC_i} \left( \sum_{ t \in \cC_i } \left(x_{rt }^2  - \expc \left[x_{rt}^2 \right] \right)+ 2  \sum_{ s< t \in \cC_i} x_{rs }x_{rt}\right)   \right. \nonumber \\
	& & \left. +   \sum_{i \ne  k =1 }^{\NumOfB} \frac1{\# \cC_{k} }  \sum_{ r \in \cC_i} \left( \sum_{ t \in \cC_k} \left(x_{rt }^2  - \expc \left[x_{rt}^2 \right] \right)+ 2  \sum_{ s< t \in \cC_k} x_{rs }x_{rt}\right)  \right]^j  \nonumber \\
	& = &\E \left\{    \sum_{i  =1 }^{\NumOfB}   \frac1{\# \cC_{i} } \sum_{r \in \cC_i} \sum_{ t \in \cC_i} \left(x_{rt }^2  - \expc \left[x_{rt}^2 \right] \right) \right. \label{eq:denominator in moment matching} \\
	&&  +  \sum_{i < k = 1 }^{\NumOfB }  \left(\frac1{\# \cC_{i } }+  \frac1{\# \cC_{k} }\right) \sum_{r \in \cC_i}\sum_{ t \in \cC_k} \left(x_{rt }^2  - \expc \left[x_{rt}^2 \right] \right)    \nonumber \\
	& & \left. +  \sum_{i =1 }^{\NumOfB} \frac2{\# \cC_{i} }  \sum_{ r \in \cC_i}  \sum_{ s< t \in \cC_i } x_{rs }x_{rt}  +   \sum_{i \ne  k =1 }^{\NumOfB} \frac2{\# \cC_{k} }  \sum_{ r \in \cC_i}  \sum_{ s< t \in \cC_k} x_{rs }x_{rt}  \right\}^j. \nonumber  \end{eqnarray}
}
Now it is natural to introduce ``words" and ``sentences".

\subsubsection{Words, sentences and their graphs}
We give a very brief introduction to words, sentences and their equivalence classes essential for the combinatorial analysis of random matrices. The definitions are used in  \cite{anderson2006clt}, Section 2.1, although we have more weights here, and we have an $n \times \NumOfB$ rectangular matrix.

\begin{defn}[Words]
	\label{defn:words} Given the set of letters $ [n] =\{1,2, \ldots, n\}$.
	Set of words are of the kind $x_{rt}^2 - \expc \left[ x_{rt}^2 \right], r,t \in [n]$ (two letters) or $x_{rs}x_{rt}, s,r,t \in [n]$ (three letters). 
	
\end{defn}

The interior of the  last representation in \eqref{eq:denominator in moment matching} has each word to be different and weights of words are all of order $\Theta\left( \frac{\NumOfB}{n }\right)$. Further,  the sum of weights for words of type $x_{rt }^2  - \expc \left[x_{rt}^2 \right] $ is $\frac{\NumOfB n}{2} + n$ while the sum of weights for words of type $x_{rs }x_{rt} $ is $n(n-\NumOfB -1)$:
\begin{eqnarray}
\sum_{i  =1 }^{\NumOfB}   \frac1{\# \cC_{i} } \sum_{r \in \cC_i} \sum_{ t \in \cC_i}1 +  \sum_{i < k = 1 }^{\NumOfB }  \left(\frac1{\# \cC_{i } }+  \frac1{\# \cC_{k} }\right) \sum_{r \in \cC_i}\sum_{ t \in \cC_k} 1 & = & \frac{\NumOfB n}{2} + n, \label{eq:ignore words of first type}  \\
\sum_{i =1 }^{\NumOfB} \frac2{\# \cC_{i} }  \sum_{ r \in \cC_i}  \sum_{ s< t \in \cC_i } 1 + \sum_{i \ne  k =1 }^{\NumOfB} \frac2{\# \cC_{k} }  \sum_{ r \in \cC_i}  \sum_{ s< t \in \cC_k}1 & = & n(n-\NumOfB -1). \nonumber
\end{eqnarray}
which heuristically implies that we can ignore words of type $x_{rt }^2  - \expc \left[x_{rt}^2 \right]$. This  argument  appears in the procedure of evaluating $\Var W_n$ as well.

\ignore{On the other hand, our version is a simpler version since our \texttt{letter-word-sentence} majorly only involves \texttt{word} consisting of 1 or 2 letters, i.e. of length 2. Letters are also called edge in the graph $\cG= \langle [n],\bbX\rangle$ (while set of nodes is $[n]$, $
	\bbX$ is the matrix of (centered) weights).}

\begin{defn}[Sentences]
	\label{defn:sentences}
	A sentence $\sentence$ is an ordered collection of  words $\word_1, \word_2, \ldots, \word_m$, at least one word long. 
	
\end{defn}
\begin{defn}[Weak CLT sentences] \label{defn:weak CLT sentences}
	A sentence $\sentence= [\word_i]_{i=1}^m$ is called a weak CLT sentence if the following hold
	\begin{enumerate}
		\item for each edge of the graph, $\sentence$ visits at least twice or does not visit it (that is, no such edge that $\sentence$ only visits once);
		\item For each $i\in [m]$, there is another $j \in [m] \setminus \{i \}$ such that $\word_i, \word_j$ have at least one edge in common. 
	\end{enumerate}
\end{defn}

Since we  deal with linear spectral  statistics, our definition of ``weak CLT sentences" is different from \cite{anderson2006clt,banerjee2008model,anderson2010introduction} in the sense that we have no ``closed words". 
\begin{defn}[Graph associated with words, sentences]   Let $G_{\word} = \langle V_{\word}, E_{\word} \rangle$ be the (undirected) graph associated with word $\word$. For word $\word = x_{rt }^2  - \expc \left[x_{rt}^2 \right] $, set $V_{\word} = \{r,t\}$ and multiset (rather than a set) $E_{\word} = \{ \{r,t\},  \{r,t\} \}$ where edge appears twice; for word $x_{rs }x_{rt}$, $V_{\word} = \{r,s,t\}$ and $E_{\word} = \{ \{r,s\}, \{r,t \} \}$.
	
	Let $G_{\sentence} = \langle V_{\sentence}, E_{\sentence} \rangle$ be the graph of a sentence $\sentence = \left( \word_1, \word_2 ,  \ldots , \word_j\right) $ where $\dps V_{\sentence} = \bigcup_{i=1}^j V_{\word_i} \subset [n]$ is the set of all letters, and $E_{\sentence} $ is (multiset) union of $E_{\word_i}, i \in [j]$; by multiset union, we mean we keep duplicates since each edge may appear several times.
\end{defn}


Finally, analogous to (2.1.49) in \cite{anderson2010introduction},  we re-state Lemma 4.3 in \cite{banerjee2018contiguity} or lemma A.5 in \cite{banerjee2017optimal} but focus only on our scenario:
\begin{lem} \label{lem:set of CLT sentences cardinality}
	Let $ \cA_{j,t}^{n}$ be the set of weak CLT sentences $\sentence = \left(\word_1 , \word_2 ,\ldots , \word_j \right)$ such that $\# V_{\sentence} = t$ and the letter set is $[n]$. Then
	\begin{equation}
	\#  \cA_{j,t}^n \le 8^j n^t (3C_1)^{C_2 j} (3j )^{3(3j -2t)},
	\end{equation}
	where $C_1, C_2 >0$ are numeric constants.
\end{lem}
$ \cA_{j,t}^{n}$ is related to (2.1.49) in \cite{anderson2010introduction} and (4.7) in \cite{kemp2013math} but is different in the sense that we do not define equivalent classes. Following (2.1.48) and (2.1.50) in \cite{anderson2010introduction}, we can turn
\eqref{eq:denominator in moment matching} into 
{\allowdisplaybreaks
	\begin{eqnarray*}
		&& \expc \left[   \left(W_n -\expc  W_n\right)^j  \right] \\
		& = &\E \left\{    \sum_{i  =1 }^{\NumOfB}   \frac1{\# \cC_{i} } \sum_{r \in \cC_i} \sum_{ t \in \cC_i} \left(x_{rt }^2  - \expc \left[x_{rt}^2 \right] \right) \right. \\
		&&  +  \sum_{i < k = 1 }^{\NumOfB }  \left(\frac1{\# \cC_{i } }+  \frac1{\# \cC_{k} }\right) \sum_{r \in \cC_i}\sum_{ t \in \cC_k} \left(x_{rt }^2  - \expc \left[x_{rt}^2 \right] \right)    \\
		& & \left. +  \sum_{i =1 }^{\NumOfB} \frac2{\# \cC_{i} }  \sum_{ r \in \cC_i}  \sum_{ s< t \in \cC_i } x_{rs }x_{rt}  +   \sum_{i \ne  k =1 }^{\NumOfB} \frac2{\# \cC_{k} }  \sum_{ r \in \cC_i}  \sum_{ s< t \in \cC_k} x_{rs }x_{rt}  \right\}^j \\
		& = & \sum_{t = 1}^{2j} \sum_{  \sentence = \left( \word_1 , \word_2 , \ldots , \word_j \right) \in \cA_{j,t}^{n}} c(\sentence)\cdot \E \left[\prod_{i =1}^j \word_i \right] \\
		& &\xlongequal[\mbox{ at least twice\ignore{Exercise 4.3.1 in \cite{kemp2013math}}}]{\mbox{each edge is visited}} 
		\sum_{t = 1}^{j} \sum_{  \sentence = \left( \word_1 , \word_2 , \ldots , \word_j \right) \in \cA_{j,t}^{n}} c(\sentence)\cdot \E \left[\prod_{i =1}^j \word_i \right],
	\end{eqnarray*}
}
where $\dps c(\sentence) = \prod_{i=1}^jc(\word_i)$ is multiplication of coefficients in front of words $\word_i$ in \eqref{eq:denominator in moment matching}, that is, multiplication of several coefficients (duplicates allowed):
$ \frac1{\# \cC_{i} },  \left(\frac1{\# \cC_{i } }+  \frac1{\# \cC_{k} }\right), \frac2{\# \cC_{i} },\frac2{\# \cC_{k} }$.\ignore{; further, it is worth noticing that 
$$
\frac{\dps \frac{n!}{(n-t)!} \cdot  \sum_{  \sentence = \left( \word_1 , \word_2 , \ldots , \word_j \right) \in \cA_{j,t}^{n}} c(\sentence)\cdot \E \left[\prod_{i =1}^j \word_i \right]}{\dps \frac{n! \NumOfB^j}{(n-t)! n^{j }}\cdot  \sum_{  \sentence = \left( \word_1 , \word_2 , \ldots , \word_j \right) \in \cA_{j,t}^{n}} \E \left[\prod_{i =1}^j \word_i \right]},
$$
should be \textcolor{red}{upper and lower bounded by positive constants} that are independent of $n$. } Lemma \ref{lem:set of CLT sentences cardinality} implies
\begin{eqnarray*}
\sum_{  \sentence = \left( \word_1 , \word_2 , \ldots , \word_j \right) \in \cA_{j,t}^{n}} c(\sentence)\cdot \E \left[\prod_{i =1}^j \word_i \right] & \le & \frac{\NumOfB^j }{n^j}\# \cA_{j,t}^{n} \cdot \max_{\cS \in \cA_{j,t}^{n}}\left\{ \frac{n^j c(\sentence)}{\NumOfB^j} \E \left[\prod_{i =1}^j \word_i \right]  \right\}  \\
& \le & C\max_{\cS \in \cA_{j,t}^{n}}\left\{ \E \left[\prod_{i =1}^j \word_i \right]  \right\} \cdot \frac{\NumOfB^j }{n^j}\# \cA_{j,t}^{n} ,
\end{eqnarray*}
goes to $0$ as $n \to \infty$ as long as $t<j$; $C$ is a constant independent of $n$.
As a result, 
\begin{eqnarray}
&& \lim_{ n \to \infty} \expc \left[   \left(\frac{W_n -\expc  W_n}{\sqrt{\Var W_n }}\right)^j  \right] \\
& = &  \left\{ \begin{array}{cc}
0 , & \text{if } j\text{ is odd}; \\ \dps \lim_{ n \to \infty}  \left( \Var W_n\right)^{-\frac{j}2} \sum_{  \sentence = \left( \word_1 , \word_2 , \ldots , \word_j \right) \in \cA_{j,j}^{n}} c(\sentence)\cdot \E \left[\prod_{i =1}^j \word_i \right], & \text{if } j\text{ is even}. \nonumber
\end{array}\right.
\end{eqnarray}

For $j$ even, first thing that is analogous to is that $ \sentence  \in \cA_{j,j}^{n}$ can be viewed as an ordered sequence of distinct $\word_1', \ldots, \word_{ \frac{j}2}'$, each of which appears twice in $\sentence$ ($\word_i'$ does not necessarily have to be $i$th word in $\sentence$). 
{\allowdisplaybreaks
	\begin{eqnarray*}
		&&	 \lim_{ n \to \infty} \expc \left[   \left(\frac{W_n -\expc  W_n}{\sqrt{\Var W_n }}\right)^j  \right]  \\
		& = &  \lim_{ n \to \infty}  \left( \Var W_n\right)^{-\frac{j}2}  \sum_{  \sentence = \left( \word_1 , \word_2 , \ldots , \word_j \right) \in \cA_{j,j}^{n}} c(\sentence)\cdot  \expc \left[\prod_{i =1}^j \word_i \right] \\
		& = &  \lim_{ n \to \infty}  \left( \Var W_n\right)^{-\frac{j}2}  \sum_{\substack{ \sentence \in {\cA}_{j,j}^{n}\mbox{ is an ordered }\\ \mbox{sequence of distinct }\word_1', \ldots, \word_{ \frac{j}2}' \\ \mbox{ each appears twice}}}\prod_{i =1}^{\frac{j}2} \left[ c\left( \word_i' \right)   \right]^2   \expc \left[ \left( \word_i' \right)^2 \right]
	\end{eqnarray*}
}

It remains to calculate $$\dps \sum_{\substack{ \sentence \in {\cA}_{j,j}^{n}\mbox{ is an ordered }\\ \mbox{sequence of distinct }\word_1', \ldots, \word_{ \frac{j}2}' \\ \mbox{ each appears twice}}}\prod_{i =1}^{\frac{j}2} \left[ c\left( \word_i' \right)   \right]^2   \expc \left[ \left( \word_i' \right)^2 \right].$$ 
Similar to (2.1.52) of \cite{anderson2010introduction}, we introduce permutation  $\pi: [j] \to [j]$, all of whose cycles have length $2$ (that is, a matching), such that the connected components of $G_{\sentence}$ are the graphs $\left\{G_{\left( \word_i, \word_{ \pi(i)}\right) } \right\}$; letting $\Sigma_j$ denote the collection of all possible matchings. In this sense, the way we determine $\cS$ is to determine $\pi\in \Sigma_j$ and determine $\frac{j}2$ distinct words $\word_1'. \ldots, \word_{ \frac{j}2}'$; Dyck path \citep{kemp2013math} may be an alternative structure to explain the procedure of determination. One thus obtains that for $j$ even,
{ \allowdisplaybreaks
	\begin{eqnarray*}
&&	 \lim_{ n \to \infty} \expc \left[   \left(\frac{W_n -\expc  W_n}{\sqrt{\Var W_n }}\right)^j  \right] \\
&=&  \lim_{ n \to \infty}  \left( \Var W_n\right)^{-\frac{j}2}  \sum_{  \sentence 	= \left( \word_1 , \word_2 , \ldots , \word_j \right) \in \cA_{j,j}^{n}} c(\sentence)\cdot  \expc \left[\prod_{i =1}^j \word_i \right] \\
& = &  \lim_{ n \to \infty}  \left( \Var W_n\right)^{-\frac{j}2}  \sum_{\substack{ \sentence \in {\cA}_{j,j}^{n}\mbox{ is an ordered }\\ \mbox{sequence of distinct }\word_1', \ldots, \word_{ \frac{j}2}' \\ \mbox{ each appears twice}}}\prod_{i =1}^{\frac{j}2} \left[ c\left( \word_i' \right)   \right]^2   \expc \left[ \left( \word_i' \right)^2 \right] \\
& = &  \lim_{ n \to \infty}  \left( \Var W_n\right)^{-\frac{j}2}  \sum_{\pi \in \Sigma_j } \sum_{\word_1', \ldots, \word_{ \frac{j}2}'\mbox{distinct}} \prod_{i =1}^{\frac{j}2} \left[ c\left( \word_i' \right)   \right]^2   \expc \left[ \left( \word_i' \right)^2 \right] \\
& = &  \lim_{ n \to \infty}  \left( \Var W_n\right)^{-\frac{j}2}  \# \Sigma_j \cdot \sum_{\word_1', \ldots, \word_{ \frac{j}2}'\mbox{distinct}} \prod_{i =1}^{\frac{j}2} \left[ c\left( \word_i' \right)   \right]^2   \expc \left[ \left( \word_i' \right)^2 \right] \\
& = &  (j-1)!! \cdot \lim_{ n \to \infty}  \left( \Var W_n\right)^{-\frac{j}2}  \sum_{\word_1', \ldots, \word_{ \frac{j}2}'\mbox{distinct}} \prod_{i =1}^{\frac{j}2} \left[ c\left( \word_i' \right)   \right]^2   \expc \left[ \left( \word_i' \right)^2 \right],
\end{eqnarray*}
} 

Finally, we propose and apply a novel combinatorial technique  to evaluate 
$$
\lim_{ n \to \infty}  \left( \Var W_n\right)^{-\frac{j}2}  \sum_{\word_1', \ldots, \word_{ \frac{j}2}'\mbox{distinct}} \prod_{i =1}^{\frac{j}2} \left[ c\left( \word_i' \right)   \right]^2   \expc \left[ \left( \word_i' \right)^2 \right]
	$$
that does not appear in \cite{anderson2010introduction,kemp2013math}. The technique  is just to apply the form of \eqref{eq:denominator in moment matching} and a procedure of calculating $\Var W_n$ to give a sufficient approximation as $n \to \infty$ of 
$$\dps \sum_{\word_1', \ldots, \word_{ \frac{j}2}'\mbox{distinct}} \prod_{i =1}^{\frac{j}2} \left[ c\left( \word_i' \right)   \right]^2   \expc \left[ \left( \word_i' \right)^2 \right]$$. The the approximation is just 
\begin{eqnarray*}
	& &\E \left\{    \sum_{i  =1 }^{\NumOfB}   \frac1{\left(\# \cC_{i}\right)^2 } \sum_{r \in \cC_i} \sum_{ t \in \cC_i} \expc\left(x_{rt }^2  - \expc \left[x_{rt}^2 \right] \right)^2 \right. \\
&&  +  \sum_{i < k = 1 }^{\NumOfB }  \left(\frac1{\# \cC_{i } }+  \frac1{\# \cC_{k} }\right)^2 \sum_{r \in \cC_i}\sum_{ t \in \cC_k} \expc\left(x_{rt }^2  - \expc \left[x_{rt}^2 \right] \right)^2    \\
& & \left. +  \sum_{i =1 }^{\NumOfB} \frac4{\left(\# \cC_{i} \right)^2 }  \sum_{ r \in \cC_i}  \sum_{ s< t \in \cC_i } \expc x_{rs }^2x_{rt}^2  +   \sum_{i \ne  k =1 }^{\NumOfB} \frac4{\left(\# \cC_{k}\right)^2 }  \sum_{ r \in \cC_i}  \sum_{ s< t \in \cC_k} \expc x_{rs }^2x_{rt}^2  \right\}^{\frac{j}2},
\end{eqnarray*}
which coincidentally can be further simplified by  calculating $\Var W_n$. As a result,
{\allowdisplaybreaks
	\begin{eqnarray*}
&& \lim_{ n \to \infty} \expc \left[   \left(\frac{W_n -\expc  W_n}{\sqrt{\Var W_n }}\right)^j  \right]  \\
& = &  (j-1)!! \cdot \lim_{ n \to \infty}  \left( \Var W_n\right)^{-\frac{j}2}  \sum_{\word_1', \ldots, \word_{ \frac{j}2}'\mbox{distinct}} \prod_{i =1}^{\frac{j}2} \left[ c\left( \word_i' \right)   \right]^2   \expc \left[ \left( \word_i' \right)^2 \right] \\
& = &  (j-1)!! \cdot \lim_{ n \to \infty}  \left( \Var W_n\right)^{-\frac{j}2}  \E \left\{    \sum_{i  =1 }^{\NumOfB}   \frac1{\left(\# \cC_{i}\right)^2 } \sum_{r \in \cC_i} \sum_{ t \in \cC_i} \expc\left(x_{rt }^2  - \expc \left[x_{rt}^2 \right] \right)^2 \right. \\
&&  +  \sum_{i < k = 1 }^{\NumOfB }  \left(\frac1{\# \cC_{i } }+  \frac1{\# \cC_{k} }\right)^2 \sum_{r \in \cC_i}\sum_{ t \in \cC_k} \expc\left(x_{rt }^2  - \expc \left[x_{rt}^2 \right] \right)^2    \\
& & \left. +  \sum_{i =1 }^{\NumOfB} \frac4{\left(\# \cC_{i} \right)^2 }  \sum_{ r \in \cC_i}  \sum_{ s< t \in \cC_i } \expc x_{rs }^2x_{rt}^2  +   \sum_{i \ne  k =1 }^{\NumOfB} \frac4{\left(\# \cC_{k}\right)^2 }  \sum_{ r \in \cC_i}  \sum_{ s< t \in \cC_k} \expc x_{rs }^2x_{rt}^2  \right\}^{\frac{j}2} \\
& =&  (j-1)!! \left\{ \lim_{ n \to \infty}  \left( \Var W_n\right)^{-\frac{j}2} \right.\\
&& \left. \cdot 	\left\{\expc \left[ \sum_{i , k =1 }^{\NumOfB}  \frac1{\# \cC_{k} } \sum_{r \in \cC_i}  \left( \sum_{ t \in \cC_k} \left[ x_{r t }^2 - \expc \left(x_{r t }^2 \right) \right] + 2 \sum_{s<t \in \cC_k} x_{rs}x_{ r t} \right)  \right]^2\right\}^{\frac{j}2} \right\} \\
& = &  (j-1)!! \lim_{ n \to \infty} \left[ \left( \Var W_n\right)^{-\frac{j}2} \cdot \left( \Var W_n\right)^{\frac{j}2}  \right] = (j-1)!!.
\end{eqnarray*}
}

and as a result,  \eqref{eq:moment method verification} holds.
\subsection{Proof of Theorem \ref{thm:asym distri testing statistics AJ}} We ignore here since it is exactly same as proof of Theorem \ref{thm:central limit theorem for frobenius norm} in Section \ref{subsec: proof of CLT for customized normalization} above.
\ignore{	\begin{lem} \label{lem:combinatoric summation}
		\begin{equation}\label{eq:non repetitive summation}
		\sum_{i <j=1 }^{ \NumOfB } (\# \cC_{i }  +\# \cC_{j } )  = (\NumOfB - 1)n.
		\end{equation}
		\begin{equation}
		\sum_{  j_1 + \ldots + j_{K}  = N }             \frac{1 }{\prod_{ m = 1 }^{K } j_m! }  = \frac{K^{N}}{N !} . \label{eq:combinatoric summation}
		\end{equation}
	\end{lem}
	\begin{proof} \eqref{eq:non repetitive summation} is easy. As for \eqref{eq:combinatoric summation}, just use mathematical induction on $K = 1,2,\ldots$.
	\end{proof}

}

\section{Mean for $\sin \Theta$ distance in Frobenius norm}
\label{sec:traditional sine theta distance}
This section evaluates mean of square of $\sin \Theta$ distance in Frobenius norm \eqref{eq:sine theta distance in Frobenius norm} under the Assumption \ref{assum:size of each block}, \ref{assum:homogeneous weighted SBM}, \ref{assum:region of interest Renyi divergence}. The aim of clarifying this mean is to argue that multiplier $\Lambda_{\bbW^{(2)}}$ can simplify the calculation of mean and variance in two-sample test statistic \eqref{eq:test statistic}.

\begin{thm}[Mean for the square of $\sin \Theta$ distance in Frobenius norm]
	\label{theorem:mean for eigenspace distance}

	Suppose the Assumption \ref{assum:size of each block}, \ref{assum:homogeneous weighted SBM}, \ref{assum:region of interest Renyi divergence} hold. As for  $\sin \Theta$ distance \eqref{eq:sine theta distance in Frobenius norm}, $ \left\| \sin \Theta \left( \bbV_{\bbW_n},\bbV_{\bbE_n} \right) \right\|_F$ of $\bbV_{\bbW_n}$ observed singular components, and $\bbV_{\bbE_n}$, singular components of $\bbE_n$,
	we have
	\begin{eqnarray*}
		&& \expc \left\| \sin \Theta \left( \bbV_{\bbW_n},\bbV_{\bbE_n} \right) \right\|_F^2 \\
		&= &  \frac{\NumOfB^3}n \cdot  \frac{ (b_{\cQ}\NumOfB +b_{\cP} -2b_{\cQ})^2 - b_{\cQ}^2   }{ (b_{\cP}-b_{\cQ})^2 (b_{\cQ}\NumOfB +b_{\cP} -b_{\cQ})^2}  \zeta (2)  \sigma_{\cQ}^2 \\
		&& + \frac{\NumOfB^2 }n \cdot \frac{  \NumOfB^2 b_{\cQ}^2 \left[ \zeta(1)  \right]^2  }{ (b_{\cP}-b_{\cQ})^2 (b_{\cQ}\NumOfB +b_{\cP} -b_{\cQ})^2}  \sigma_{\cQ}^2 \\
		&& +   \frac{ \NumOfB^2 }n \cdot\frac{ (b_{\cQ}\NumOfB +b_{\cP} -2b_{\cQ})^2 + \left( \NumOfB - 1\right) b_{\cQ}^2  }{ (b_{\cP}-b_{\cQ})^2 (b_{\cQ}\NumOfB +b_{\cP} -b_{\cQ})^2} \zeta (1)\left( \sigma_{\cP}^2 - \sigma_{\cQ}^2 \right) 
		,
	\end{eqnarray*}
	where for the sake of simplicity, we define 
	\begin{equation}
	\zeta (s) \triangleq \frac{n^s}{\NumOfB^{s +1 }} \sum_{k=1}^{\NumOfB} \frac1{(\# \cC_k)^s} \asymp 1 .
	\label{eq:zeta function}
	\end{equation}
	
	The order of the mean is complicated to analyze  since it depends on $b_{\cP}, b_{\cQ},b_{\cP}- b_{\cQ}$.
\end{thm}

For the sake of simplicity, we assume $w_{ii} \sim \cP$ due to same argument as the one below \eqref{eq:expectation dominant term}. Same as Theorem 2.1.31 in \cite{anderson2010introduction},  the first step is to evaluate mean and variance of 
$$W_{n} \triangleq \left\| \left(\bbW_n - \bbE_n \right) \bbZ_n \left( \bbZ_n^T \bbZ_n\right)^{- \frac12}\left[ \left( \bbZ_n^T \bbZ_n \right)^{\frac12} \BExpect \left( \bbZ_n^T \bbZ_n \right)^{\frac12} \right]^{-1} \right\|_F^2 .$$ For the sake of convenience, denote $\bbX_n \triangleq \bbW_n - \bbE_n$ with each entry zero mean.  By noticing
\begin{eqnarray*}
	&& \left[ \left( \bbZ_n^T \bbZ_n \right)^{\frac12} \BExpect \left( \bbZ_n^T \bbZ_n \right)^{\frac12} \right]^{-1}   \\
	& = &   \left( \bbZ_n^T \bbZ_n \right)^{- \frac12} \BExpect^{-1} \left( \bbZ_n^T \bbZ_n \right)^{  -  \frac12}  \\
	& =& \frac1{b_{\cP} - b_{\cQ}} \begin{bmatrix}  \#\cC_1 \\ &\ddots \\&& \#\cC_{\NumOfB} \end{bmatrix}^{ - \frac12 } \left[\bbI_{\NumOfB} - \frac{ b_{\cQ} \1_{\NumOfB} \1_{\NumOfB}^T}{b_{\cQ} \NumOfB + b_{\cP}- b_{\cQ}} \right] \begin{bmatrix}  \#\cC_1 \\ &\ddots \\&& \#\cC_{\NumOfB} \end{bmatrix}^{ - \frac12 }  \\
	& = &   \frac1{b_{\cP} - b_{\cQ}} \cdot \\
	&& \left\{ \begin{bmatrix}  \#\cC_1 \\ &\ddots \\&& \#\cC_{\NumOfB} \end{bmatrix}^{ - 1 } - \frac{b_{\cQ}}{b_{\cQ}\NumOfB + b_{\cP} - b_{\cQ}} \begin{bmatrix}  (\#\cC_1)^{ - \frac12 } \\ \ldots \\ (\#\cC_{\NumOfB})^{ - \frac12 }   \end{bmatrix}  \begin{bmatrix}  (\#\cC_1)^{ - \frac12 } \\ \ldots \\ (\#\cC_{\NumOfB})^{ - \frac12 }   \end{bmatrix}^T \right\},
\end{eqnarray*}
where for $\bbB^{-1}$ we utilize Assumption \ref{assum:homogeneous weighted SBM}.  Combining with \eqref{eq:ikth entry W-P Z inv(sqrt(ZtZ))}, we get the $ik$-th entry with $i \in [n],k\in \left[\NumOfB\right]$:
\begin{eqnarray}
& & \frac1{(b_{\cP}- b_{\cQ})  }\left\{  \frac1{\sqrt{\# \cC_{k}} \# \cC_{k}  }  \sum_{t \in \cC_k} x_{it} \right. \label{eq:ikth entry}  \\
&& \left. - \frac{b_{\cQ}}{b_{\cQ}\NumOfB + b_{\cP}-b_{\cQ}}  \frac1{\sqrt{\# \cC_{k}} }\sum_{j =1}^{\NumOfB} \frac{ \sum_{t \in \cC_j} x_{it}}{\# \cC_{j }}  \right\}  \nonumber  \\
&  = & \frac1{(b_{\cP}-b_{\cQ}) \sqrt{\# \cC_{k} } }\left\{ \frac{b_{\cQ}\NumOfB + b_{\cP}-2b_{\cQ}}{b_{\cQ}\NumOfB + b_{\cP}-b_{\cQ}} \frac1{ \# \cC_{k}  }  \sum_{t \in \cC_k} x_{it} \right. \nonumber\\
&& \left. - \frac{b_{\cQ}}{b_{\cQ}\NumOfB + b_{\cP}-b_{\cQ}} \sum_{\substack{ j =1 \\ j \ne k} }^{\NumOfB} \frac{ \sum_{t \in \cC_j} x_{it}}{\# \cC_{j }}  \right\},
\nonumber
\end{eqnarray}

Hence as for mean, by assumption \ref{assum:homogeneous weighted SBM}
{\allowdisplaybreaks\begin{eqnarray*}
	&& (b_{P}-b_{Q})^2 (b_Q\NumOfB +b_P -b_Q)^2 \E W_n \\
	& = & \sum_{i =1}^n \sum_{k =1 }^{\NumOfB} \frac1{\# \cC_{k} } \E  \left[ (b_Q\NumOfB +b_P -2b_Q)  \frac{\sum_{t \in \cC_k} x_{it}}{\#\cC_k} - b_Q\sum_{\substack{ j =1 \\ j \ne k}}^{\NumOfB} \frac{\sum_{t \in \cC_{j}} x_{it}}{\#\cC_{j}} \right]^2\\
	& = &  \sum_{i =1}^{\NumOfB} \sum_{s \in \cC_{i }}  \sum_{k =1 }^{\NumOfB} \frac1{\# \cC_{k} } \left[(b_Q\NumOfB +b_P -2b_Q)^2 \frac{\sum_{t \in \cC_k}\E  \left[x_{ s t}^2\right]}{(\#\cC_k)^2} + b_Q^2 \sum_{\substack{ j =1 \\ j \ne k}}^{\NumOfB} \frac{\sum_{t \in \cC_{j}} \E \left[x_{s t}^2\right]}{(\#\cC_{j})^2}  \right] \\
	& = &(b_Q\NumOfB +b_P -2b_Q)^2  \sum_{i =1}^{\NumOfB} \sum_{s \in \cC_{i }}   \sum_{k =1 }^{\NumOfB} \frac1{(\# \cC_{k})^3 } \sum_{t \in \cC_k} \E  \left[x_{ s t}^2\right]  \\
	& & + b_Q^2 \sum_{i =1}^{\NumOfB} \sum_{s \in \cC_{i }}  \sum_{k =1 }^{\NumOfB} \frac1{\# \cC_{k} } \sum_{\substack{ j =1 \\ j \ne k}}^{\NumOfB}\frac{ \sum_{t \in \cC_{j}}   \E \left[x_{s t}^2\right]    }{(\#\cC_{j})^2}  \\
	& = &(b_Q\NumOfB +b_P -2b_Q)^2   \sum_{k =1 }^{\NumOfB} \frac1{(\# \cC_{k})^3 } \sum_{t \in \cC_k}  \sum_{i =1}^{\NumOfB} \sum_{s \in \cC_{i }} \bbE  \left[x_{ s t}^2\right]  \\
	& & + b_Q^2 \sum_{j =1 }^{\NumOfB}\frac{1 }{(\#\cC_{j})^2}  \sum_{\substack{ k =1 \\ k \ne j}}^{\NumOfB} \frac1{\# \cC_{k} } \sum_{t \in \cC_{j}}   \sum_{i =1}^{\NumOfB} \sum_{s \in \cC_{i }}  \bbE \left[x_{s t}^2\right]    \\
	& = &   (b_Q\NumOfB +b_P -2b_Q)^2\left[ n \sigma_Q^2\cdot  \sum_{ k =1 }^{\NumOfB} \frac{1 }{(\# \cC_{k})^2 }  +     \left( \sigma_P^2 - \sigma_Q^2 \right) \sum_{ k =1 }^{\NumOfB} \frac{ 1}{\# \cC_{k} }\right]  \\
	&&  + b_Q^2 \left\{ n \sigma_Q^2 \left[ \left( \sum_{k =1 }^{\NumOfB} \frac{1}{\# \cC_{k}  }   \right)^2   -  \sum_{k =1 }^{\NumOfB} \frac{1}{(\# \cC_{k})^2  }   \right]
	+  \left( \sigma_P^2 - \sigma_Q^2 \right) (\NumOfB - 1) \sum_{k =1 }^{\NumOfB}\frac{ 1 }{\# \cC_{k} } \right\}   \\
	& = &n \sigma_Q^2\left\{ \left[(b_Q \NumOfB +b_P -2b_Q)^2 - b_Q^2 \right] \cdot  \sum_{ k =1 }^{\NumOfB} \frac{1 }{(\# \cC_{k})^2 } + b_Q^2   \left( \sum_{k =1 }^{\NumOfB} \frac{1}{\# \cC_{k}  }   \right)^2 \right\}\\
	&&   +  \left[(b_Q\NumOfB + b_P -2b_Q)^2 + \left( \NumOfB - 1\right) b_Q^2 \right]  \left( \sigma_P^2 - \sigma_Q^2 \right) \sum_{k =1 }^{\NumOfB}\frac{1}{\# \cC_{k} },
\end{eqnarray*}}
or by \eqref{eq:zeta function} equivalently, 
\begin{eqnarray*}
	&& \bbE W_n \\ 
	& = &  \frac{\NumOfB^3}n \cdot  \frac{ (b_Q\NumOfB +b_P -2b_Q)^2 - b_Q^2   }{ (b_P-b_Q)^2 (b_Q\NumOfB +b_P -b_Q)^2}  \zeta (2)  \sigma_Q^2+ \frac{\NumOfB^2 }n \cdot \frac{  \NumOfB^2 b_Q^2 \left[ \zeta(1)  \right]^2\sigma_Q^2  }{ (b_P-b_Q)^2 (b_Q\NumOfB +b_P - b_Q)^2}   \\
	&& +   \frac{ \NumOfB^2 }n \cdot\frac{ (b_Q\NumOfB +b_P -2b_Q)^2 + \left( \NumOfB - 1\right) b_Q^2  }{ (b_P-b_Q)^2 (b_Q\NumOfB +b_P -b_Q)^2} \zeta (1)\left( \sigma_P^2 - \sigma_Q^2 \right)= \Theta \left(\frac{\NumOfB^3 }n\right),
\end{eqnarray*}

\section{Proof of theorem for asymptotic power}
This section proves Theorem \ref{thm:asymptotic power guarantee}. Recall that we only consider the  hypothesis test \eqref{eq:hypothesis test -- Hamming distance} with $\epsilon' >0$ and we are in this simple scenario with equal-size assumption,. Hence,
  \begin{eqnarray*}
	T_{n,\NumOfB}  	&=& \frac1{n\NumOfB}\left\|  \left[\bbV_{\bbW_n^{(1)}}  \Procrustes \left( \bbV_{\bbW_n^{(1)}}  , \bbV_{\bbW_n^{(2)}} \right)   - \bbV_{\bbW_n^{(2)}} \right] \Lambda_{\bbW_n^{(2)}}\right\|_F^2  \\
	& \asymp_P & \frac1{n\NumOfB}\left\| \left[\bbZ_n^{(1)} \left( \left[ \bbZ_n^{(1)}\right]^T  \bbZ_n^{(1)}\right)^{-\frac12 }\Procrustes_{\bbZ}	-  \bbZ_n^{(2)} \left( \left[ \bbZ_n^{(2)}\right]^T  \bbZ_n^{(2)}\right)^{-\frac12 } \right] \Lambda_{\bbW_n^{(2)}} \right\|_F^2 \\
	& = & \frac1{n\NumOfB}\left(\sqrt{\frac{\NumOfB}{n}} n\ell_n \cdot \frac{nb_{\cP}}{\NumOfB}\right)^2  = \frac{n^2 b_{\cP}^2 \ell_n^2}{\NumOfB^2} \succeq  n^{2 \epsilon'}\mu_{n}\succ  \mu_n,
\end{eqnarray*}
where $$\Procrustes_{\bbZ} =  \Procrustes \left(\bbZ_n^{(1)} \left( \left[ \bbZ_n^{(1)}\right]^T  \bbZ_n^{(1)}\right)^{-\frac12 } , \bbZ_n^{(2)} \left( \left[ \bbZ_n^{(2)}\right]^T  \bbZ_n^{(2)}\right)^{-\frac12 } \right).$$ Consequentially,
	$$
	\frac{		T_{n,\NumOfB} - \mu_n }{ \sqrt{\Var_n }}\succeq_P   n^{\frac12 + 2\epsilon'}\cdot \frac{\mu_n \sqrt{\NumOfB}}{b_{\cP}} \succeq n^{\frac12 + 2\epsilon'}\sqrt{\NumOfB}.
	$$
	
Therefore, for any two-sided $\alpha$-level test with $q_{\frac{\alpha}{2}}$,$ q_{\left(1-\frac{\alpha}{2}\right)}$ the $\frac{\alpha}{2}$-quantile and $\left(1-\frac{\alpha}{2}\right)$-quantile of Gaussian distribution, the probability under the alternative $H_1'$ of \eqref{eq:hypothesis test -- Hamming distance} satisfies
	
	$$
	\P_{H'_1} \left(q_{\frac{\alpha}{2}} < \frac{		T_{n,\NumOfB} - \mu_n }{ \sqrt{\Var_n }} < q_{\left(1-\frac{\alpha}{2}\right)}\right) \to 1.
	$$

\section{Proofs of auxiliary lemmas}

\begin{defn}[Induced norms] An operator (or induced) matrix norm is a norm $\|\cdot \|_{a,b} : \Reals^{m \times n} \to \Reals$ defined as
	$\|\bbA\|_{a,b}=\max_{\|\x\|_b \le 1}\|\bbA\x\|_a
	$,
	where $\|.\|_a$ is a vector norm on $\Reals^m$ and $\| \cdot \|_b$ is a vector norm on $\Reals^n$.
	\label{defn: induced norms}
\end{defn}

\begin{lem}[Matrix norm inequalities] \label{lem:matrix norm inequality}
	For $\bbA \in \Reals^{m \times n}, \bbB \in \Reals^{n \times p}$, we have 
	\begin{equation}
	\| \bbA \bbB \|_F \le \| \bbA \|_2 \cdot  	\|  \bbB \|_F . \label{eq:matrix norm inequality}
	\end{equation}
	Every induced norm in \eqref{defn: induced norms} is submultiplicative, i.e.
	\begin{equation}
	\| \bbA \bbB \|_{a,b} \le \| \bbA \|_{a,b} \cdot  	\|  \bbB \|_{a,b} . \label{eq:induced norm submultiplicative}
	\end{equation}
\end{lem}
\begin{proof} 
	\begin{enumerate} 
		\item By letting $\bbB = \left( \b_1,\ldots, \b_p\right)$, we have
		$$\left\| \bbA \bbB  \right\|_F^2 = \sum_{i=1}^p \left\| \bbA \b_i \right\|_2^2 \le \left\|  \bbA \right\|_2^2 \sum_{i=1}^p \left\|\b_i \right\|_2^2 =  \left\|  \bbA \right\|_2^2  \cdot \| \bbB\|_F^2. $$
		\item Refers to \cite{ahmadi:website}.
	\end{enumerate}
\end{proof}

%

\section{Additional simulation results}
When $\NumOfB >2$, estimation of the community memberships becomes more difficult, which can lead to  slower convergence rates (see Table \ref{tab:type I error unweighted SBM Kn>2}), although  type I errors are approximately under control. 
\begin{table}[ht]
	\caption{Average type I error of two sided test with significance level $\alpha = 5\%$ for unweighted SBM with $p = 0.5, q=0.1$ based on 2000 replications. $\bbT = \Procrustes  \left( \bbV_{\bbW^{(1)}} , \bbV_{\bbW^{(2)}} \right) $. For $\NumOfB = 3$,  pick $\# \cC_1 = \frac{n}6, \# \cC_2 = \frac{n}3, \# \cC_3 = \frac{n}2$; for  $\NumOfB = 4$, pick $\# \cC_{ 1} = \frac{n}9, \# \cC_{ 2} = \frac{2n}9, \# \cC_{ 3} = \frac{n}3, \# \cC_{ 1} = \frac{n}3   $
		.}
	\label{tab:type I error unweighted SBM Kn>2}
	\begin{center}
		\begin{tabular}{ccccc} \hline 
			$n$ & $\NumOfB$& $\gamma = 1.5$& $\gamma= 1$  &  $\gamma = 0.7$
			\\ \hline
			\multirow{3}{*}{500} 
			& 2 &5.3\%&4.7\% &9.7\%
			\\ 	
			&3 &7.6\% &14.8\% &36.3\% 
			\\ 
			& 4& 18.75\%  & 46.9\% & 91.9\%  
			\\  
			\hline 
			\multirow{3}{*}{1000} 
			&2& 4.9\% & 4.8\% &7.3\% 
			\\ 
			& 3 &5.8\%  & 10.4\% & 20.2\% 
			\\ 
			& 4 &10.7\% & 26.8\%  &  59.9\% 
			\\  \hline
			\multirow{3}{*}{2000} 
			&2 &  4.9\% & 5.3\% & 5.6\%  
			\\ 
			&  3&5.0\%&7.8\%&15.4\%
			\\  
			&4 &  8.3\% & 15.7\%   & 34.8\%   
			\\ 
			\hline
			\multirow{3}{*}{4000} 
			&2 & 4.5\% & 5.2\%  & 6.1\% 
			\\ 
			&3&4.9\% & 5.7\% & 9.2\% 
			\\ 
			&4 &6.3\% &  11.0\%  & 26.2\% 
			\\ 	 \cline{2-5}
			\hline
			\multirow{3}{*}{8000} & 2 & 5.1\% & 4.9\% &5.3\% 
			\\  
			& 3 & 4.6\% & 5.4\% &6.4\% 
			\\  
			& 4&  5.3\% &  6.1\% & 11.7\%   
			\\  \hline
		\end{tabular}
	\end{center}
\end{table}